\documentclass[onefignum,onetabnum]{siamart190516}

\usepackage{lipsum}
\usepackage{amsfonts}
\usepackage{graphicx}
\usepackage{epstopdf}
\usepackage{algorithmic}
\ifpdf
  \DeclareGraphicsExtensions{.eps,.pdf,.png,.jpg}
\else
  \DeclareGraphicsExtensions{.eps}
\fi

\usepackage{graphicx}              
\usepackage{amsmath}               
\usepackage{amsfonts}              

\usepackage{amsmath}
\usepackage{amsfonts}
\usepackage{amssymb}
\usepackage{latexsym}
\usepackage{multirow}
\usepackage{pict2e}
\usepackage{picture}

\usepackage[normalem]{ulem}

\newcommand{\diff}{a}
\newcommand{\bform}{B}

\newcommand{\mesh}{{\mathcal T}}

\newcommand{\be}{\begin{equation}}
\newcommand{\ee}{\end{equation}}
\newcommand{\bea}{\begin{eqnarray}}
\newcommand{\eea}{\end{eqnarray}}

\newcommand{\nno}{\nonumber}

\newcommand{\mbf}[1]{\mbox{\boldmath$\rm{#1}$}}

\newcommand{\ddd}{\text{\rm D}}
\newcommand{\dn}{\text{\rm N}}
\newcommand{\el}{ \k \in \mathcal{T}}

\newcommand{\diam}{\operatorname{diam}}

\newcommand{\dint}{\text{\rm int}}
\newcommand{\fes}{S_{\mathcal{T}}}
\newcommand{\fesp}{S^{p}_{\mathcal{T}}}

\newcommand{\feso}{S^{{\bf 0}}_{\mathcal{T}}}

\newcommand{\trineq}{C_{{\rm tr}}}
\newcommand{\tinv}{C_{{\rm inv}}}
\newcommand{\tinvT}{\widehat{C}_{{\rm inv}}}
\newcommand{\einv}{C_{{\rm inv,}\k}}
\newcommand{\einvT}{C_{{\rm inv,}{T}}}

\newcommand{\Csh}{\tau}
\newcommand{\CshT}{\widehat{\Csh}}
\newcommand{\Csz}{C_{\rm sz}}
\newcommand{\Cin}{C_{\rm in}}
\newcommand{\Cco}{C_{\rm co}}
\newcommand{\Cnc}{C_{\rm nc}}

\newcommand{\tnabla}{\nabla_{\hspace{-.5mm} {\rm T}}}

\newcommand{\amesh}{\widehat{{\mathcal{T}}}}

\newcommand{\mean}[1]{\{\!\!\{#1\}\!\!\}}                
\newcommand{\jump}[1]{[\![#1]\!]}                        
\newcommand{\uu}[1]{\mathbf{#1}}
\renewcommand{\k}{K}

\newcommand{\su}{\sum_{\el}}
\newcommand{\ud}{\,\mathrm{d}}
\newcommand{\ndg}[1]{| \kern -.25mm \|{#1}| \kern -.25mm \|}
\newcommand{\ncdg}[1]{| \kern -.25mm \|{#1}| \kern -.25mm \|_{\rm DG}}
\newcommand{\nsdg}[1]{| \kern -.25mm \|{#1}| \kern -.25mm \|_{\rm s}}
\newcommand{\nstdg}[1]{| \kern -.25mm \|{#1}| \kern -.25mm \|_{L_2 (J; \mathcal{D} )}}
\newcommand{\ltwo}[2]{\|{#1}\|_{{#2}}}
\newcommand{\linf}[2]{\|{#1}\|_{L_{\infty}({#2})}}

\usepackage{dsfont}

\newcommand{\no}{[\kern -.8mm [}
\newcommand{\nc}{]\kern -.8mm ]}

\newcommand{\norm}[2]{\| {#1} \|_{#2}}

\newsiamremark{remark}{Remark}
\newsiamremark{hypothesis}{Hypothesis}
\crefname{hypothesis}{Hypothesis}{Hypotheses}
\newsiamthm{claim}{Claim}
\newtheorem{assumption}[theorem]{Assumption}

\headers{\emph{A posteriori} error estimates for dG}{A.~CANGIANI, Z.~DONG, E.H.~GEORGOULIS}

\title{\emph{A posteriori} error estimates \\ for  \\
	 discontinuous Galerkin methods\\
	  on polygonal and polyhedral  meshes
}

\author{ANDREA CANGIANI\thanks{SISSA, International School for Advanced Studies, via Bonomea 265, I-34136 Trieste, Italy (\email{andrea.cangiani@sissa.it}).}
		\and ZHAONAN DONG\thanks{1) Inria, 2 rue Simone Iff, 75589 Paris, France,
				2) CERMICS, Ecole des Ponts, 77455 Marne-la-Vall\'{e}e, France (\email{zhaonan.dong@inria.fr}) }
				\and Emmanuil H. Georgoulis\thanks{1) The Maxwell Institute for Mathematical Sciences and Department of Mathematics, School of Mathematical and Computer Sciences, Heriot-Watt University,   Edinburgh EH14 4AS, United Kingdom (\email{E.Georgoulis@hw.ac.uk}), 2) Department of Mathematics, School of Applied Mathematical and Physical Sciences, National Technical University of Athens, Zografou 15780, Greece \& 3) IACM-FORTH, Greece.}
}

\usepackage{amsopn}

\ifpdf
\hypersetup{
	pdftitle={\emph{A posteriori} error estimates for
		discontinuous Galerkin methods
for polytopic meshes},
	pdfauthor={}
}
\fi

\begin{document}

\maketitle

\begin{abstract}
	We present a new residual-type energy-norm \emph{a posteriori} error analysis for interior penalty discontinuous Galerkin (dG) methods for linear elliptic problems. The new error bounds are also applicable to dG methods on meshes consisting of elements with very general polygonal/polyhedral shapes.  The case of simplicial and/or box-type elements is included in the analysis as a special case. In particular, for the upper bounds, arbitrary number of very small faces are allowed on each polygonal/polyhedral element, as long as certain mild shape regularity assumptions are satisfied. As a corollary, the present analysis generalizes known \emph{a posteriori} error bounds for dG methods, allowing in particular for meshes with arbitrary number of irregular hanging nodes per element. The proof hinges on a new conforming recovery strategy in conjunction with a Helmholtz decomposition. The resulting \emph{a posteriori} error bound involves jumps on the tangential derivatives along elemental faces. Local lower bounds are also proven for a number of practical cases.
	Numerical experiments are also presented, highlighting the practical value of the derived \emph{a posteriori} error bounds as error estimators.
\end{abstract}



\section{Introduction}
Recent years have witnessed an extensive activity in the development of various Galerkin methods posed on meshes consisting of general polygonal/polyhedral (henceforth, collectively referred to as \emph{polytopic}) elements. A central question arising is the derivation of computable error bounds for such discretisations, so that the extreme geometric flexibility of such meshes can be harnessed.

Residual-type \emph{a posteriori} error bounds for interior penalty dG methods on composite/polytopic meshes appeared in \cite{giani_houston,cui}. Also, in the context of virtual element methods, corresponding bounds are proven in \cite{vem_apost,Carsten22}, while for the weak Galerkin approach, an \emph{a posteriori} error analysis can be found in \cite{weakgal}. In addition, corresponding results for the hybrid high-order method can be found in \cite{HHO_a_posterior}. All aforementioned results are proven for shape-regular polytopic meshes, under the additional assumption that the diameters of elemental faces are of comparable size to the element diameter. Although the latter may be a reasonable assumption in the context of standard, simplicial meshes, it can be rather restrictive for general polytopic elements. This is because general polygons/polyhedra with more than $d$ faces can be simultaneously shape-regular yet containing \emph{small} faces, i.e., faces whose diameter is arbitrarily small compared to the element diameter.

This work aims exactly at rectifying this restrictive state of affairs. We prove new energy-norm \emph{a posteriori} upper error bounds for interior penalty discontinuous Galerkin (dG) methods posed on meshes containing polytopic elements, including with arbitrary number of small faces, as long as certain mild shape-regularity assumptions are satisfied.
The case of simplicial and/or box-type elements is included in the analysis as a special case. For accessibility, we restrict the discussion to a model elliptic problem, noting, nevertheless, that various generalizations are possible with minor modifications.

As a general principle, residual-based \emph{a posteriori} error analysis of non-conforming and, in particular, dG methods requires a recovery of the numerical solution into a related conforming function. The pioneering work of Karakashian \& Pascal \cite{KP} (see also \cite{KP_conv}) proposed the recovery of the dG solution by a nodal averaging operator for which a crucial stability result was proven \cite[Theorem 2.2]{KP}; cf.,~also \cite[Theorem 2.1]{KP_conv} for an extension. This construction allowed for the first rigorous \emph{a posteriori} error analysis of a dG method for elliptic problems.  A number of related results followed, improving various aspects of the theory; for instance,  see~\cite{ains1,hp-DG,max, MR2502931, ains2,bonito_nochetto,MR2475951,De_Ge, Kr_Ge}. A key reason for the aforementioned restrictive assumption that all elemental faces are of comparable size to the element diameter in existing \emph{a posteriori} error analysis for polytopic dG methods \cite{giani_houston,cui} is exactly the lack fo availability of a  stability result corresponding to \cite[Theorem 2.2]{KP} for polytopic element meshes containing elements with small faces.

In this work, we crucially avoid the use of averaging operators. Instead, the proof of the upper error bound hinges on a new recovery into $H^1$-conforming functions, in conjunction with a Helmholtz decomposition. To complete the analysis, we also require the existence of appropriate auxiliary simplicial meshes on which quasi-interpolants are defined. This can be verified in practice using simple and efficient algorithms. We provide two such algorithms, one  based on a sub-mesh and one employing tools from computational geometry and, in particular, constrained Delaunay triangulations~\cite{Chew89,Shew98}.
The resulting \emph{a posteriori} error bound involves also jumps on the tangential derivatives along elemental faces.

Local lower bounds are also proven for a number of practical cases, indicating the optimality of the new estimators. The key challenge in the proof of the latter is, again, the treatment of small/degenerating element faces and the construction of respective bubble functions. In particular, local lower bounds for the element residuals are proven allowing for arbitrarily small faces. Lower bounds for the flux residuals are proven under more restrictive assumptions (see Assumption \ref{A3} below), allowing, nevertheless, for arbitrarily small faces.

We note that the case of meshes consisting of simplicial and/or box-type elements is included in the analysis as a special case. For such `classical' mesh concepts, the developments presented below provide a new class of \emph{a posteriori} error bounds applicable even to $hp$-version dG methods, allowing in particular for meshes with arbitrary number of irregular hanging nodes per element.

The remainder of this work is structured as follows. In Section \ref{mod}, we define the elliptic model problem, the admissible meshes, and finite element spaces and the interior penalty dG method on polytopic meshes. We also prove some important technical results regarding the construction of auxiliary meshes which will be instrumental in the proof of a posteriori error bounds. In Section \ref{apost_sec}, we prove the \emph{a posteriori} upper error bound, using the aforementioned technical developments, while in Section \ref{lower_bounds} we provide respective lower bounds for the energy-norm error for a number of practical cases. Finally, in Section \ref{num_ex}, we present some numerical experiments confirming the robustness and efficiency of the of the derived \emph{a posteriori} error bound and highlighting its practical value  as an error estimator.

\section{Model problem and numerical method}\label{mod}

For a Lipschitz domain $\omega \subset {\mathbb R}^d$, $d=1,2,3$,  we denote by $H^s(\omega)$ the Hilbertian Sobolev space of  index $s\ge 0$ of real--valued functions defined on
$\omega$, endowed with the seminorm $|\cdot |_{H^s(\omega)}$ and norm $\|\cdot\|_{H^s(\omega)}$.   Furthermore, we let~$L_p(\omega)$, $p\in[1,\infty]$,  be the standard Lebesgue space on $\omega$, equipped with the norm~$\|\cdot\|_{L_p(\omega)}$. In the case $p=2$, we shall simply write $\|\cdot\|_{\omega}$  to denote the $L_2$-norm over $\omega$ and simplify this further to $\|\cdot\|$ when $\omega=\Omega$, the physical domain.   Finally,  $|\omega|$ denotes the $d$-dimensional Hausdorff measure of $\omega$.

\subsection{Model problem}\label{Parabolic problem}
Let $\Omega$ be a bounded, simply connected, and open polygonal/polyhedral domain in $\mathbb{R}^d$, $d=2,3$. The boundary $\partial\Omega$ of $\Omega$ is split into two disjoint parts $\Gamma_{\ddd}$ and $\Gamma_{\dn}$ with $|\Gamma_{\ddd}|\neq 0$.
For technical reasons, when $d=3$ and $|\Gamma_{\dn}|\neq 0$,  the interface between $\Gamma_{\ddd}$ and $\Gamma_{\dn}$ is  assumed that is made up of straight planar segments.
 We consider the linear elliptic problem: find $u \in H^1(\Omega)$, such that
\begin{equation}\label{Problem}
\begin{aligned}
- \nabla \cdot (\diff \nabla u) =&\ f   &\text{in }   \Omega , \\
\quad u =&\ g_{\ddd}   &\text{on } \Gamma_{\ddd}, \\
\quad a\nabla u\cdot \mbf{n} =&\ g_{\dn}   &\text{on }  \Gamma_{\dn},
\end{aligned}
\end{equation}
with known $f\in L_2(\Omega)$, $g_{\ddd}\in H^{1/2}(\Gamma_{\ddd})$ and $g_{\dn}\in L_2(\Gamma_{\dn})$ and symmetric diffusion tensor $\diff\in [L_\infty (\Omega)]^{d\times d}$ such that
\begin{equation}\label{uniform ellipticity}
\alpha^* |\xi|^2\geq  \xi^\top \diff(x) \xi \geq \alpha_* |\xi|^2 >0 \quad  \forall \xi \in  \mathbb{R}^d, \quad \text{a.e.} \quad  x\in   \Omega,
\end{equation}
for some constants $\alpha^*,\alpha_*>0$.  For simplicity of the presentation, we assume that $a$ is piecewise constant, although this is not an essential restriction for the validity of the developments below.

Setting $H^1_\ddd:= \{v\in H^1(\Omega): v = 0  \text{ on } \Gamma_{\ddd}  \}$, 
the weak formulation of \eqref{Problem} is: find $u\in H^1(\Omega)$, $ u = g_{\ddd}$ on $\Gamma_{\ddd}$ such that
\begin{equation}\label{eq:prob}
\int_\Omega \diff{\nabla u}\cdot {\nabla v}\ud \uu{x}  =  \int_\Omega f  v \ud \uu{x} +\int_
{\Gamma_{\dn}}  g_{\dn} v \ud s,
\end{equation}
for all $v\in H^1_{\ddd}(\Omega)$.  The well-posedness is guaranteed by the Lax-Milgram Lemma.

\subsection{Finite element spaces and trace operators}
We consider meshes $\mesh$ consisting of general polygonal (for $d=2$) or polyhedral (for $d=3$) mutually disjoint open elements $\k\in\mesh$, henceforth termed collectively as \emph{polytopic}, with $\cup_{\el}\bar{\k} =\bar{\Omega}$. Given $h_{\k}:=\diam(\k)$, the diameter of  $\el$, we define the mesh-function $\mbf{h}:\cup_{\el}\k\to\mathbb{R}_+$ by $\mbf{h}|_{\k}=h_{\k}$, $\el$.
Further, we let $\Gamma:=\cup_{\el}\partial\k$ denote the mesh skeleton and set $\Gamma_{\dint}:=\Gamma\backslash\partial\Omega$. The mesh skeleton $\Gamma$ is decomposed into $(d-1)$--dimensional simplices $F$ denoting the mesh \emph{faces}, shared by at most two elements. These are distinct from elemental \emph{interfaces}, which are defined as the simply-connected components of the intersection between the boundary of an element and either a neighbouring element or $\partial\Omega$.
As such,  an interface between two elements may consist of more than one face, separated by hanging nodes/edges shared by those two elements only. This includes both `classical' hanging nodes, typically created by local mesh refinement, and non-standard ones separating non-co-planar faces. The latter may be created, for instance,  by a mesh agglomeration procedure; we refer to Fig.~\ref{fig_hanging_nodes} for an illustration for $d=2$.
\begin{figure}[t]
	\centering
	\includegraphics[scale=.25]{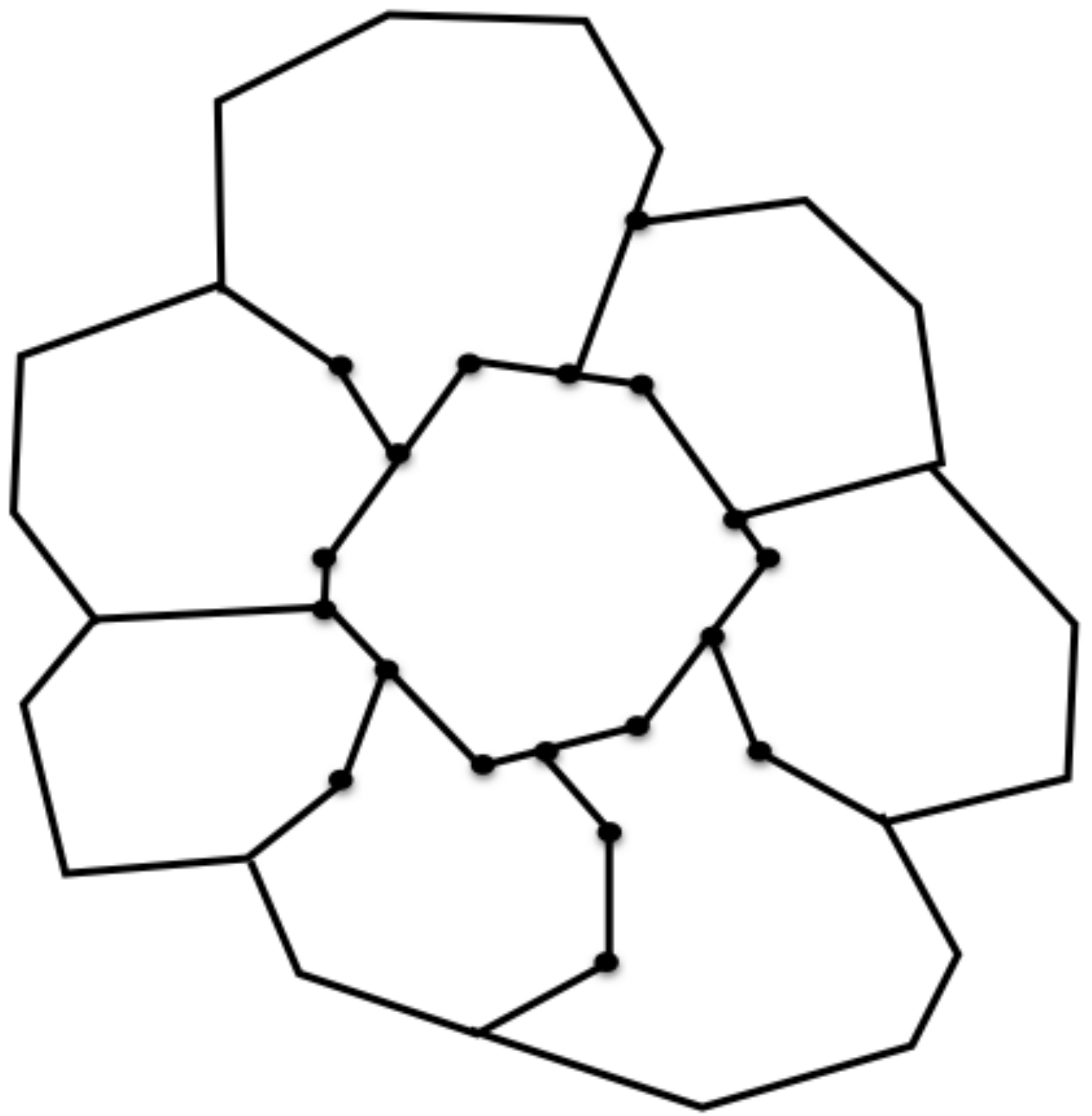}
	\caption{Element $\el$ with its face-wise neighbours;
		hanging nodes are highlighted by a bullet.}
	\begin{picture}(0,0)
	\put(-5,90){$\k$}
	\end{picture}
	\label{fig_hanging_nodes}
\end{figure}
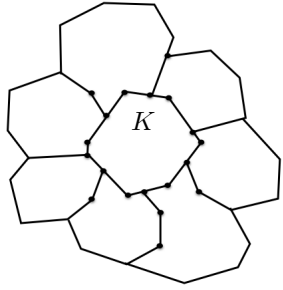

The \emph{finite element space} $\fes$ with
respect to $\mesh$ is defined by
\[
\fes\equiv \fesp :=\{u\in L_2(\Omega)
:u|_{\k}\in\mathcal{P}_{p}(\k),\el\},
\]
for some $p\in\mathbb{N}$
with
$\mathcal{P}_{p}(\k)$ denoting the space of $d$-variate polynomials of total degree up to $p$ on $\k$. We stress that the local elemental polynomial spaces employed within $\fes$ are defined in the \emph{physical coordinate system}, i.e., without mapping from a given reference or canonical frame. This approach allows to retain the full local approximation properties of the underlying finite element space. We refer to \cite{DGpoly1,DGpolybook} for a detailed discussion on the benefits and implementation issues resulting from this choice.

Let $\k_i$ and $\k_j$ be two
adjacent elements of $\mathcal{T}$ sharing a face $F\subset\partial \k_i \cap \partial \k_j \subset \Gamma_{\dint}$.
For $v$ and $\mbf{q}$ element-wise continuous scalar- and vector-valued functions, respectively,  we define the \emph{average} across $F$ by $
\mean{v}|_F:=\frac{1}{2}(v|_{F\cap \k_i}+v|_{F\cap \k_j})$, $\mean{\mbf{q}}|_F:=\frac{1}{2}(\mbf{q}|_{F\cap \k_i}+\mbf{q}|_{F\cap \k_j})$,
respectively, and the \emph{jump} across $F$ by
$
\jump{v} :=v|_{F\cap \k_i} -v|_{F\cap \k_j}$, $\jump{\mbf{q}} := \mbf{q}|_{F\cap \k_i} -\mbf{q}|_{F\cap \k_j}$,
using the convention $i>j$ in the element numbering to determine the sign.
On a boundary face $F\subset  \Gamma_{\ddd}$, with $F \subset \partial \k_i$, $\k_i\in \mathcal{T}$,
we set
$
\mean{v}:=v_i, $  $ \mean{\mbf{q}}:=\mbf{q}_i,  $
$\jump{v} :=v_i ,$ and $ \jump{\mbf{q}} := \mbf{q}_i$, respectively.

For $v\in\fes$  we denote by $\nabla_h v$ the element-wise gradient; namely, $(\nabla_h v)|_{K}:=\nabla (v|_K)$ for all $\el$.
Also, we denote by $\tnabla v$ the face-wise tangential gradient operator acting on the traces of $v$ on $\Gamma$, noting that $\tnabla v$ is double-valued on $\Gamma_{\dint}$. With a slight abuse of notation, we use the same symbol to denote the tangential gradient of boundary functions such as  the Dirichlet datum $g_{\ddd}^{}$.

\subsection{Mesh assumptions, inverse inequalities, and approximation results}\label{meshes}

Each mesh $\mesh$ is required to conform to the problem data in the following basic way.
First, $\mesh$  must represent exactly the domain, namely $\cup_{\el}\k=\Omega$, and be consistent with the subdivision of $\partial\Omega$ into $\Gamma_{\ddd}$ and $\Gamma_{\dn}$.
Moreover, we require resolution of multiscale features of the domain, such as complex boundaries and bottlenecks. Note that, in the context of polytopic meshes, such resolution is not intrinsic in that multiscale geometrical features can be represented by relatively `large' elements with `small' faces.
Hence we assume that
the local mesh size of each mesh $\mesh$ is comparable to the local finest scale of $\Omega$.
It is clear that such saturation-type assumption can always be satisfied, possibly  after a finite number of refinements of an original coarse mesh.
Further, we require the following general polytopic mesh regularity assumption.
\begin{assumption}[{Mesh regularity}]\label{A1}
We assume that each mesh $\mesh$ satisfies the following mesh regularity conditions. For each $\el$ it holds:
	\begin{itemize}
	\item[(a)]
	$\k$ is star-shaped with respect to an inscribed ball of radius $  r_\k\ge \Csh^{-1}h_\k$, centred at some point $\mathbf{x}^0\equiv\mathbf{x}^0_\k\in\k$;
	see Figure \ref{polygons} (right panel)  for an illustration when $d=2$.
	\item[(b)]
	Each face $F\subset\partial\k\cap\partial\Omega$
	is star-shaped  with respect to a $(d-1)$-dimensional  ball of radius $r_F\ge\Csh^{-1}h_\k$.
\end{itemize}
Here, $\Csh> 1$ is a constant independent of the discretization parameters. In what follows, we assume that the centres of the inscribed balls are selected to be chosen so that $\Csh$ is as minimal.
\end{assumption}

\begin{remark}
	All results  below generalize immediately to meshes containing polytopic elements that are finite unions of star-shaped polytopes; see Figure~\ref{polygons} (left panel) for an example. The minimal modifications required in the proofs are detailed in Remark~\ref{remark:union} below.

	Assumption~\ref{A1} allows for very general element shapes, including non-convex polytopes with arbitrary number of degenerating faces, i.e., element faces $F\subset\partial K$ with $|F|<<h_K^{d-1}$.
	Notable examples of acceptable subdivisions comprise elements with a bounded number of possibly degenerate hanging nodes and elements with `many' faces obtained by agglomeration  of very fine triangulations of the problem domain.
	Note, however, that Assumption~\ref{A1}(b) forbids degenerating/shape-irregular elemental faces on the boundary.
	This restriction can be relaxed to some extent at the expense of introducing unknown/hard-to-estimate constants, as discussed in Remark~\ref{remark:dirichlet} below.

	 To the best of our knowledge, Assumption \ref{A1} allows for the most general  polytopic meshes for which a posteriori error bounds are proven, for any Galerkin discretization and for any PDE problem. Nevertheless, it is, perhaps inevitably, more restrictive compared to the respective ones required for stability and a priori error analysis of dG methods; see ~\cite{DGpolyparabolic}, \cite[Section 4]{DGpolybook} and~\cite{DGease}, for details. The key advantage of the present setting is that it allows us to be as explicit as possible in the constants involved in the a posteriori error bounds.
	\end{remark}

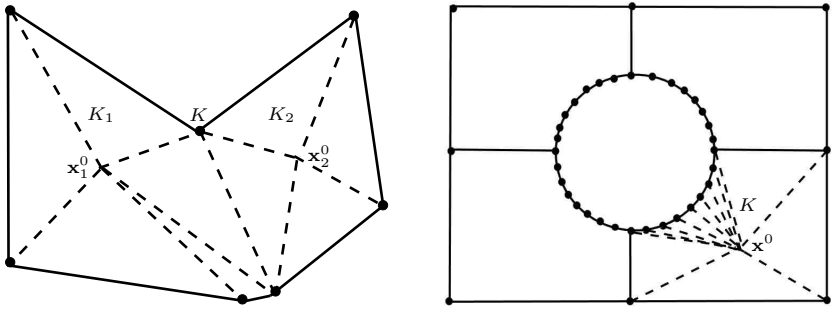
\begin{figure}[!t]
	\begin{center}
		\begin{tabular}{cc}
			\includegraphics[scale=0.5]{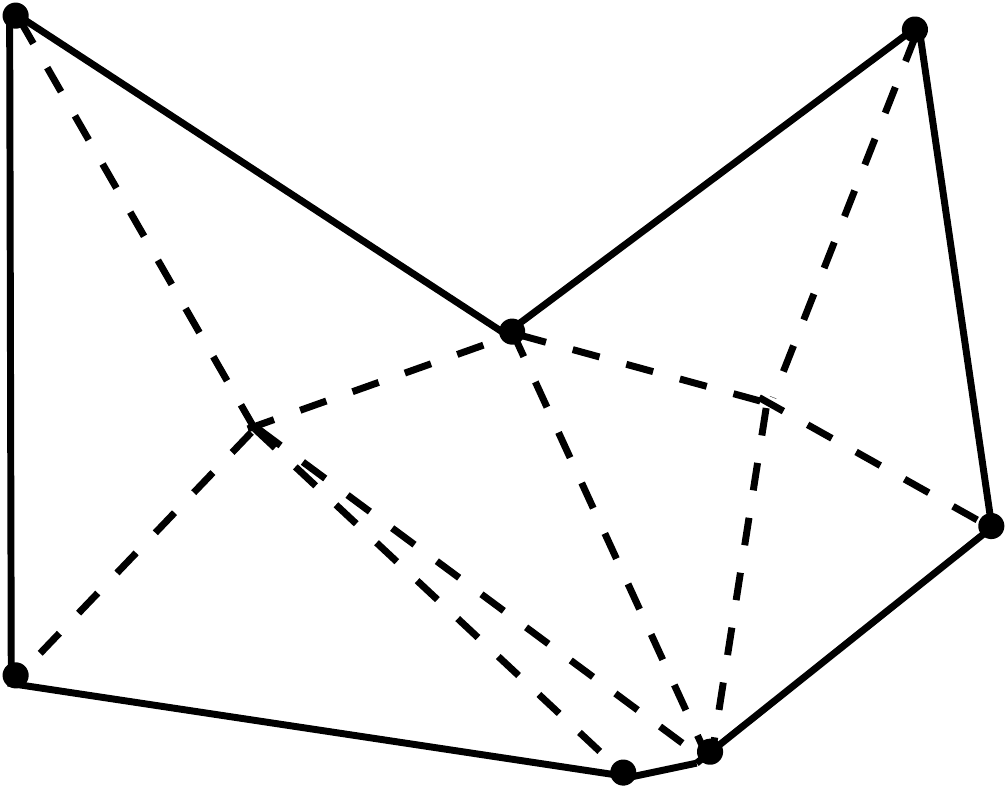} 	\hspace{0.5cm}
			\includegraphics[scale=0.2]{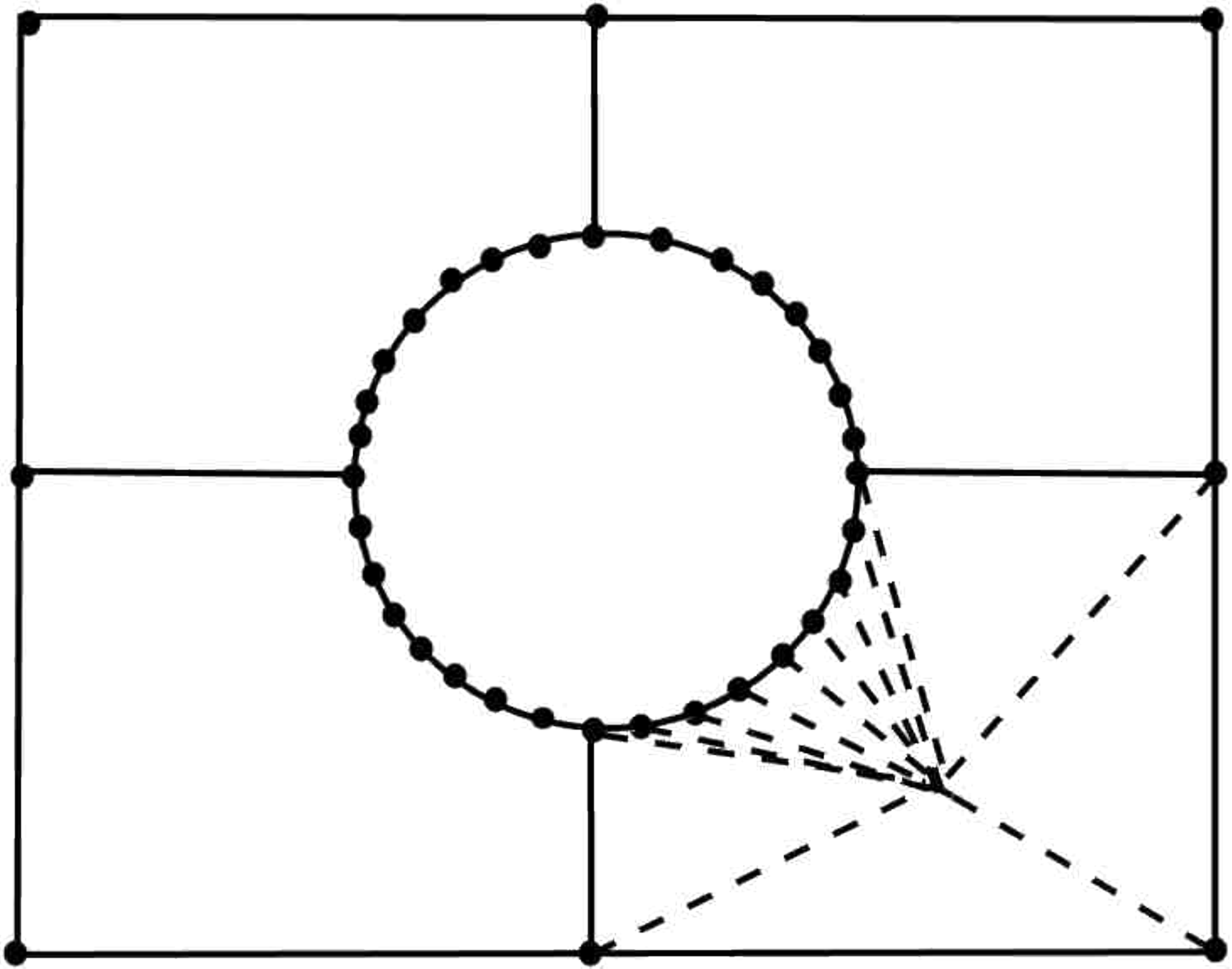}
		\end{tabular}
	\end{center}
	\setlength{\unitlength}{1cm}
	\begin{picture}(0,0)
	\put(3.6,2.7){{\scriptsize{\makebox(0,0){{$\k$}}}}}
	\put(2,2){{\scriptsize{\makebox(0,0){{${\bf x}_1^0$}}}}}
	\put(2.3,2.7){{\scriptsize{\makebox(0,0){{$\k_1$}}}}}
	\put(5.2,2.15){{\scriptsize{\makebox(0,0){{${\bf x}_2^0$}}}}}
	\put(4.7,2.7){{\scriptsize{\makebox(0,0){{$\k_2$}}}}}
	\put(10.9,1.5){{\scriptsize{\makebox(0,0){{$\k$}}}}}
	\put(11.1,1){{\scriptsize{\makebox(0,0){{${\bf x}^0$}}}}}
	\end{picture}
	\caption{Left panel: A polygonal element $\k$  with $7$ nodes is subdivided into two polygons $\k_1$ and $\k_2$ star-shaped with respect to the points ${\bf x}_1^0$ and ${\bf x}_2^0$, respectively. Right panel: a patch of polygonal elements. The element $\k$ has $12$ nodes and is star-shaped with respect to a ball centred in  ${\bf x}^0$.}
	\label{polygons}
\end{figure}

We also require the following, standard, local quasi-uniformity assumption.
\begin{assumption}[local quasi-uniformity]\label{A2} Each mesh $\mesh$ is locally quasiuniform, i.e., there exists
	a constant $\rho>0$ such that
	$
	\rho^{-1}\leq h_{\k_i} / h_{\k_j}\leq \rho,
	$
	whenever $\k_i,\k_j\in\mesh$ share  a common face.
\end{assumption}

We now state and prove a few results stemming from the above shape-regularity and local quasiuniformity assumptions.
A first, geometrical, consequence is that the number of  interface neighbours of each $\el$ is, in fact, bounded.
\begin{lemma}\label{lem:neigh}
Under Assumptions~\ref{A1} and~\ref{A2}, the number of interface neighbours of each element of $\mesh$ is uniformly bounded.
\end{lemma}
\begin{proof}
	 Denote by ${\omega}_\k$ the set of elements in $\mesh$ which are neighbours of $\k$. We derive a rough upper-bound on the cardinality $n$ of ${\omega}_\k$ as follows. If $\k'\in{\omega}_\k$, then $\k'\subset B_{h_{\k}+h_{\k'}}(\mathbf{x}^0_\k
	 )\subseteq B_{(1+\rho)h_\k}(\mathbf{x}^0_\k)$ thanks to Assumption~\ref{A2}. On the other hand, Assumption~\ref{A1} implies that $\k'$ contains the ball $B_{r_{\k'}}(\mathbf{x}^0_{\k'})$ with $r_{\k'}\ge\Csh^{-1}h_{\k'}$. Letting $C_2:=\pi$ and $C_3:=4\pi/3$, we thus have $|K'|\ge C_dr_{\k'}^d\ge C_d \Csh^{-d}h_{\k'}^d\ge C_d \Csh^{-d}\rho^{-d}h_{\k}^d$. Therefore, $n(C_d\Csh^{-d}\rho^{-d}h_{\k}^d)\le C_d (1+\rho)^dh_\k^d$, or $n\le \Csh^{d}\rho^{d}(1+\rho)^d$, thereby showing that the number of neighbours of $\k$ is uniformly bounded as required.
\end{proof}

In particular, the presence of `many, small' faces is allowed if these are grouped into a few interfaces only; see Figure \ref{polygons} (right panel) for an example. Complex interfaces may be produced by agglomeration procedures used to perform numerical upscaling of complex domains described through very fine triangulations~\cite{Bassi2012,DGpolybook,Corti22} or within adaptive algorithms  to align the mesh to solution features and coefficients anisotropies~\cite{Elias97,DGpolybook,Sutton21}

\begin{lemma}\label{lem:trace}
	Let $\k\in\mesh$ satisfying Assumption \ref{A1}. Then, for all $v\in H^1(\Omega)$, we have the trace estimate
	\begin{equation}\label{eq:trace}
	\|v\|_{\partial \k}^2
	\le \trineq\left( \frac{\zeta}{h_\k} \|v\|_{\k}^2
	+ \frac{h_\k}{\zeta} \|\nabla v\|_{\k}^2 \right),
	\end{equation}
	for any   $\zeta>0$ and  $\trineq$ a positive constant only depending on $\Csh$ and on $d$.

Also, for each $v\in\mathcal{P}_p(\k)$, the inverse estimate
	\begin{equation}\label{eq:inverse_trace}
		\norm{v}{\partial \k}^2
		\le \frac{\tinv}{h_\k}
		 \norm{v}{\k}^2,
	\end{equation}
holds with $\tinv:= \Csh {(p+1)(p+d)}$.
\end{lemma}
\begin{proof}
These estimates are special cases of the corresponding ones presented in~\cite{DGease}.
The trace estimate \eqref{eq:trace} follows from \cite[Lemma 4.7]{DGease}
 in conjunction with Assumption \ref{A1}, while \eqref{eq:inverse_trace} follows from \cite[Lemma 4.4]{DGease}
along with Assumption \ref{A1}.
	\end{proof}

\begin{lemma}\label{PF}
	Given $\k\in\mesh$ satisfying Assumption \ref{A1}, for each $v\in H^1(\k)$, $\k\in\mesh$, we have the bounds
	\begin{equation}\label{eq:PF}
	\ltwo{v-\Pi_0v}{\k} \leq C_{PF} h_\k \ltwo{\nabla v}{\k}
	\end{equation}
	and
	\begin{equation}\label{eq:PF_trace}
	\ltwo{v-\Pi_0v}{\partial \k} \leq \tilde{C}_{PF} \sqrt{h_\k} \ltwo{\nabla v}{\k},
	\end{equation}
	with $\Pi_0:L_2(\Omega)\to S_{\mathcal{T}}^{0}$, the orthogonal $L_2$-projection onto  $S_{\mathcal{T}}^{0}$, the space of element-wise constants; here ${C}_{PF},\tilde{C}_{PF}>0$  depend on $d$ and on $\Csh$ only.
\end{lemma}
\begin{proof}
	The key technical difficulty is to show that ${C}_{PF}$ and $\tilde{C}_{PF}$ are independent of the shape of $\k$. Under  Assumption~\ref{A1}, for $\k\in \mesh$ we can apply the Poincare-Friedrichs inequalities proven in \cite[Theorem 3.5]{zheng2005friedrichs} and \cite[Proposition 2.10]{veeser2011poincare}, with explicit  dependence on the shape-regularity constant $\Csh$ and dimension $d$, yielding \eqref{eq:PF}. Then, \eqref{eq:PF_trace} follows using the trace inequality~\eqref{eq:trace}.
\end{proof}

A crucial technical aspect of the analysis below is the availability of a shape-regular, auxiliary triangulation defined as follows.
\begin{definition}[Auxiliary mesh] \label{dual}
Given the sequence of meshes $\mesh$, we name \emph{auxiliary  mesh} a corresponding  sequence of  conforming simplicial meshes $\amesh$ satisfying: for each $\amesh$ and $T\in\amesh$
\begin{itemize}
\item[(a)] (shape-regularity) the radius $r_T$ of the largest circle inscribed in $T$ is such that $r_T\ge \widehat{\tau}^{-1} h_T$, where $h_T$ denotes the diameter of $T$;
\item[(b)] (local mesh-size compatibility) if $\el$ is such that $\k\cap T\neq \emptyset$, it holds $\widehat{\rho}^{-1}\leq h_T / h_{\k}\leq \widehat{\rho}$,
\end{itemize}
with $\widehat{\tau},\widehat{\rho}>1$ constants independent of the discretization parameters.
\end{definition}

An immediate consequence of the above definition is that, if $\amesh$ is an auxiliary mesh sequence, then the number of intersections of each $T\in\amesh$ with the elements of the corresponding polytopic mesh $\mesh$ is uniformly bounded in function of the shape-regularity and local quasi-uniformity constants of both the polytopic and auxiliary mesh. The proof of this fact follows along the lines of that of Lemma~\ref{lem:neigh}.

We note that the evaluation of the error estimator presented below does \emph{not} require the construction of auxiliary meshes  in practice.
As long as their existence  can be assumed, the a posteriori error bound holds.
Moreover, we do not expect such assumption to limit in any possible way the configurations of very general polytopic mesh sequences allowed by Assumptions~\ref{A1} and~\ref{A2}. Rather, the issue is to show that auxiliary
meshes tightly close to the polytopic meshes can be constructed in principle.
To this end, we present two possible algorithms for the  construction of auxiliary meshes which apply to progressively complicated primal mesh configurations.
Both algorithms are easily and cheaply implementable. As such, if desired, the corresponding auxiliary mesh quality parameters may be computed in practice, thus permitting the explicit evaluation of their impact on the a posteriori error bound.

\subsubsection{Auxiliary sub-mesh}\label{submesh}
Assume that  the mesh $\mesh$ is \emph{fully shape-regular} in the sense that, for each $\el$, each face $F\in\partial\k$ satisfies the shape-regularity property which is stated in Assumption~\ref{A1}(b) for boundary faces. Then, an auxiliary mesh can be simply constructed by joining $F$ to $\mathbf{x}^0_\k$, the centre of star-shapedness of $\el$, for each $\el$  and $F\in\partial\k$.
This approach can be extended to the more general case in which every interface can be replaced by a shape-regular triangulated surface which does not compromise the shape-regularity of the neighbouring elements. For instance, any of the four circular interfaces appearing in Figure~\ref{polygons} (right panel) may be replaced by the segment joining its end-points. This will result in the \emph{auxiliary sub-mesh} shown in Figure~\ref{auxi_mesh} (left panel).

\subsubsection{Constrained Delauney auxiliary mesh}\label{Delauneymesh}
Auxiliary sub-meshes are in general not obvious to construct and, moreover,
employing a sub-mesh is not possible when the element faces are shape-irregular and/or of arbitrarily small size with respect to the elemental size. In this case, it is necessary to consider auxiliary meshes which are not logically sub-meshes of the corresponding polytopic meshes.
One possibility, designed to maximise shape-regularity while maintaining the auxiliary mesh as close as possible to the polytopic mesh in terms of local mesh-size, is to exploit the concept of \emph{constrained Delaunay triangulations}, introduced in~\cite{Chew89} for $d=2$ and generalized to any $d$ in~\cite{Shew98}; see also~\cite{Edel,Shew08}.
We recall their definition, limiting ourselves to the case of interest, namely when the constraints are given by the mesh faces laying on the boundary of $\Omega$.
\begin{definition}[Constrained Delaunay triangulation]\label{Delaunay}
Let $\mathcal{X}=\{\mathcal{P},\mathcal{F}\}$ with $\mathcal{P}$ a set of points in $\Omega \subset\mathbb{R}^d$ and $\mathcal{F}$ the set  of boundary faces of the mesh $\mesh$.
A \emph{constrained Delaunay triangulation}  (CDT) associated to $\mathcal{X}$ is a triangulation of $\Omega$, which conforms to $\mathcal{F}$,  has $\mathcal{P}$  as its set of internal vertices, and satisfies the following \emph{constrained Delaunay} property: for every $k\in\{1,\dots,d\}$ and every $k$-dimensional simplex $S$ in the triangulation which is not on $\partial\Omega$,  there exists a circle $\mathcal{C}$ such that:
\begin{itemize}
\item[(1)] the vertices of $S$ are on the boundary of $\mathcal{C}$;
\item[(2)] if a vertex $P$ of the CDT is in the interior of $\mathcal{C}$, then the straight line connecting $P$ to at least one of the vertices of $S$ intersects $\partial\Omega$. (Then, we say that $P$ cannot be seen from one of the vertices of $S$.)
\end{itemize}
\end{definition}

This definition generalizes the concept of Delaunay triangulations in that, if no constraints are given, it would coincide with the definition of Delaunay triangulations. Moreover, as Delaunay triangulations, CDTs maximizes the minimum angle among all triangulations generated by the cloud of points $\mathcal{P}$ and constrained by $\mathcal{F}$. The existence of CDTs is analyzed in ~\cite{Chew89,Shew98,Shew08}:
constrained Delaunay triangulations always exists for $d=2$ while for $d=3$ they exist if any ridge formed by $\mathcal{F}$ is \emph{strongly Delaunay}. A simplex is strongly Delaunay if  the circle $\mathcal{C}$ of Definition~\ref{Delaunay} does not \emph{enclose} any other point in $\mathcal{X}$.  As shown in~\cite{Shew98,Shew08}, this condition can always be satisfied, possibly after the insertion of a finite number of regular nodes on non-strongly Delaunay ridges in the skeleton of $\mathcal{F}$. Moreover,  once every boundary edge is strongly Delaunay, the restriction of the CDT on each interface is Delaunay.

Given the polytopic mesh $\mesh$, here we consider the  constrained Delaunay triangulation of $\Omega$  with  seeds $\mathcal{P}=\{\mathbf{x}^0_\k\}_{\el}$, possibly after the modification of $\mathcal{F}$ discussed above. We refer to Figure~\ref{auxi_mesh} (right panel) for an illustration.

\begin{remark}
We expect the CDTs associated to $\mesh$ to always satisfy the auxiliary mesh Definition~\ref{dual} owing to their shape-regularity maximisation property and the fact that the seeds in $\mathcal{P}$ are well-distanced by assumption. However,  due to the extreme generality of $\mesh$, proving this fact
appears to be challenging and would result into overly-pessimistic estimation of the shape-regularity and quasi-uniformity constants.
Specifically, the difficulty comes from the contrasting requirements of shape-regularity and local mesh-size compatibility, due to which an element of the CDT may overlap with elements of $\mesh$ which are not direct neighbours; see Figure~\ref{meshT}(right panel) for an example.
Hence, in the \emph{a posteriori} error analysis below, we have opted for keeping the requirements of Definition~\ref{dual}  as an assumption which can be verified economically and sharply in practice.
Indeed, contrary to auxiliary sub-meshes, CDTs can always be constructed and their  construction is, in fact, simpler in general.
If desired, their qualitative parameters  may be easily evaluated using well established and efficient algorithms~\cite{Chew89,Shew98,Shew08}.
\end{remark}

\begin{figure}[t]
	\begin{center}
		\begin{tabular}{cc}
			\includegraphics[scale=0.34]{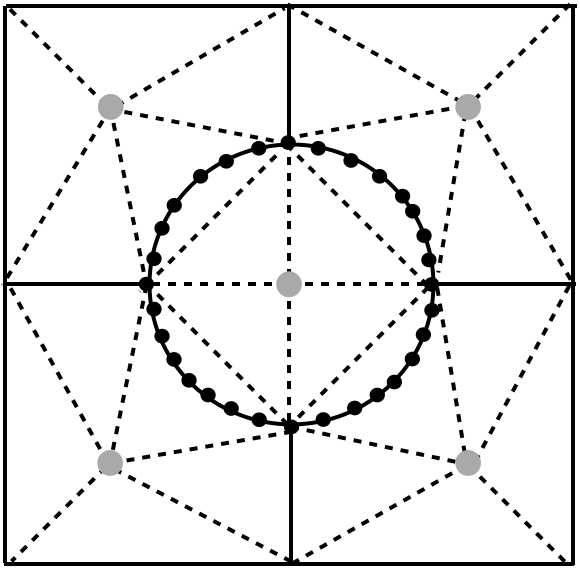}
			 \hspace{0.8cm}
			\includegraphics[scale=0.34]{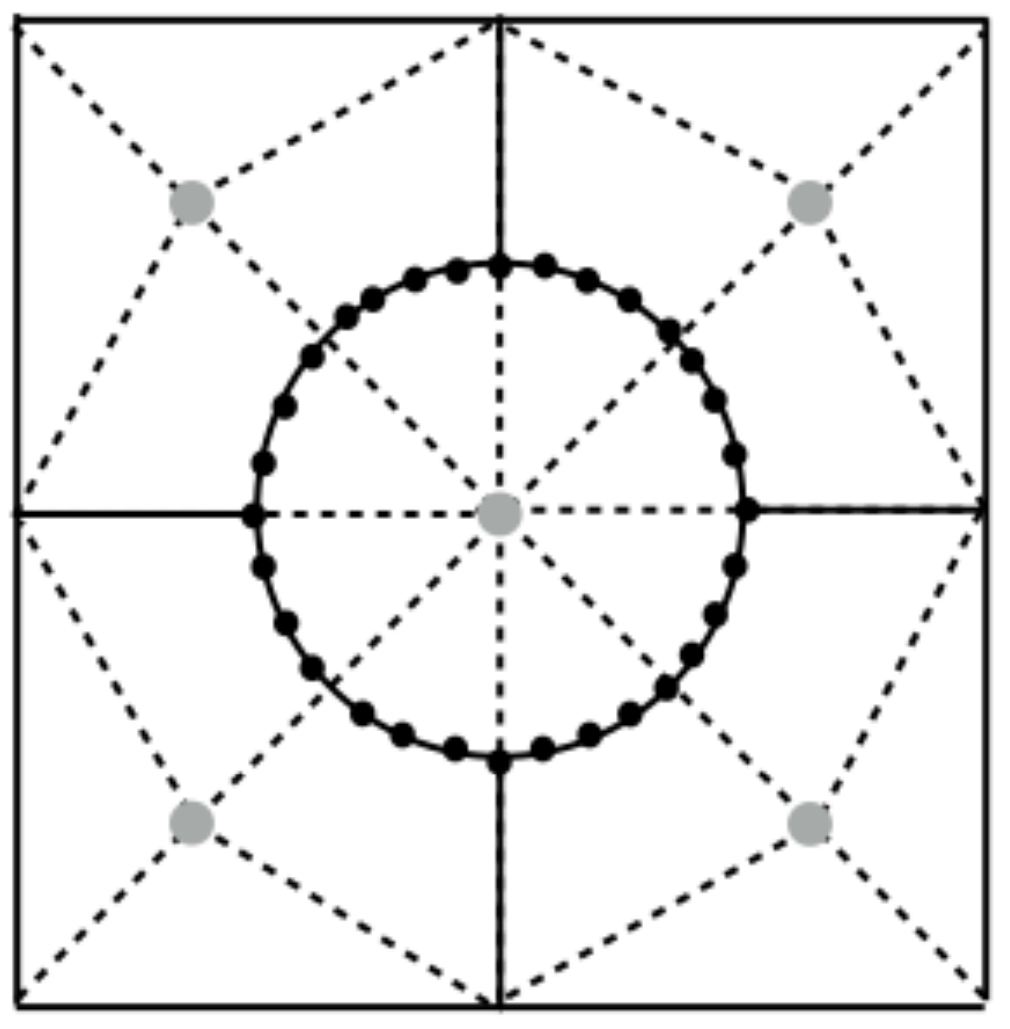}
		\end{tabular}
	\end{center}
	\caption{A mesh with 5 elements (solid lines) and associated auxiliary meshes (dash lines). Left panel: auxiliary sub-mesh with 24 elements. Right panel:  auxiliary constrained Delauney mesh with 16 elements. The dual nodes belonging to the set $\mathcal{P}$  of Definition~\ref{dual} are marked with a grey bullet.}
	\label{auxi_mesh}
\end{figure}

\begin{remark}\label{remark:union}
	In the case of meshes with elements  made of finite unions of star-shaped sub-polytopes,  auxiliary meshes should be constructed starting from the centres of the sub-polytopes instead. The results below still hold, as long as local quasi-uniformity is assumed for the sub-mesh comprising the star-shaped sub-polytopes.
\end{remark}

\subsubsection{Auxiliary mesh interpolation and inverse estimates}

We recall the following Scott-Zhang-type quasi-interpolation result \cite{MR1011446,Melenk,MR1726480,MR1807259}.

\begin{lemma}[Quasi-interpolant]\label{lem:interp}
	Let $\amesh$ be a shape-regular simplicial  subdivision of $\Omega$ not containing any hanging nodes. Then, there exists a quasi-interpolation operator $I_h:H^1(\Omega)\rightarrow H^1(\Omega)\cap S^1_{\widehat{\mathcal{T}}}$, such that
	\begin{equation}
	\label{eq:interpT}
	\ltwo{v-I_h v}{T}+{h_T}\ltwo{\nabla (v-I_h v)}{T}\le \Csz {h_T} \ltwo{\nabla v}{{{\omega}}_T},
	\end{equation}
	where  ${\omega}_T$ denotes the patch of elements in $\amesh$ with non-empty intersection with $\overline{T}$. The constant $\Csz>0$ depends only on the shape-regularity constant $\CshT$ of the auxiliary mesh $\widehat{\mesh}$. If the function $v$ has nonhomogeneous piecewise linear trace on $\partial\Omega$,  we have $	I_h v |_F = v |_F$, for all $ F \subset \partial T \cap \partial \Omega$,  $T \in \amesh$.

Moreover, we have
	\begin{eqnarray}
	\ltwo{v-I_h  v}{\partial\k}^2
	\le \Cin {h_\k} \ltwo{\nabla v}{\widehat{\omega}_{\k}}^2,
	\label{eq:interpF}
\end{eqnarray}
	with $\Cin>0$ depending only on the shape-regularity of $\mesh$ and of $\amesh$. Here, $\widehat{\omega}_{\k}=\cup\{{\omega}_T:\, {T}\cap {\k}\neq \emptyset,\, T\in\amesh \}$.
\end{lemma}
\begin{proof} The proof of \eqref{eq:interpT} can be found in \cite{MR1011446,Melenk,MR1726480,MR1807259} for various levels of generality.
	Noting that \eqref{eq:interpF} refers to the trace on the skeleton of of the \emph{original} polytopic mesh, we apply the trace inequality~\eqref{eq:trace} with  $\zeta=1$,
	\begin{eqnarray}
		\ltwo{v-I_h  v}{\partial\k}^2
		&\le& \trineq \left(h_\k^{-1}\ltwo{v-I_h  v}{\k}^2
		+{h_\k}\ltwo{\nabla(v-I_h  v)}{\k}^2
		\right)
		\nonumber\\
		&\le& \trineq \sum_{T\in\amesh: {T}\cap {\k}
			\neq \emptyset} \left(h_\k^{-1}\ltwo{v-I_h  v}{T}^2+{h_\k}\ltwo{\nabla(v-I_h  v)}{T}^2\right),
		\nonumber
	\end{eqnarray}
	and~\eqref{eq:interpF} follows from~\eqref{eq:interpT}, depending on the shape-regularity of $\mesh$ through $C_{tr}$ and on the shape-regularity of $\amesh$ through~\eqref{eq:interpT}.
\end{proof}

The next polynomial inverse estimation result, relating $L_2$-norms on subsets of the mesh skeleton $\Gamma$ of $\mesh$ with $L_2$-norms over elements of the auxiliary mesh $\amesh$ will be important for the analysis below.
 In this context, for each $T\in\amesh$, we consider the
set of \emph{cut interfaces} obtained  by the intersection of $\Gamma$ with the simplex $T$,
which we characterise as follows:
\begin{equation}\label{eq:gammaT}
\Gamma_T:=\Gamma\cap T=\Gamma_T^{{\rm c}} \cup  \Gamma_T^{{\rm o}},
\end{equation}
with $\Gamma_T^{{\rm c}}$  the set of interfaces $\gamma_T\in\Gamma_T\cap\Omega$ such that $\gamma_T\subset\partial\k$ with $\k$ an element of $\mesh$ whose centre of star-shapedness  is a vertex of $T$, and  $\Gamma_T^{{\rm o}}:=\Gamma_T\setminus \Gamma_T^{{\rm c}}$. Note that the number of interfaces in $\Gamma_T$ is bounded since the number of intersections of $T$ with the elements $\el$ is bounded; however, each such interface may be made of an arbitrary number of cut faces, due to the complexity of the intersecting polytopic elements.

The subdivision in~\eqref{eq:gammaT} reflects increasing levels of difficulty, with $ \Gamma_T^{{\rm o}}$ collecting complex interfaces for which the proof of the inverse estimate is more challenging. As usual, the proof rests in employing simplices obtained by joining each face composing $\Gamma_T$ with the most appropriate vertex of $T$ and summing up all contributions. When $\Gamma_T^{{\rm o}}\neq\emptyset$, such simplices may overlap. In this case, the constant of the resulting inverse estimate depends on  the number of  such overlaps, and thus reflect the complexity of  $\Gamma_T$. The complex interface  $\gamma_T$ highlighted in Figure~\ref{meshT} (right panel) provides an example in which this eventuality may occur.

\begin{remark}
In the case of sub-mesh auxiliary meshes, we always have $\Gamma_T^{{\rm o}}=\emptyset$. Instead, for constrained Delauney auxiliary meshes, $\Gamma_T^{{\rm o}} \neq\emptyset$ in general;
we refer to Figure \ref{meshT} for some examples.
\end{remark}
\begin{figure}[t]
	\begin{center}
	\includegraphics[scale=0.37]{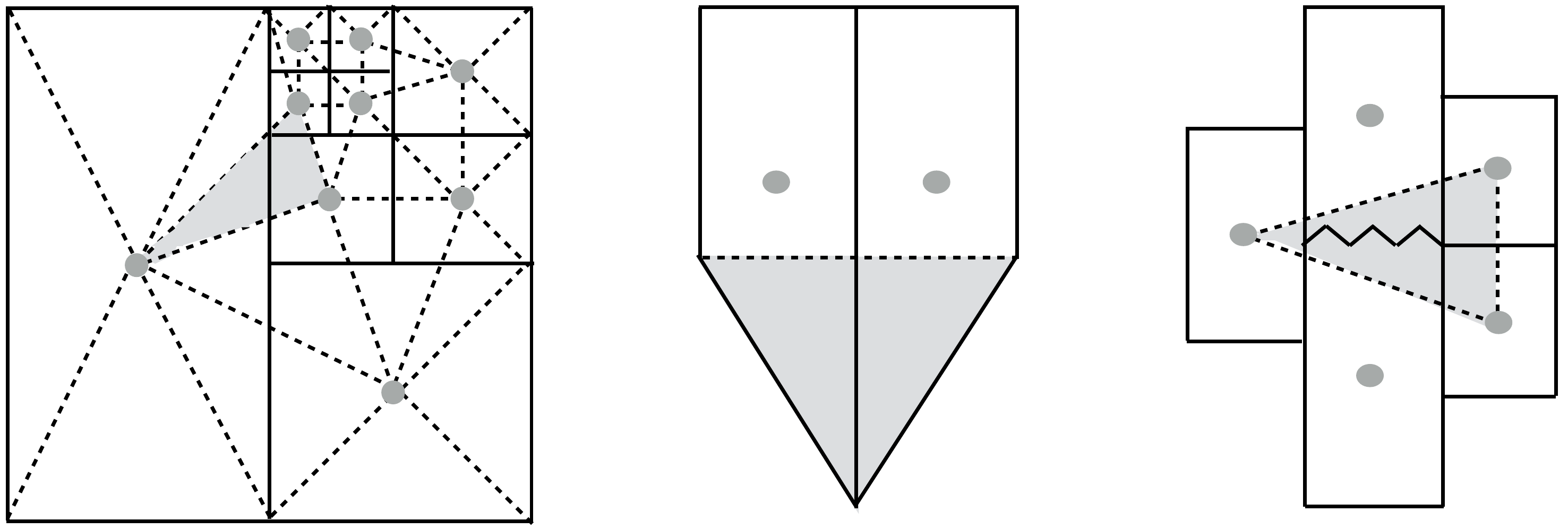}
	\setlength{\unitlength}{1cm}
	\begin{picture}(8,0)
	\put(.2,2.8){{\scriptsize{\makebox(0,0){{$T$}}}}}
	\put(4.3,1.8){{\scriptsize{\makebox(0,0){{$T$}}}}}
	\put(8.85,2.7){{\scriptsize{\makebox(0,0){{$T$}}}}}
	\put(8.25,2.35){{\scriptsize{\makebox(0,0){{$\gamma_{T}$}}}}}
	\end{picture}
	\end{center}
	\caption{Constrained Delauney auxiliary mesh examples. Different configurations for the set   $\Gamma_T$ of~\eqref{eq:gammaT} for the auxiliary element $T$ marked in grey. Left:  $\Gamma_T=\Gamma_T^{{\rm c}}$.
	Centre: $\Gamma_T=\Gamma_T^{{\rm o}}$.
	Right: $\Gamma_T^{{\rm o}}$ is made of a single interface $\gamma_T$ while the rest of $\Gamma_T$ belongs to $\Gamma_T^{{\rm c}}$. The simplices required in the proof of Lemma~\ref{lem:inverseT} to estimate the contributions from each edge in $\gamma_T$ partially overlap (not shown).
	 }
	\label{meshT}
\end{figure}
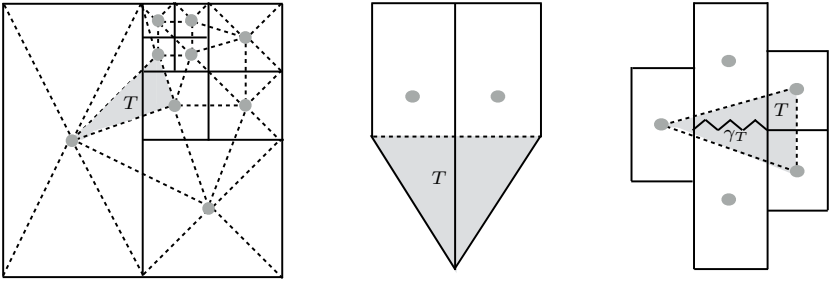
\begin{lemma}\label{lem:inverseT}
	Let Assumptions \ref{A1} and \ref{A2} hold  and let $\amesh$ be an auxiliary mesh, related to $\mesh$. Given $q\in\mathbb{N}$, for all $T \in \amesh$ and $v\in \mathcal{P}_{q}(T)$, we have
	\begin{equation}
	\ltwo{v}{\Gamma_T}^2
	\le \frac{\tinvT}{h_T}\norm{v}{T}^2.
	\label{eq:inverseF}
	\end{equation}
	The constant $\tinvT$ depends on $d$, $q$, and on shape-regularity and local quasi-uniformity of both $\mesh$ and $\amesh$ only.  If  $\Gamma_T\equiv \Gamma_T^{{\rm c}}$ then  $\tinvT=\Csh {{(q+1)(q+d)}}$. If  $\Gamma_T^{{\rm o}}\neq\emptyset$, $\tinvT$ depends also on the number  of overlaps $n_{{\rm o}}$ required to cover the elements of $\Gamma_T^{{\rm o}}$, cf.~\eqref{eq:gamma_o_bound}. In particular, if $\Gamma_T^{{\rm o}}\subset\partial\Omega$, then  $n_{{\rm o}}\le d$.
\end{lemma}
\begin{proof}
	We consider $\Gamma_T^{{\rm c}}$ and $\Gamma_T^{{\rm o}}$ separately, starting with $\Gamma_T^{{\rm c}}$.
	Exploiting the star-shapedness property of $\Gamma_T^{{\rm c}}$ with respect to the vertices of $T$, which is  inherited from Assumption~\ref{A1},
	the inverse inequality
	\[
	\ltwo{v}{\Gamma_T^{{\rm c}}}^2\le \frac{\Csh {{(q+1)(q+d)}}}{h_T}
	\norm{v}{T}^2,
	\]
	follows in the same way  as~\eqref{eq:inverse_trace}.

	Considering now the set $\Gamma_T^{{\rm o}}$,
	we observe that each cut interface belonging to $\Gamma_T^{{\rm o}}$ may be partitioned into a set of $(d-1)$-dimentional simplices.
	Indeed, each interface inherits a set of, possibly cut, simplicial faces from $\Gamma$. If a face $F\in\Gamma$ is only partially contained in $\Gamma_T^{{\rm o}}$, then $F\cap T$ is still an interval if $d=2$ while it can always be subdivided into four triangles if $d=3$.
	Let now $\widehat{F}$ be one such ($d-1$)-dimentional simplex within $\Gamma_T^{{\rm o}}$.
	We note that, if  a simplex  $T$ has inradius $r_T$, then,  for any given intersecting hyperplane $Z$, there exists a vertex $V$ of $T$ such that $d(V,Z)\ge r$. Otherwise, $S$ must be contained in the region $\{ {\bf x}\in {\mathbb R}^d: d({\bf x}, Z)<r\}$, in contradiction with the fact that $T$ contains a (closed) ball of radius $r$. It follows that we can always construct a non degenerate simplex $T_{\widehat{F}}$ by joining $\widehat{F}$ with a vertex $V$ of $T$ such that $d(V,\widehat{F})\ge r$. We thus have, cf.~\eqref{eq:inverse_trace}, the inverse estimate:
	\begin{equation}\label{eq:invFhat}
		\ltwo{v}{\widehat{F}}^2\le \CshT\frac{(q+1)(q+d)}{h_T}
		\norm{v}{T_{\widehat{F}}}^2\qquad\forall v\in \mathcal{P}_{q}(T).
	\end{equation}
	Then, summing up over all $\widehat{F}\in\Gamma_T^{{\rm o}}$ we conclude
	\begin{equation}\label{eq:gamma_o_bound}
	\ltwo{v}{\Gamma_T^{{\rm o}}}^2\le \frac{\CshT {{(q+1)(q+d)}}}{h_T} \sum_{\widehat{F}\in\Gamma_T^{{\rm o}}}
	\norm{v}{T_{\widehat{F}}}^2
	\le n_{{\rm o}} \frac{\CshT {{(q+1)(q+d)}}}{h_T}
	\norm{v}{T}^2,
	\end{equation}
	 with $n_{{\rm o}}$ the number of overlaps of the simplices $T_{\widehat{F}}$, $\widehat{F}\in\Gamma_T^{{\rm o}}$.
	In particular, if $\Gamma_T^{{\rm o}}$ is only made of boundary interfaces, then  $n_{{\rm o}}\le d$, as there may be at most $d$ such interfaces, each made of a single $(d-1)$-dimensional simplex.
	The required estimate now follows by summing up the contributions from $\Gamma_T^{{\rm c}}$ and $\Gamma_T^{{\rm o}}$.
\end{proof}

	\begin{remark}
	The constant $\tinvT$ appearing in \eqref{eq:inverseF} accounts for the complexity of the mesh in terms of topology and shape, quantified by the number $n_{{\rm o}}$ of overlap required to cover the mesh skeleton, see the mesh shown in Figure~\ref{meshT} (right) for an illustrative example.
	In typical practical cases, e.g., meshes stemming from standard algorithms such as Voronoi tessellations, as well as shape-regular adaptively generated meshes, we expect to have $\Gamma_T\equiv \Gamma_T^{{\rm c}}$ for the vast majority of auxiliary elements.
	For instance, for the adaptively refined mesh with multiple hanging-nodes shown Figure~\ref{meshT} (left), $\Gamma_T^{{\rm o}}$ is either empty or it contains a single boundary edge, i.e., no overlaps are required, resulting in  the `ideal' constant $\tinvT=\Csh {{(q+1)(q+d)}}$ for each $T\in\amesh$.
	\end{remark}

\subsection{Discontinuous Galerkin method}

Let $\mathcal{V}:=\fes+H^1(\Omega)$. The symmetric interior penalty discontinuous Galerkin method
reads: find $u_h\in \fes$ such that
\begin{equation}\label{dg}
	\bform(u_h,v_h)=\ell(v_h) \qquad\text{for all }\quad v_h\in \fes,
\end{equation}
whereby $\bform(\cdot, \cdot ):\mathcal{V}\times \mathcal{V}\to \mathbb{R}$ is defined by
\begin{equation}\label{dg-bilinear_inconsistent}
	\begin{aligned}
		\bform(w,v) :=&  \int_\Omega   \diff\nabla_h w \cdot \nabla_h v \ud \uu{x}+\int_{\Gamma\backslash \Gamma_{\dn}}\sigma \jump{w} \jump{v}  \ud{s},  \\
		&-\int_{\Gamma\backslash \Gamma_{\dn}} ( \mean{a (\mbf{\Pi} \nabla w) \cdot \mbf{n}} \jump{v} + \mean{a (\mbf{\Pi}  \nabla v) \cdot \mbf{n}} \jump{w}  ) \ud{s},
	\end{aligned}
\end{equation}
for $w,v\in \mathcal{V}$, and $\ell(\cdot): \mathcal{V}\to \mathbb{R}$ by
\begin{equation*} \label{dg-linear-inconsistent}
	\begin{aligned}
		\ell(v)
		:= \int_{\Omega} f  v \ud \uu{x}
		-\int_{\Gamma_{\ddd}} g_{\rm D}^{} \big( (a (\mbf{\Pi} \nabla v) ) \cdot \mbf{n} -\sigma v \big)\ud s+
		\int_{\Gamma_{\dn}} g_{\dn} v \ud s,
	\end{aligned}
\end{equation*}
with $\mbf{\Pi}: [L_2(\Omega)]^d\to [\fes]^d$ denoting the orthogonal $L_2$-projection operator onto the (vectorial) finite element space, and $\sigma\in L_\infty(\Gamma\backslash \Gamma_{\dn})$ being the, so-called, \emph{discontinuity-penalization function} given by
\begin{equation}
	\sigma({ \uu{x}}) :=\left\{
	\begin{array}{ll}
		C_{\sigma}{\displaystyle   \max_{\k\in\{\k_i,\k_j\}}
			\Big\{  \frac{\bar{\diff}_\k  \tinv }{h_\k} \Big\} }, & ~~ \uu{x}\in  F \in  \Gamma_{\dint}, ~F \subseteq \partial\k_i\cap\partial\k_j, \\ \\
		C_{\sigma}  \displaystyle  \frac{\bar{\diff}_\k \tinv}{h_\k}, & ~~ \uu{x}\in  F \in  \Gamma_{\ddd}, ~F\subset\partial\k,
	\end{array}
	\right. \label{eq:penalty}
\end{equation}
with $C_\sigma$ a positive constant and $\bar{\diff}_\k:=|\sqrt{\diff}|_2^2|_\k$, $\el$; here $|\cdot|_2$ denotes the natural matrix-$l_2$-norm.
The known dependence of the penalty on the local polynomial degree is included in $\tinv$ for brevity; see \cite{DGpolybook,DGease} for details.
Note that, using \eqref{eq:prob} and that $\jump{v}= 0$ on $\Gamma\backslash\Gamma_{\dn}$ for all $v \in H^1_\ddd$, we have
$\bform(u,v) = \ell(v)$ for all $v \in H^1_\ddd$,
with $u\in H^1(\Omega)$ the solution to \eqref{eq:prob}.

\begin{remark}
		To avoid further notational overhead, we opted in exposing the main results for element-wise constant diffusion tensors, i.e., $\diff\in(\feso)^{d\times d}$, and for the classical interior penalty dG method. With minor modifications, the results below can also be extended to more general coefficients. Moreover, we expect that a corresponding analysis to what is presented below holds also for the interior penalty dG variants from \cite{GL_05,DG_22}.
\end{remark}

Upon defining the dG-norm by
$\ndg{v}:=\big(\ltwo{\sqrt{\diff} \nabla_h v}{}^2 +\ltwo{\sqrt{\sigma}
	\jump{v}}{\Gamma\backslash \Gamma_{\dn}}^2\big)^{1/2}$,
we have the following result.
\begin{lemma}\label{lem:coercivity}
	Under Assumption~\ref{A1}, there exists
	$C_\sigma>0$, such that
	\begin{equation}\label{coer_cont}
	\bform(v,v) \geq C_{\rm coer} \ndg{v}^2 \quad\text{and}\quad 	\bform(w,v)\le C_{\rm cont}\ndg{w}\,\ndg{v} \quad\text{for all}\quad v\in \mathcal{V},
	\end{equation}
	respectively,
	with $C_{\rm coer},C_{\rm cont}>0$, independent of $h$ of $p$, and $\el$.
\end{lemma}
We refer to \cite{DGpolybook,DGease} for the proof and the explicit definition of $C_\sigma$. \emph{A priori} error bounds are also available \cite{DGpolybook,DGease}.

\section{\emph{A posteriori} error analysis}\label{apost_sec}
The following analysis requires  Assumptions \ref{A1} and \ref{A2} and that an auxiliary mesh  according to Definition \ref{Delaunay} is given.

We decompose the error into two components:
\[
e=u-u_h=(u-u_c)+(u_c-u_h)=:e_c+e_d,
\]
whereby $u_c\in H^1(\Omega)$ is the recovery of the discrete solution $u_h\in \fes$, defined by
\begin{equation}\label{eq:ec}
	\bform(u_c,v)=\bform(u_h,v)\qquad \forall v\in H^1_{\ddd}(\Omega),
\end{equation}
and $u_c = g_{\ddd}$ on $\Gamma_{\ddd}$.
The existence and uniqueness of $u_c$ is guaranteed by the Lax-Milgram Lemma.
\begin{remark}
	The construction of $u_c$ is known in the theory of finite element methods and has been used in various contexts, e.g., in~\cite{MR3335498} for the design of equilibrated flux \emph{a posteriori} error estimators and in \cite{Smears2018} for the analysis of domain decomposition preconditioners. A crucial reason of using this recovery instead of the averaging operator as in \cite{KP}, is that it is essentially independent of the mesh geometry and topology; this is clearly helpful in the present context of very general polytopic meshes.
\end{remark}

\subsection{Bounding the non-conforming error $e_d$}

Inspired by~\cite{dari1996posteriori,Carstensen2002}, cf. also~\cite{Becker03,Cai11}, we decompose the nonconforming error $e_d=u_c-u_h$ further via a  Helmholtz decomposition.
\begin{lemma}
\label{lem:Hel}
	Given that $\Omega$ is simply connected, for any $\mbf{w}\in( L_2(\Omega))^d$, there exists $\xi\in H^1_\ddd (\Omega)$ and $\phi\in [H^1(\Omega)]^{2d-3}$, $d=2,3$, such that
	\begin{equation}\label{eq:helm}
	a{\bf w}=a\nabla\xi+\emph{\text{curl}}\,\phi\quad \text{in }\, \Omega,
	\end{equation}
 and $\phi$ can be chosen so that
	\begin{equation}\label{eq:Neum}
	\emph{\text{curl}}\,\phi \cdot \mbf{n}=0 \quad \text{on }\, \Gamma_{\dn}.
	\end{equation}
	Moreover, the following  relations hold
	\begin{equation}\label{eq:helm-norm}
	\ltwo{\sqrt{a}{\bf w}}{}^2=\ltwo{\sqrt{a}\nabla \xi}{}^2+\ltwo{a^{-1/2}\emph{\text{curl}}\phi}{}^2,
	\end{equation}
	and
	\begin{equation}\label{eq:curl norm}
	\ltwo{\nabla \phi}{}\leq C_{\Omega} \ltwo{\emph{\text{curl}}\phi}{},
	\end{equation}
	with a constant $C_{\Omega}>0$ only depending on $\Omega$.
\end{lemma}
\begin{proof}
	The  proof of~\eqref{eq:helm} is given in~\cite[Theorem 3.1]{dari1996posteriori} for the $d=2$ case and is extended to $d=3$ in~\cite{Carstensen2002}.  Since  $\nabla \xi$ is orthogonal to ${\text{curl}}\,\phi$, from the symmetry of the diffusion tensor $a$, the orthogonality \eqref{eq:helm-norm} follows immediately. Finally, the proof of~\eqref{eq:curl norm}  can be found in \cite{Carstensen2002}; for $d = 2$, we have $C_\Omega = 1$.
\end{proof}
\begin{remark}
\label{rem:Hel}
The Helmholtz decomposition can be generalised to multiply connected domains~\cite{GR}. However, concerning the validity in this setting of the relation~\eqref{eq:curl norm}, which is fundamental to our analysis, we are only aware of the recent preprint~\cite{bertrand2022stabilization}.
For this reason, we prefer to limit the current analysis to the simply connected setting leaving possible extensions to future work.
\end{remark}.

 Condition \eqref{eq:Neum} imposes a constraint on $\Gamma_{\dn}$ for $\phi \in [H^1(\Omega)]^{2d-3}$. Namely,  $\text{curl}\,\phi \cdot \mbf{n}_F=0$ on $F\subset  \Gamma_{\dn}$, implying that  $\phi \in [H^1(\Omega)]^{3}$  has constant components on each $(d-1)$-dimensional planar subset of $\Gamma_\dn$.

We apply the Helmholtz decomposition with $\mbf{w}=\nabla_h e_d$. Hence $\xi\in H^1_\ddd(\Omega)$ and $\phi\in [H^1(\Omega)]^{2d-3}$ are such that $a\nabla_h e_d=a\nabla\xi+\text{curl}\,\phi$, and we have
\begin{align} \label{nonconforming error}
		 \ltwo{\sqrt{\diff} \nabla_h e_d}{}^2
		=& \int_\Omega a \nabla_h e_d\cdot\nabla\xi  \ud \uu{x}+\int_\Omega \nabla_h e_d\cdot \text{curl}\,\phi \ud \uu{x}.
\end{align}

Since $u_c\in H^1(\Omega)$ with $u_c|_{\Gamma_{\ddd}} = g_{\ddd}^{}$ and $\xi\in H^1_\ddd(\Omega)$,
\eqref{eq:ec} implies
\[
\int_\Omega a\nabla_h e_d\cdot\nabla\xi  \ud \uu{x} =
-\int_{\Gamma_\dint}  \mean{a (\mbf{\Pi}  \nabla \xi) \cdot \mbf{n}} \jump{u_h}  \ud{s}
-\int_{\Gamma_\ddd} (a (\mbf{\Pi}  \nabla \xi) \cdot \mbf{n}) (u_h-g_\ddd)  \ud{s}.
\]
Hence, using the Cauchy-Schwarz inequality,  the trace inverse estimate \eqref{eq:inverse_trace},
the definition of $\sigma$, and the orthogonality \eqref{eq:helm-norm}, we have, respectively,
\begin{align}\label{eq:nonc_first_term}
& \int_\Omega a \nabla_h e_d\cdot\nabla\xi  \ud \uu{x}\nno \\
& \le
\ltwo{\sigma^{-1/2}   \mean{a (\mbf{\Pi}  \nabla \xi) }}{\Gamma\backslash\Gamma_{\dn}}
\Big(\ltwo{\sqrt{\sigma}
	\jump{u_h}}{\Gamma_{\dint}}^2
+\ltwo{\sqrt{\sigma}
	(u_h-g_\ddd)}{\Gamma_{\ddd}}^2 \Big)^{1/2}  \nno \\
& \le
\Big({C_\sigma}^{-1}  \su
\ltwo{ { \mbf{\Pi} (\sqrt{a} \nabla \xi )}}{\k}^2
\Big)^{1/2}
\Big( \ltwo{\sqrt{\sigma}
	\jump{u_h}}{\Gamma_{\dint}}^2
+\ltwo{\sqrt{\sigma}
	(u_h-g_\ddd)}{\Gamma_{\ddd}}^2 \Big)^{1/2}   \\
& \le
(C_\sigma)^{-1/2}
\ltwo{ {  \sqrt{a} \nabla \xi }}{}
\Big( \ltwo{\sqrt{\sigma}
	\jump{u_h}}{\Gamma_{\dint}}^2
+\ltwo{\sqrt{\sigma}
	(u_h-g_\ddd)}{\Gamma_{\ddd}}^2 \Big)^{1/2}  \nno \\
& \le
\ltwo{ {  \sqrt{a} \nabla_h e_d }}{}
\Big( \ltwo{\sqrt{\sigma C_\sigma^{-1}}
	\jump{u_h}}{\Gamma_{\dint}}^2
+\ltwo{\sqrt{\sigma C_\sigma^{-1}}
	(u_h-g_\ddd)}{\Gamma_{\ddd}}^2 \Big)^{1/2}  \nno .
\end{align}
To bound the second term on the right-hand side of~\eqref{nonconforming error},  we first decompose
\begin{equation}\label{non-conforming-relation}
\begin{aligned}
	&\int_\Omega \nabla_h e_d\cdot \text{curl}\,\phi \ud \uu{x}
	=\int_\Omega \nabla_h e_d\cdot \text{curl}(\phi-I_h\phi) \ud \uu{x}
	+\int_\Omega \nabla_h e_d\cdot \text{curl}\,I_h \phi \ud \uu{x},
\end{aligned}
\end{equation}
with $I_h $ the (component-wise if $d=3$) quasi-interpolation operator of Lemma \ref{lem:interp}.

Starting with the first term, observing that  $ \text{curl}\nabla u_c =\mbf{0}$ and using the fact that $\phi\in [H^1(\Omega)]^{2d-3}$ satisfying  \eqref{eq:Neum}, implying that $\phi$ is a constant function  on each planar section of $\Gamma_{\dn}$, and choosing $\phi=I_h\phi$ on $\Gamma_{\dn}$. Then we have
$$
	\int_\Omega \nabla u_c\cdot \text{curl}(\phi-I_h\phi) \ud \uu{x}
	= \int_{\Gamma_\ddd}   (\phi -I_h \phi)\tnabla g_\ddd  \ud{s} . \nno
$$
Applying integration by parts, observing that  $ \text{curl}\nabla e_d|_\k=\mbf{0}$, and using $(\phi - I_h\phi)$ is single valued on each face, and $(\phi - I_h\phi)_{\Gamma_{\dn}}=0$, yields
\begin{equation}\label{eq:nonc_second_term_I}
\begin{aligned}
	&\sum_{\k\in\mesh}
	\int_\k \nabla e_d\cdot \text{curl}(\phi-I_h\phi) \ud \uu{x}
	\\
	=&\ -\int_{\Gamma} (\phi - I_h\phi) \jump{ \tnabla u_h}\ud{s}
	+\int_{\Gamma_\ddd}   (\phi -I_h \phi)\tnabla g_\ddd  \ud{s}
	\\
	=&\ -\int_{\Gamma_\dint} (\phi -I_h \phi)\jump{\tnabla u_h}\ud{s}
	-\int_{\Gamma_\ddd} (\phi -I_h \phi)\tnabla (u_h-g_\ddd)\ud{s}.
\end{aligned}
\end{equation}

Further, using~\eqref{eq:interpF} and, finally,  \eqref{eq:curl norm} and \eqref{eq:helm-norm}, the right-hand side of \eqref{eq:nonc_second_term_I} can be further estimated from above by
\begin{equation}
\begin{aligned} \label{eq:nonc_second_term_II}
	&  \su
	\ltwo{h_\k^{-1/2}(\phi -I_h \phi)}{\partial \k \backslash \Gamma_\dn}
	\left(\ltwo{\sqrt{h}_\k\jump{\tnabla u_h}}{\partial\k\cap \Gamma_{\dint}}+ \ltwo{\sqrt{h}_\k\tnabla (u_h-g_\ddd)}{\partial \k\cap \Gamma_{\ddd}} \right) \\
	& \leq \big( \su \trineq C_{I} \ltwo{ \nabla \phi }{\widehat{\omega}_{\k} }^2 \big)^{1/2} \big(  \ltwo{\sqrt{h}\jump{\tnabla u_h}}{\Gamma_{\dint}}^2
	+\ltwo{\sqrt{h}\tnabla (u_h-g_\ddd)}{\Gamma_{\ddd}}^2 \big)^{1/2}
	\\
	& \leq C_1  \ltwo{ \nabla \phi }{\Omega } \big(  \ltwo{\sqrt{h}\jump{\tnabla u_h}}{\Gamma_{\dint}}^2
	+\ltwo{\sqrt{h}\tnabla (u_h-g_\ddd)}{\Gamma_{\ddd}}^2 \big)^{1/2}
	   \\
	& \leq   C_1 C_\Omega \sqrt{\alpha^*}\ltwo{\sqrt{\diff}\nabla_h e_d}{}
	\big(  \ltwo{\sqrt{h}\jump{\tnabla u_h}}{\Gamma_{\dint}}^2
	+\ltwo{\sqrt{h}\tnabla (u_h-g_\ddd)}{\Gamma_{\ddd}}^2 \big)^{1/2},
\end{aligned}
\end{equation}
for $C_1$ a constant depending only on  the shape-regularity constant $\CshT$ of the  auxiliary mesh $\amesh$, and on the local quasi-uniformity constants $\rho$ and $\hat{\rho}$.

We now consider the second term in~\eqref{non-conforming-relation}. Since $I_h\phi\in [H^1(\Omega)]^{2d-3}$, we have $\nabla\cdot\text{curl}\, I_h\phi=0$ on $\Omega$ and, hence,  $\jump{\text{curl} \,I_h \phi \cdot \mbf{n}}=0$ on $\Gamma_\dint$. Moreover, given that  $I_h \phi$ is constant on each component of $\Gamma_\dn$, we also have ${\text{curl} \,I_h \phi \cdot \mbf{n}}=0$ on $\Gamma_\dn$.
Then, integration by parts and working as above gives
\begin{align}
	\int_\Omega \nabla_h e_d\cdot \text{curl}\, I_h \phi \ud \uu{x}
	=& \sum_{\k\in\mesh}
	\int_{\partial\k}e_d\mbf{n}_\k \cdot \text{curl}\, I_h \phi   \ud{s} & \nno \\
	= & -\int_{\Gamma_\dint} \jump{u_h} (\mbf{n} \cdot \text{curl}\, I_h \phi ) \ud{s}
	-\int_{\Gamma_\ddd} (u_h-g_\ddd) (\mbf{n} \cdot \text{curl}\, I_h \phi ) \ud{s}.
	\nno
\end{align}
Next, applying to $I_h \phi$ the trace inverse  inequality with respect to the auxiliary mesh $\amesh$ given in~\eqref{eq:inverseF}, we obtain
\[
\begin{aligned}
	&\int_\Omega \nabla_h e_d\cdot \text{curl}\, I_h \phi \ud \uu{x} \\
	&  \leq
	\big( \ltwo{\sqrt{\sigma}\jump{u_h}}{\Gamma_\dint}^2
	+ \ltwo{\sqrt{\sigma}(u_h-g_\ddd)}{\Gamma_\ddd}^2 \big)^{1/2}
	\Big(\sum_{T\in \amesh }
	\ltwo{{\sigma^{-1/2}}\text{curl}\, I_h \phi}{ T\cap(\Gamma\backslash \Gamma_\dn) }^2  \Big)^{1/2}  \\
	& \leq
	\big( \ltwo{\sqrt{\sigma}\jump{u_h}}{\Gamma_\dint}^2
	+ \ltwo{\sqrt{\sigma}(u_h-g_\ddd)}{\Gamma_\ddd}^2 \big)^{1/2} \Big(\sum_{T\in \amesh }  {\tinvT h_T^{-1}} \ltwo{{\sigma^{-1/2}}\text{curl}\, I_h \phi}{T}^2   \Big)^{1/2}  \\
	&\leq
	C_2
	\big( \ltwo{\sqrt{\sigma C_\sigma^{-1}}\jump{u_h}}{\Gamma_\dint}^2
	+ \ltwo{\sqrt{\sigma C_\sigma^{-1}}(u_h-g_\ddd)}{\Gamma_\ddd}^2 \big)^{1/2}
	\ltwo{\text{curl}\, I_h \phi}{},
\end{aligned}
\]
for $C_2>0$ constant depending on $\tinv$, $\tinvT$, $\widehat{\rho}$, and on $\alpha^*$.
Next, we use the stability of $I_h$ from Lemma~\ref{lem:interp}, together with~\eqref{eq:curl norm} to deduce
\begin{equation}\label{eq:nonc_second_term_III}
	\ltwo{\text{curl}\, I_h \phi}{\Omega}
	\leq
	C_c
	\ltwo{\nabla \phi}{\Omega}  \le
	C_c C_\Omega \sqrt{\alpha_*}
	\ltwo{a^{-1/2}\nabla \phi}{} \leq     C_c C_\Omega \sqrt{\alpha_*}
	\ltwo{\sqrt{\diff}\nabla_h e_d}{}.
\end{equation}
Hence, combing \eqref{eq:nonc_first_term}, \eqref{eq:nonc_second_term_I}, \eqref{eq:nonc_second_term_II}, and \eqref{eq:nonc_second_term_III}, we arrive at the bound
\begin{align}\label{nonconforming bound}
	\ltwo{\sqrt{\diff} \nabla_h e_d}{}^2
	&\leq
	\Cnc \Big(  \ltwo{\sqrt{\sigma C_\sigma^{-1}}\jump{u_h}}{\Gamma_\dint}^2
	+ \ltwo{\sqrt{\sigma C_\sigma^{-1}}(u_h-g_\ddd)}{\Gamma_\ddd}^2 \nno \\
	&\hspace{1cm}+ \ltwo{\sqrt{h}\jump{\tnabla u_h}}{\Gamma_{\dint}}^2
	+\ltwo{\sqrt{h}\tnabla (u_h-g_\ddd)}{\Gamma_{\ddd}}^2\Big);
\end{align}
the constant $\Cnc>0$ depends on  $\tinv$, $\tinvT$, $\widehat{\rho}$, $C_\Omega$, and on $\alpha^*$, $\alpha_*$, but is independent from $h$ and the number and measure of the mesh faces.

\begin{remark}\label{KP_noKP}
	We stress that the mesh-size $h$ in \eqref{nonconforming bound} is the local element diameter for $\mesh$, i.e.,\emph{independent} of the measure of the faces and the number of faces per element.
	This new bound refines the, now classical, results in~\cite{KP}, by showing that the dG error has, in fact, two sources: the normal flux and the tangential gradient. By applying the inverse inequality on each face $F$, the $L_2$-norm of the tangential jump can be bounded from above by the $L_2$-norm of the jump term itself, thus recovering the bound in~\cite{KP}. However, such bound would be proportional to $\diam(K)/\diam(F)$, for each face $F\in \Gamma$, which may be severely pessimistic for increasingly small faces $F$.
\end{remark}

\subsection{Bounding the conforming error $e_c$}
 For $v\in H^1_\ddd$, we have
\begin{equation} \label{eq:error equation}
\begin{aligned}
\bform(e,v)&=\ell(v)-B(u_h,v)
= \ell(\eta)-\bform(u_h,\eta),
\end{aligned}
\end{equation}
with $\eta:=v-v_h$,  for any $ v_h\in \fes$. Recalling that $e=e_c+e_d$, since $e_c\in H^1_\ddd$ we can fix $v=e_c$ in~\eqref{eq:error equation} to further deduce
\begin{equation}\label{eq:ec_error}
\ltwo{\sqrt{a}\nabla e_c}{}^2=\bform(e_c,e_c)=(\ell(\eta)-\bform(u_h,\eta))-\bform(e_d,e_c)=\ell(\eta)-\bform(u_h,\eta),
\end{equation}
from~\eqref{eq:ec}.  The right-hand side of \eqref{eq:ec_error} can now be bounded via standard arguments \cite{KP}: integration by parts, application of \cite[Eq.~(3.3)]{unified}, the observation that $\mbf{\Pi}\nabla u_h = \nabla u_h$, and elementary manipulations yield
\begin{align}\label{conformoing_bound}
\hspace{0cm}\ell(\eta)-\bform(u_h,\eta)
&= \int_{\Omega} (f+\nabla_h \cdot (a\nabla_h u_h))  \eta \ud \uu{x}
-\int_{\Gamma_{\dint}}  \jump{a \nabla u_h \cdot \mbf{n}} \mean{\eta}
\ud{s}
\nno \\
&\hspace{-.5cm}
-\int_{\Gamma_{\dn}} (a\nabla u_h \cdot \mbf{n}-  g_{\dn}) \eta \ud s
-\int_{\Gamma_{\dint}}   \sigma \jump{u_h} \jump{\eta} \ud{s}
-\int_{\Gamma_{\ddd}}  \sigma (u_h - g_{\rm D}^{} ) \eta \ud s
\nno \\
& \hspace{-.5cm}
+\int_{\Gamma_{\dint}}  \mean{a (\mbf{\Pi}  \nabla \eta) \cdot \mbf{n}} \jump{u_h}  \ud{s}
+\int_{\Gamma_{\ddd}} \big( a (\mbf{\Pi} \nabla \eta)  \cdot \mbf{n}\big) (u_h - g_{\rm D}^{} )\ud s
=: \sum_{i=1}^7 T_i.
\end{align}
Setting $\eta = e_c-\Pi_0 e_c$ and using \eqref{eq:PF}, we have
\begin{equation} \label{comfirming relation1}
T_1
 \leq C_{PF}\alpha_*^{-1/2} \Big( \sum_{\k\in \mesh} \ltwo{h_\k(f+\nabla \cdot  (\diff\nabla u_h)}{\k}^2 \Big)^{1/2}\ltwo{\sqrt{a}\nabla e_c}.
\end{equation}
Employing \eqref{eq:PF_trace}, along with standard manipulations, we also have
\begin{align} \label{comfirming relation2}
T_2+T_3& \leq   \su \sum_{F \subset \partial \k\cap \Gamma_{\dint}}  \ltwo{h_\k^{-1/2}\eta}{F}
\ltwo{\sqrt{h}_\k\jump{a \nabla u_h \cdot \mbf{n}}  }{F}
\nno \\
& +\su  \sum_{F \subset \partial \k\cap \Gamma_{\dn}}   \ltwo{h_\k^{-1/2}\eta}{F}
\ltwo{\sqrt{h}_\k(a\nabla u_h \cdot \mbf{n}-  g_{\dn})}{F}   \\
& \hspace{-1cm} \leq \tilde{C}_{PF}\alpha_*^{-1/2}\ltwo{\sqrt{a}\nabla e_c}{}\Big( \ltwo{\sqrt{h}\jump{\diff\nabla u_h\cdot \mbf{n}}}{ \Gamma_{\dint}}^2
+  \ltwo{\sqrt{h}(\diff\nabla u_h\cdot \mbf{n}-g_\dn) }{ \Gamma_{\dn}}^2
\Big)^{1/2}. \nno
\end{align}
Similarly, using the definition of $\sigma$ from \eqref{eq:penalty},  we have
\begin{align} \label{comfirming relation3}
T_4+T_5& \leq   \big( \ltwo{\sqrt{\sigma}\jump{u_h}}{\Gamma_\dint}^2
+ \ltwo{\sqrt{\sigma}(u_h-g_\ddd)}{\Gamma_\ddd}^2 \big)^{1/2} \ltwo{{\sigma^{1/2}} \jump{\eta}}{\Gamma \backslash \Gamma_{\dn}}
\nno \\
& \leq
\Big(\su  \max_{F\in \partial \k \backslash \Gamma_{\dn}} \sigma
\ltwo{\eta}{\partial \k \backslash \Gamma_{\dn}}^2 \Big)^{1/2}
\big( \ltwo{\sqrt{\sigma}\jump{u_h}}{\Gamma_\dint}^2
+ \ltwo{{\sqrt\sigma}(u_h-g_\ddd)}{\Gamma_\ddd}^2 \big)^{1/2} \\
& \hspace{-1cm}\leq
\sqrt{C_{\sigma}\alpha^*/\alpha_*}\tinv \tilde{C}_{PF} \rho
\ltwo{\sqrt{a}\nabla e_c}{}
\big( \ltwo{{\sigma}^{1/2}\jump{u_h}}{\Gamma_\dint}^2
+ \ltwo{\sqrt{\sigma}(u_h-g_\ddd)}{\Gamma_\ddd}^2 \big)^{1/2}.   \nno
\end{align}
Now, using the trace inverse estimate \eqref{eq:inverse_trace}, the  stability of the $L_2$-projection operator and that $\nabla \eta|_\k = \nabla e_c|_\k$, we deduce
\begin{align} \label{comfirming relation4}
T_6+T_7& \leq   \big( \ltwo{\sqrt{\sigma}\jump{u_h}}{\Gamma_\dint}^2
+ \ltwo{\sqrt{\sigma}(u_h-g_\ddd)}{\Gamma_\ddd}^2 \big)^{1/2} \ltwo{{\sigma^{-1/2}} \mean{a (\mbf{\Pi}  \nabla e_c) \cdot \mbf{n}} }{\Gamma \backslash \Gamma_{\dn}}
\nno \\
&  \hspace{-0cm}\leq
\Big(\su   \max_{F\in \partial \k \backslash \Gamma_{\dn}} \sigma^{-1}   \ltwo{\sqrt{a} (\mbf{\Pi}  (\sqrt{a}\nabla e_c)) \cdot \mbf{n}}{ F}^2 \Big)^{1/2} \nno \\
& \hspace{1cm}\times \big( \ltwo{\sqrt{\sigma}\jump{u_h}}{\Gamma_\dint}^2
+ \ltwo{\sqrt{\sigma}(u_h-g_\ddd)}{\Gamma_\ddd}^2 \big)^{1/2}  \\
& \hspace{0cm}\leq
C_\sigma ^{-1/2}\rho
\ltwo{\sqrt{a}\nabla e_c}{}
\big( \ltwo{\sqrt{\sigma}\jump{u_h}}{\Gamma_\dint}^2
+ \ltwo{\sqrt{\sigma}(u_h-g_\ddd)}{\Gamma_\ddd}^2 \big)^{1/2}.   \nno
\end{align}
Hence, by collecting above bounds \eqref{comfirming relation1}, \eqref{comfirming relation2}, \eqref{comfirming relation3}, \eqref{comfirming relation4} and \eqref{conformoing_bound}, we arrive at the following bound on the conforming error:
\begin{align}\label{eq:conf}
\ltwo{\sqrt{a}\nabla e_c}{}
& \hspace{0cm} \leq
\Cco \Big(
\ltwo{h(f+\nabla_h \cdot  (\diff\nabla_h u_h)}{}^2
+ \ltwo{\sqrt{\sigma}\jump{u_h}}{\Gamma_\dint}^2
+ \ltwo{\sqrt{\sigma}(u_h-g_\ddd)}{\Gamma_\ddd}^2  \nno \\
& \hspace{1cm}+  \ltwo{\sqrt{h}\jump{\diff\nabla u_h\cdot \mbf{n}}}{ \Gamma_{\dint}}^2
+  \ltwo{\sqrt{h}(\diff\nabla u_h\cdot \mbf{n}-g_\dn) }{ \Gamma_{\dn}}^2
\Big)^{1/2},
\end{align}
with $\Cco$ depending on  $C_\sigma$, $\rho$, $\Csh$,  the polynomial degree $p$, $C_{PF}$, and $\tilde{C}_{PF}$, but independent of $h$ and the number and measure of the elemental faces.

We are now ready to present the \emph{a posteriori} error upper bound.

\begin{theorem}[upper bound]\label{thm:upperbound}
	Let $u$ be the solution of \eqref{Problem} and let $u_h\in \fes$ be its dG approximation on a polytopic mesh satisfying Assumptions \ref{A1} and \ref{A2}. Also let an auxiliary mesh  according to Definition \ref{Delaunay} is given. Then, we have the following a posteriori error bound
	\begin{align}\label{full bound}
	\ndg{u-u_h}^2
	&\leq
	C_{\rm up}  \su (R_\k^2 + O_\k^2),
	\end{align}
	with the local estimator  $ R_\k^2 = R_{\k,E}^2+R_{\k,N}^2+R_{\k,J}^2+R_{\k,T}^2$, and the data oscillation  $ O_\k^2 = O_{\k,E}^2+O_{\k,N}^2+O_{\k,J}^2+O_{\k,T}^2$, given by
	\begin{align*}\label{each term}
	& R_{\k,E}:= \ltwo{h(\Pi f+\nabla \cdot  (\diff\nabla u_h))}{\k},
	\\
	& R_{\k,N}:= \big( \ltwo{\sqrt{h}\jump{\diff\nabla u_h\cdot \mbf{n}}}{ \partial \k \cap  \Gamma_{\dint}}^2
	+  \ltwo{\sqrt{h}(\diff\nabla u_h\cdot \mbf{n}-\bar{g}_\dn) }{\partial \k \cap   \Gamma_{\dn}}^2\big)^{1/2} , \\
	& R_{\k,J} := \big(\ltwo{\sqrt{\sigma}\jump{u_h}}{\partial \k \cap   \Gamma_\dint}^2
	+ \ltwo{\sqrt{\sigma}(u_h-  \bar{g}_\ddd)}{\partial \k \cap \Gamma_\ddd}^2\big)^{1/2}, \\
	&R_{\k,T}:=   \big(\ltwo{\sqrt{h}\jump{\tnabla u_h}}{\partial \k \cap \Gamma_{\dint}}^2
	+\ltwo{\sqrt{h}\tnabla (u_h-\bar{g}_\ddd)}{\partial \k \cap \Gamma_{\ddd}}^2\big)^{1/2}, \\
	&
	O_{\k,E} := \ltwo{h(f - \Pi f)}{\k},
	\quad
	O_{\k,N}:=
	\ltwo{\sqrt{h}(g_\dn-\bar{g}_\dn) }{\partial \k \cap   \Gamma_{\dn}},
	\\
	&O_{\k,J} := \ltwo{\sqrt{\sigma}(g_\ddd- \bar{ g}_\ddd)}{\partial \k \cap \Gamma_\ddd},
	\quad
	O_{\k,T} := \ltwo{\sqrt{h}\tnabla (g_\ddd- \bar{ g}_\ddd)}{\Gamma_{\ddd}},
	\end{align*}
	with $C_{\rm up}$ depending on $\Cco$ and $\Cnc$ only, but is independent of $h$ and of the number and measure of the elemental faces; here, for any for $ \k \in \mesh$, such that $\partial \k \cap \Gamma_{S} \neq  \emptyset$ with $S\in\{D,N\}$,  we set $\bar{g}_S^{}|_{\partial \k\cap \Gamma_{S}}^{} \in \mathcal{P}_{p_k}(\partial \k\cap \Gamma_{S})$ with $g_S$ denoting an approximation of the Dirichlet and Neumann data, respectively.
\end{theorem}
\begin{proof}
	The proof follows immediately from the bounds \eqref{nonconforming bound} and \eqref{eq:conf}, together with the triangle inequality $\ndg{u-u_h}\leq \ndg{e_c}+ \ndg{e_d}$.
\end{proof}
\begin{remark}
	In the above, we followed a known approach in splitting the estimator into a `residual part' and a `data oscillation part', assuming that $f \in L_2(\Omega)$ and for sufficiently smooth boundary data. In this setting the data oscillation error is typically dominated by the residual estimators. However, if the forcing data $f \in H^{-1}(\Omega)$, then data oscillation may dominate the error~\cite{KV}. It would be an interesting future development to investigate the approach from~\cite{KV} in the context of discontinuous Galerkin methods.
\end{remark}

\begin{remark}\label{remark:dirichlet}
	Theorem~\ref{thm:upperbound} has been proven under Assumption~\ref{A1}(b) which disallows boundary  faces with arbitrarily small size relative to the local mesh size. This assumption is reasonable in as much resolution of the problem domain is required in order for the numerical solution to incorporate the boundary conditions.
	Nevertheless, this assumption can be relaxed in the case of Dirichlet boundary conditions as follows. Noting that the latter is only required to construct the interpolant of the divergence-free component $\phi$ of the non-conforming error, which is not constrained on the Dirichlet boundary. Thus, the interpolant may be constructed for an extension $\tilde\phi$, (e.g., as a Stein-type extension operator) defined on an extended domain $\tilde{\Omega}\supset \Omega$ whose respective mesh $\tilde{\mesh}$ would correspond `closely' to the primal mesh $\mesh$ and is constructed so that it may contain \emph{no} small boundary faces.
	The resulting bounds, however, would depend on the, typically unknown, boundedness constant of the extension operator.
	\end{remark}

\section{Lower bounds}\label{lower_bounds}
We now derive lower bounds for the a posteriori error estimator of Theorem~\ref{thm:upperbound}. Of particular interest is the extend to which the efficiency of the estimator can be shown to be independent of the number and of the relative sizes of $(d-1)$-dimensional faces in the mesh.
The situation differs for the elemental residual and face jump residuals; for clarity, we deal with them separately.

\subsection{Elemental residual}\label{sec:ele}

Lower bounds for the elemental residual can be derived under no further assumptions on the mesh. The analysis is based on a new element bubble function and some auxiliary results.

\begin{lemma}[{\cite[Corollary 4.24]{DGease}}]\label{inverse_h1}
	Let $\mesh$ satisfy the Assumption \ref{A1}.
	Then, for  each $\el$, $p\in\mathbb{N}$ and  $v\in\mathcal{P}_p(\k)$, the following inverse inequality holds
	\begin{equation}\label{eq:inverse_h1}
	\norm{\nabla v}{\k}^2 \leq  \einv h_\k^{-2} \norm{v}{\k}^2.
	\end{equation}
with  $\einv$ a positive constant depending only on $d$, $p$ and $\Csh$. Note also the trivial inequality $\norm{{\rm curl}\, v}{\k}^2 \leq  (d-1) 	\norm{\nabla v}{\k}^2$.
\end{lemma}

Next, for a generic  $d-$dimensional simplex $T$, we denote its  barycentric co-ordinates by $\lambda_{T}^i$,
$d=0,\dots,d$, and denote by $F_i$, $i=0,\dots,d$
the corresponding $(d-1)$-dimensional simplicial  face of $T$ such that $\lambda_T^i|_{F_i}=0$. Note that
$
\linf{\nabla \lambda_{T}^i}{T} = d|F_i|/|T|,
$
since $\nabla \lambda_{T}^i$ is constant. Importantly, the maximum norm is determined by the distance of the $i$-th vertex from the face $F_i$, but it is \emph{independent of the measure of face $F_i$}, see Figure \ref{fig:bubble_function} (left) for an illustration.
\begin{figure}[!t]
	\begin{center}
		\begin{tabular}{cc}
			\includegraphics[angle=90,scale=0.35]{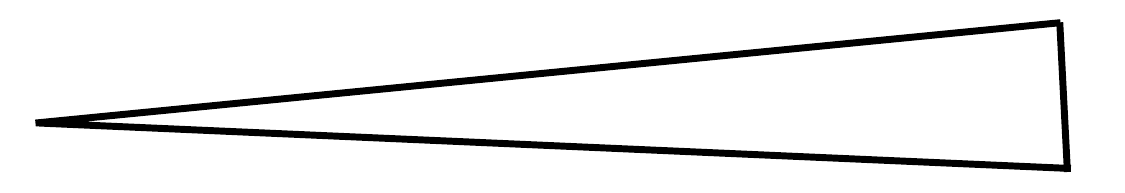} 	\hspace{2cm}
			\includegraphics[scale=0.35]{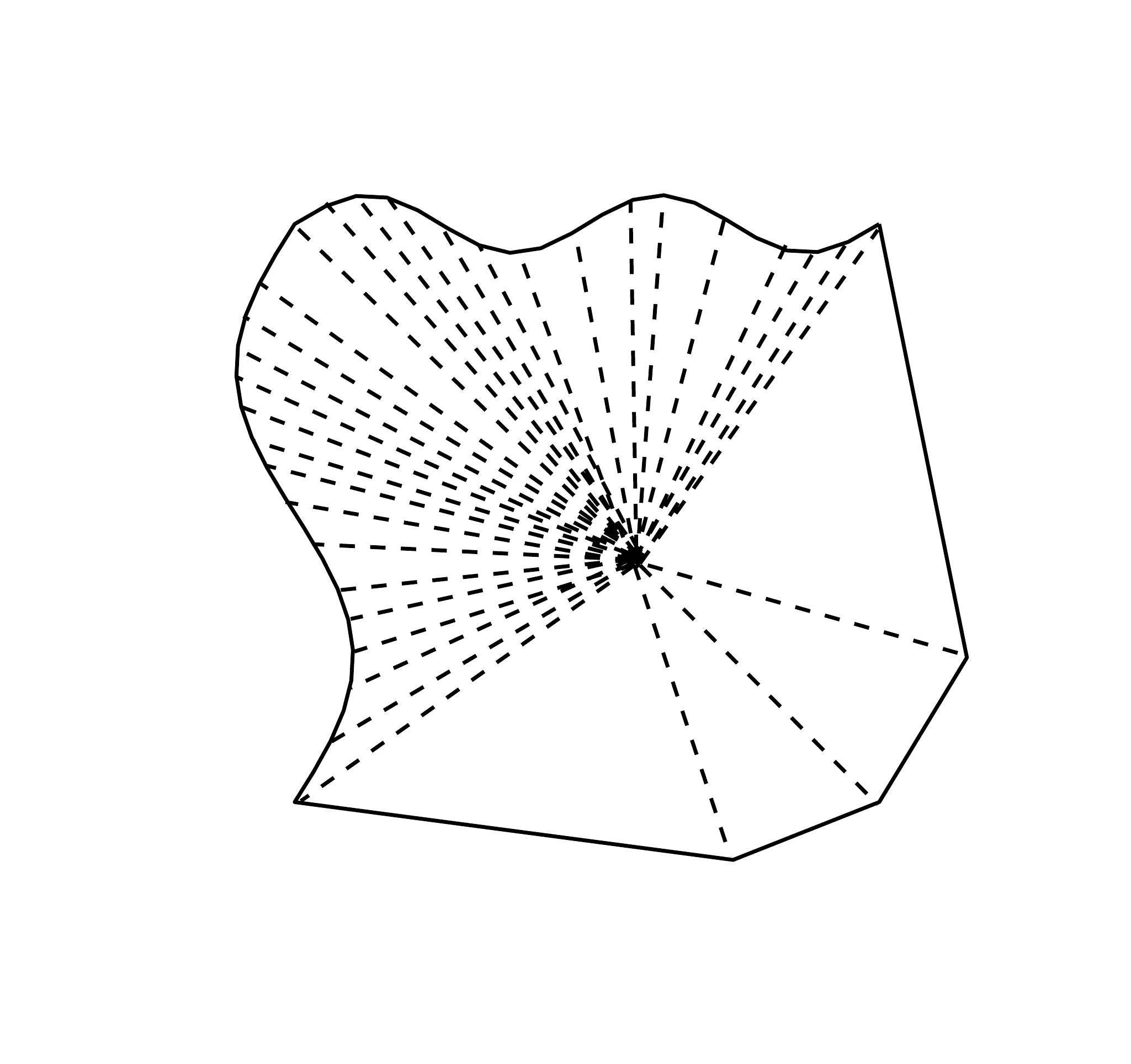}
		\end{tabular}
	\end{center}
	\caption{Left: a triangle with  one `small' face. Right: a polytopic element with many `small' faces.}
	\label{fig:bubble_function}
	\setlength{\unitlength}{1cm}
	\begin{picture}(0,0)
	\put(8.05,2.85){\makebox(-10.,4.8){{$F_i$}}}
	\put(8.05,2.85){\makebox(-10.1,3){{$T$}}}
	\put(8.05,2.85){\makebox(4,4){{$\k$}}}
	\put(8.05,2.85){\makebox(1,0.7){{$O$}}}
	\put(8.05,2.85){\makebox(2,5.2){{$F_i$}}}
	\put(8.05,2.85){\makebox(1.6,4){{$\tau_i$}}}
	\end{picture}
\end{figure}
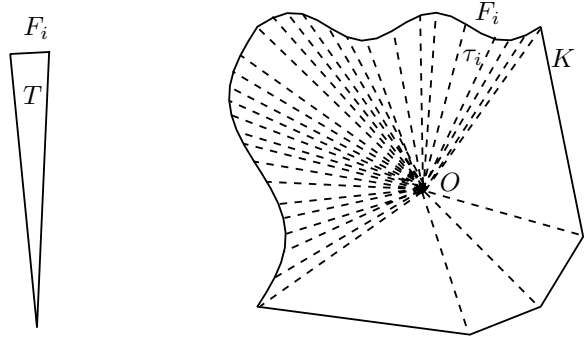

Let $\el$ and let $m_\k$ be the number of its faces. Given that $\k$ is star-shaped by Assumption \ref{A1}, we can construct a non-overlapping subdivision of $\k$ into $m_\k$ simplicial sub-elements $\tau_j$ by joining the face $F_j$, $j=1,\dots,m_\k$,  of $\k$ with the centre of the largest ball inscribed in $\k$; see Figure \ref{fig:bubble_function} (right) for an illustration. Note that 
$h_\k\ge\diam(\tau_j)\ge r_\k\ge \Csh^{-1}h_\k$.
Moreover, letting  $\lambda_j:=\lambda_{\tau_j}^i$ with $i$ such that $\lambda_{\tau_j}^i$ is the barycentric coordinate of $\tau_j$ corresponding to the vertex of $\tau_j$ which is internal to $\k$, it follows that
\begin{equation}\label{nodal_basis_nabla}
 h_\k^{-1}\le\linf{\nabla \lambda_j}{\tau_j} = d|\partial\tau_j\cap \partial \k|/|\tau_j|\le  \Csh h_K^{-1}.
\end{equation}
\begin{definition}[Element bubble]
Let $\el$ and let $m_\k$ be the number of its faces. With the above notation, the element bubble function $b_\k$ is defined as
	\begin{equation}\label{eq:bubble function}
	b_\k|_{\tau_j}  = \lambda_{j},
	\end{equation}
for $j=1,\dots,m_K$.
\end{definition}

By construction, $b_\k$ is a continuous piecewise polynomial function  with zero trace and with values in $[0, 1]$ on $\k$.
Next, we will derive some important properties of the new bubble function \eqref{eq:bubble function}.
\begin{lemma}
	For each $\k\in \mesh$ satisfying the Assumption \ref{A1} and for each $v\in \mathcal{P}_{p}(\k)$, we have
	\begin{equation}\label{eq: H1_norm_element_bubble}
	\ltwo{ \nabla (b_\k v)}{\k}^2 \leq 2(\Csh^2+\einv) h_\k^{-2}\ltwo{v}{\k}^2,
	\end{equation}
 with $\Csh$ as in Assumption \ref{A1} and $\einv$ is in \eqref{inverse_h1},  and
	\begin{equation}\label{eq: L2_norm_element_bubble}
	\ltwo{v}{\k}^2 \leq  C_{b,\k} \ltwo{b_\k^{1/2}v}{\k}^2,
	\end{equation}
	with $C_{b,\k}:=[2(p+2)]^{2d} \frac{d^{d+1}}{ (d-1)!}$.
\end{lemma}
\begin{proof}
	Using the triangle inequality, the bound \eqref{nodal_basis_nabla}, and the inverse inequality \eqref{inverse_h1},   we have, respectively,
	\begin{align}
	\ltwo{ \nabla (b_\k v)}{\k}^2 &\leq 2\ltwo{(\nabla b_\k)v}{\k}^2
	+2\ltwo{b_\k (\nabla v)}{\k}^2   \nno \\
	& \leq 2 \sum_{j=1}^{m_\k} \|\nabla \lambda_j\|_{L_\infty(\tau_j)}^2 \ltwo{v}{\tau_j}^2
	+ 2 \|b_\k\|_{L_\infty(\k)}^2\ltwo{\nabla v}{\k}^2 \nno \\
	& \leq  2 
	\Csh^2h_\k^{-2}
	\sum_{i=1}^{m_\k}  \ltwo{v}{\tau_i}^2
	+ 2 \einv h_\k^{-2} \ltwo{v}{\k}^2
	= 2(\Csh^2+\einv) h_\k^{-2}\ltwo{v}{\k}^2,
	\end{align}
	which is the bound required in~\eqref{eq: H1_norm_element_bubble}.
	We now prove the norm equivalence relation \eqref{eq: L2_norm_element_bubble}. Recalling the  norm equivalence relation for each $v\in \mathcal{P}_p(T)$ on a simplex $T$ from  \cite[Section 3.6]{MR3059294}, we have
	\begin{equation}
	\ltwo{v}{T}^2 \leq  [2(p+2)]^{2d} \Big( \frac{d}{d+1}\Big)^{d+1} \frac{1}{(d-1)!} \ltwo{(\psi_T)^{1/2} v}{T}^2,
	\end{equation}
	with $\psi_T:= (d+1)^{d+1} (\prod_{i=0}^{d}\lambda_T^i)$. Then, by using  $\|\lambda_T^i\|_{L_\infty(T)} = 1$, $i=0,\dots,d$, we deduce
	\begin{equation}\label{eq:classical_result}
	\ltwo{v}{T}^2
	\leq   C_{b,\k}\ltwo{(\lambda_{T}^{i})^{1/2} v}{T}^2,
	\end{equation}
	for each $i\in\{0,\dots,d\}$.
	Hence, the bound \eqref{eq: L2_norm_element_bubble} is proven using the definition of $b_\k(x)$ in \eqref{eq:bubble function}:
	\begin{align}
	\ltwo{v}{\k}^2 =  \sum_{j=1}^{m_\k} \ltwo{v}{\tau_j}^2
	&\leq \sum_{j=1}^{m_\k}  C_{b,\k} \ltwo{(\lambda_{j})^{1/2} v}{\tau_j}^2
	= C_{b,\k} \ltwo{b_\k^{1/2} v}{\k}^2.
	\end{align}
\end{proof}
\begin{remark}
	The new element bubble function on polytopic meshes  in the above Lemma is different from the classical element bubble function on simplices. In particular, we note that the important relations \eqref{eq: H1_norm_element_bubble} and \eqref{eq: L2_norm_element_bubble} are independent of the number and measure of the faces of the element $\k$.
\end{remark}

\begin{theorem}[Elemental residual lower bound]\label{thm:lower_bound_element}
	Let $u$ be the solution of \eqref{Problem} and let $u_h\in \fes$ be its dG approximation under Assumption \ref{A1} and \ref{A2}. Then, for each $\k\in \mesh$, we have
	\begin{equation}\label{eq:lower bound elem}
	\ltwo{h(\Pi f+ \nabla  \cdot  (\diff\nabla u_h))}{\k}^2\leq 2C_{b,\k}
	\big(2\big(\Csh^2+\einv) (\alpha^*)^2 \ltwo{\sqrt{a}\nabla e}{\k}^2+O_{\k,E}^2 \big).
	\end{equation}
\end{theorem}
\begin{proof}
	We  fix $v\in H^1_{\rm D}(\Omega)$ as  $v|_\k = b_\k{(\Pi f+\nabla \cdot (\diff\nabla u_h))}$, where $b_\k$ is the element  bubble function in \eqref{eq:bubble function},  and extended   to zero  outside $\k$. Using  relations \eqref{eq:error equation},  \eqref{conformoing_bound}, \eqref{eq: H1_norm_element_bubble}, and $b_\k\leq 1$,  we obtain
	\begin{align}\label{eq:efficiency 1}
	&\hspace{-0.5cm}\int_{\k}{b_\k{(\Pi f+\nabla \cdot (\diff\nabla u_h))}}^2 \ud \uu{x}
	= \int_{\k}(\Pi f -f)v \ud \uu{x} + \int_{\k} \diff \nabla e \cdot \nabla v \ud \uu{x} \nno \\
	& \leq \ltwo{\Pi f -f}{\k}\ltwo{v}{\k} +\ltwo{\sqrt{a}\nabla e}{\k}\ltwo{\sqrt{a} \nabla v}{\k}  \nno \\
	& \leq
	(\ltwo{\Pi f -f}{\k}
	+ \sqrt{2\big(\Csh^2+\einv)} \alpha^* h_\k^{-1} \ltwo{\sqrt{a} \nabla e}{\k}\big)
	\ltwo{(\Pi f+\nabla \cdot (\diff\nabla u_h)) }{\k},
	\end{align}
	Recalling  \eqref{eq: L2_norm_element_bubble}, we have
	$\ltwo{\Pi f+\nabla \cdot(\diff\nabla u_h)}{\k}^2
	\leq
	C_{b,\k} \ltwo{b_\k^{1/2}(\Pi f+\nabla \cdot (\diff\nabla u_h))}{\k}^2 $ from which the result  \eqref{eq:lower bound elem} already follows.
\end{proof}

\subsection{Flux residuals}\label{sec:jump}

In view of proving the lower bound for the flux residuals, we require the number of faces of each element to be uniformly bounded. Furthermore,  in the case $d=3$ we shall assume that each face $F$ is shape-regular. Note that such assumptions still allows for arbitrarily small faces.
\begin{assumption}\label{A3}
The number of faces of every element $\el$ is uniformly bounded.
For $d=3$ only, for every $F\in\Gamma_{\dint}$,  the radius $r_F$ of the largest $(d-1)$-dimensional ball inscribed in $F$ satisfies $r_F\ge\Csh^{-1}h_F$, with $\Csh$ as in Assumption \ref{A1}.
\end{assumption}
Note that the above assumption does \emph{not} forbid the size of a mesh face to be arbitrarily smaller than that of the elements it belongs to.

To construct the face bubble function, we consider the standard face bubble functions supported in a pair of simplices  contained in the neighbouring elements.
\begin{definition}[Face bubble]\label{def:edge_bubble}
Let $\el$ and $F\subset\partial\k$  a mesh face satisfying Assumption~\ref{A3}. Define $T_F^\k \subset \k$ to be the simplex having  $F$ as a face and opposite vertex the point at distance $h_F$ from $F$ along the segment joining the barycentre of $F$ with the centre of star-shapedness of $\k$.
The face  bubble function $b_F$ is defined on $\k$ as the standard bubble function of $T_F^\k$, cf.~\cite{Ainsworth-Oden:2000,KP}, extended by zero to the rest of $\k$.
\end{definition}
\begin{lemma}\label{edgebubble}
	Let $\el$ and $F\subset\partial\k$  a mesh face under Assumption \ref{A1}. Let $v\in \mathcal{P}_{p}(F)$ and denote by $v$ also the constant extension in $T_F^\k$ of $v$ in the direction normal to $F$.
	We have
	\begin{equation}\label{eq:L2_norm_face_bubble}
		\ltwo{v}{F}^2 \leq  C_{b,F} \ltwo{b_F^{1/2}v}{F}^2,
	\end{equation}
	with $C_{b,F}:=[2(p+2)]^{2(d-1)} \frac{(d-1)^{d}}{ (d-2)!} $, and
	\begin{equation}\label{eq:H1_norm_edge_bubble}
		\ltwo{\nabla (b_F v )}{T_F^\k}^2
		\leq  \einvT h_{\k} h_F^{-2}  	\ltwo{v}{F}^2,
	\end{equation}
	with $\einvT$ the constant of the inverse inequality \eqref{eq:inverse_h1} in the case of $d-1$-dimensional simplices. Moreover, we have
	\begin{equation}\label{eq:curl_edge_bubble}
	\ltwo{{\rm curl}\, (b_T v)}{T_F^\k}^2 \leq (d-1)\einvT h_{\k} h_F^{-2}  \ltwo{v}{F}^2.
	\end{equation}
\end{lemma}
\begin{proof}
The bound~\eqref{eq:L2_norm_face_bubble} is given in~\cite{KP} and~\eqref{eq:H1_norm_edge_bubble} follows immediately from~\eqref{eq:inverse_h1}, the fact that  $b_F\le 1$ and the fact that $\ltwo{v}{T_F^\k}^2 \leq h_{\k} \ltwo{v}{F}^2$. Finally, the bound~\eqref{eq:curl_edge_bubble} is a trivial consequence of ~\eqref{eq:H1_norm_edge_bubble} observing once again that $\norm{{\rm curl}\, v}{\k}^2 \leq  (d-1) 	\norm{\nabla v}{\k}^2$.
\end{proof}

\begin{theorem}[flux residuals lower bound]\label{thm:lower_bound_face}
	Let $u$ be the solution of \eqref{Problem} and let $u_h\in \fes$ be its dG approximation under Assumption \ref{A1},  \ref{A2} and~\ref{A3}. Then, for each $\k\in \mesh$, we have
	\begin{align}\label{eq:lower bound face}
		&\ltwo{{h}^{1/2}\jump{\diff\nabla u_h\cdot \mbf{n}}}{ \partial \k \cap  \Gamma_{\dint}}^2\\
		 &\qquad\le  6C_{b,F}\Big(
		 2C_{b,\k}
		\big(\Csh^2+\einv) (\alpha^*)^2 \ltwo{\sqrt{a}\nabla e}{\omega_\k}^2
		+\sum_{\k'\in\omega_\k}(1+2C_{b,\k} )O_{\k',E}^2 \nno\\
		&\qquad\qquad\qquad+\alpha^* \einvT \sum_{F\in\partial\k\cap\Gamma_\dint}
		 h_{F}^{-{1}}h_{F^\perp} \ltwo{\sqrt{a}\nabla e}{\omega_F}^2\Big),\nno
	\end{align}
	and
	\begin{equation}\label{eq:lower bound tang}
		\ltwo{h^{1/2}\jump{\tnabla u_h}}{\Gamma_{\dint}}^2\le
 		2C_{b,F}  (d-1)\alpha_*\einvT
		\sum_{F\in\partial\k\cap\Gamma_\dint}h_F^{-1} h_{F^\perp} \ltwo{\sqrt{a}\nabla e}{\omega_F}^2;
	\end{equation}
	here, $\omega_\k$ is the patch of elements neighbouring $\k$, $\omega_F=T_F^\k\cup T_F^{\k'}$ with $\k'$ the element neighbouring $\k$ across $F$ and $h_{F^\perp}:=\max \{h_\k, h_{\k'}\}$.
\end{theorem}
\begin{proof}
	In view of proving~\eqref{eq:lower bound face}, we first consider any $F\in\partial\k$ with $F\in\Gamma_\dint$.  Further, we fix $v\in H^1(\omega_F)$ as the constant extension of $\jump{\diff\nabla u_h\cdot \mbf{n}}|_F$ in the direction normal to $F$, so that $b_F v\in H^1_0(\omega_F)$, cf. Lemma~\ref{edgebubble}. Then, testing the error equation \eqref{eq:error equation} with $b_F v$ extended to zero on the whole of $\Omega$, we get
	\begin{align}
		\int_{\omega_F} \diff\nabla_h e\cdot \nabla (b_F v) \ud \uu{x}
		= &\int_{\omega_F}(f - \Pi f)(b_F v) \ud \uu{x} + \int_{\omega_F} {(\Pi f+\nabla \cdot (\diff\nabla u_h))}(b_F v)  \ud \uu{x}\nno\\
		&- \int_F \jump{\diff\nabla u_h\cdot \mbf{n}}^2 b_F  \ud{s}. \nno
	\end{align}
	From this, using~\eqref{eq:H1_norm_edge_bubble} and the fact that, $\ltwo{b_Fv}{\k}^2=\ltwo{b_Fv}{T_F^\k}^2 \leq h_{K} \ltwo{\jump{\diff\nabla u_h\cdot \mbf{n}}}{F}^2 $, the same bound being true on $\k'$, we obtain
	\begin{align}
		\ltwo{\jump{\diff\nabla u_h\cdot \mbf{n}}b_F^{1/2}}{F}^2
		\leq  & \sum_{\mathcal{K}\in\{\k,\k'\}}
		\left[(\ltwo{\Pi f -f}{T_F^\mathcal{K}}
		+\ltwo{(\Pi f+\nabla \cdot (\diff\nabla u_h))}{T_F^\mathcal{K}})\ltwo{b_Fv}{T_F^\mathcal{K}}
		\right.\nno \\
		&\left. \quad\quad\qquad+ \ltwo{\sqrt{a}\nabla e}{T_F^\mathcal{K}}\ltwo{\sqrt{a} \nabla (b_F v)}{T_F^\mathcal{K}} \right] \nno \\
		\leq &
		\sum_{\mathcal{K}\in\{\k,\k'\}}\left[
		h_{\mathcal{K}}^{1/2}(\ltwo{\Pi f -f}{T_F^\mathcal{K}}
		+\ltwo{(\Pi f+\nabla \cdot (\diff\nabla u_h))}{T_F^\mathcal{K}})
		\right.\nno\\
		&\quad\quad\left.
		+(\alpha^*\einvT h_F^{-1}  h_{\mathcal{K}} )^{1/2} h_F^{-1/2}\ltwo{\sqrt{a}\nabla e}{T_F^\mathcal{K}}\right]
		\ltwo{\jump{\diff\nabla u_h\cdot \mbf{n}}}{F}.\nno
		\nno
	\end{align}
	This, together with \eqref{eq:L2_norm_face_bubble},  gives
	\begin{align}
		\ltwo{h^{1/2}\jump{\diff\nabla u_h\cdot \mbf{n}}}{F}^2
		\leq &
		2C_{b,F}\!\!
		\sum_{\mathcal{K}\in\{\k,\k'\}}\left[
	 (\ltwo{h(\Pi f -f)}{T_F^\mathcal{K}}+\ltwo{h(\Pi f+\nabla \cdot (\diff\nabla u_h))}{T_F^\mathcal{K}})
		\right. \nno \\
		&\qquad\quad\quad\left.
		+ (\alpha^* \einvT )^{1/2} h_{F}^{-1}h_\mathcal{K} \ltwo{\sqrt{a}\nabla e}{T_F^\mathcal{K}}\right]^2.\nno
	\end{align}
Summing over all internal faces of $\k$, noting carefully that the involved domains do not overlap, we finally obtain
\begin{align}
\ltwo{{h}^{1/2}\jump{\diff\nabla u_h\cdot \mbf{n}}}{ \partial \k \cap  \Gamma_{\dint}}^2 \le & 6C_{b,F}\Big(\ltwo{h(\Pi f -f)}{\omega_\k}^2+\ltwo{h(\Pi f+\nabla \cdot (\diff\nabla u_h))}{\omega_\k}^2 \nno\\
&+\alpha^* \einvT \sum_{F\in\partial\k\cap\Gamma_\dint}
(h_{F}^{-{1}}h_{F^\perp})^2 \ltwo{\sqrt{a}\nabla e}{\omega_F}^2\Big),
\end{align}
as $h_{F^\perp}:=\max \{h_\k, h_{\k'}\}$, from which~\eqref{eq:lower bound face} now follows by employing~\eqref{eq:lower bound elem}.

The proof concerning the tangential jump residual is similar.
Given $F\in\Gamma_\dint$, we fix $v\in H^1(\omega_F)$ as the constant extension of $\jump{\tnabla u_h}|_{F}$ in the direction normal to $F$, so that $b_F v\in H^1_0(\omega_F)$.
Using the fact that $\text{curl}\, \nabla u =\mbf{0}$, we have the key observation
\begin{equation}\label{eq: curl grad 0}
\int_{\omega_F} \text{curl}\, (b_F v) \cdot \nabla u \ud \uu{x} = 0.
\end{equation}
Integration by parts and~\eqref{eq:curl_edge_bubble}, give
\begin{align}\label{eq:efficiency3}
	\ltwo{\jump{\tnabla u_h}b_F^{1/2}}{F}^2
	=& \int_{\omega_F} \text{curl}\, (b_F v) \cdot \nabla u_h \ud \uu{x}
	= \int_{\omega_F} \text{curl}\, (b_F v) \cdot \nabla (u_h-u) \ud \uu{x} \nno \\
	&\leq \alpha_*\sum_{\mathcal{K}=\k,\k'} \ltwo{\sqrt{a}\nabla e}{T_F^\mathcal{K}}\ltwo{ \text{curl}\,  (b_T v)}{T_F^\mathcal{K}} \nno\\
	&\leq \alpha_*\big((d-1)\einvT h_{\mathcal{K}} \big)^{1/2} h_F^{-1}\ltwo{\jump{\tnabla u_h}}{F}\sum_{\mathcal{K}\in\{\k,\k'\}} \ltwo{\sqrt{a}\nabla e}{T_F^\mathcal{K}}.\nno
\end{align}
Hence, by using~\eqref{eq:L2_norm_face_bubble}, we obtain
$$
	\ltwo{h^{1/2}\jump{\tnabla u_h}}{F}^2
	\leq 2 \alpha_*(d-1)\einvT C_{b,F} (h_F^{-1}h_{F^\perp})^{2} \ltwo{\sqrt{a}\nabla e}{\omega_F}^2.
$$
The required lower bound on the jump of the tangential gradient now follows by summing over all internal faces belonging to $\k$.
\end{proof}

By construction, we have $|\omega_F|\sim h_F^d$ for the patches $\omega_F$. The terms involving norms over  $\omega_F$ in Theorem~\ref{thm:lower_bound_face}  reflect and account for the presence of relatively small faces.
Indeed, linking the size of $\omega_F$ to that of $F$, instead of that of the element $\k$, allows very large ratios $h_\k/h_F$.

If, on the other hand,
the size of each of the element's face is comparable to that of the element itself,  then we may modify the construction of the face bubble of Definition~\ref{def:edge_bubble} by moving the opposite vertex all the way to the  centre of star-shapedness of $\k$.  In such a case,
$h_F^{-1} h_{F^\perp} \ltwo{\sqrt{a}\nabla e}{\omega_F}\lesssim  \ltwo{\sqrt{a}\nabla e}{\k\cup\k'}$.
Thus, for meshes with potentially many but regular hanging nodes, the new flux-residuals' lower bounds revert to the classical ones, as encapsulated in the following corollary.
\begin{corollary}\label{cor:lower_regular}
Under the assumption of Theorem~\ref{thm:lower_bound_face}, if, moreover, for every $\el$ and $F\in\partial\k$  it holds $h_F\ge\Csh^{-1}h_K$, then
	\begin{align}\label{eq:lower bound face reg}
		&\ltwo{{h}^{1/2}\jump{\diff\nabla u_h\cdot \mbf{n}}}{ \partial \k \cap  \Gamma_{\dint}}^2\\
		 &\qquad\le  6C_{b,F}\Big(
		 (2C_{b,\k}
		\big(\Csh^2+\einv) (\alpha^*)^2+ \Csh\alpha^*\einv
		)
		\ltwo{\sqrt{a}\nabla e}{\omega_\k}^2
		 \nno\\
		&\qquad\qquad\qquad+\sum_{\k'\in\omega_\k}(1+2C_{b,\k} )O_{\k',E}^2\Big),\nno
	\end{align}
	and
	\begin{equation}\label{eq:lower bound tang reg}
		\ltwo{h^{1/2}\jump{\tnabla u_h}}{\Gamma_{\dint}}^2\le
 		2C_{b,F}  (d-1)\Csh\alpha_*\einv
		\ltwo{\sqrt{a}\nabla e}{\omega_\k}^2.
	\end{equation}
\end{corollary}

\begin{remark}
Corollary~\ref{cor:lower_regular} holds in the setting of fully shape-regular meshes allowing for the sub-mesh auxiliary mesh construction, cf. Section~\ref{submesh}. Hence, Corollary~\ref{cor:lower_regular} together with Theorem~\ref{thm:lower_bound_element} and Theorem~\ref{thm:upperbound} establishes the  reliability and efficiency of the classical residual error estimator for general fully shape-regular polytopic meshes with multiple hanging nodes. The analysis, also in this case, differs  from the classical one in that the finite element space used to control the non-conforming error, being based on the auxiliary mesh, is \emph{not} a subspace of the discrete solution space $\fes$. Moreover, the resulting error bound is as explicit as the classical bound because, in this case, the auxiliary mesh quality is fully controlled by that of the polytopic mesh. The element bubble construction is also new and accounts for the polytopic nature of the mesh.

On the other hand, controlling the flux residuals in the extreme case of possibly unbounded non shape-regular interfaces requires further new ideas. Whenever it is possible to construct a face bubble function $b_F$ such that the following bound
$$
	\ltwo{\nabla (b_F v )}{T_F^\k}^2
		\leq  \einvT h_K^{-1} 	\ltwo{v}{F}^2,
$$
holds true,  the lower bounds  of the flux residuals \eqref{eq:lower bound face} and \eqref{eq:lower bound tang} will be independent of $h_F^{-1} h_{F^\perp}$.  An alternative approach could be to consider bubble functions constructed on a neighbouring set of structured elements. Then, the bubble functions $b_F$ will be independent of the individual  face size and the number of elements. For instance, this is  the approach used in~\cite{KoptevaLower} to derive a lower bound of the flux residual of the FEM employing structured anisotropic triangular meshes.  However, the construction of such face bubble functions  for the general-shaped polytopic meshes considered in this work is highly non-trivial.
\end{remark}

\section{Numerical experiments}\label{num_ex}
We present two numerical examples testing the new \emph{a posteriori} error estimator. With the first example we test the impact of polygonal elements with a large number of small faces on the effectivity index. With the second, we test the performance of the estimator within a non-standard adaptive algorithm. In all cases, we set $C_\sigma  = 10$.
\subsection{Example 1}
We construct a sequence of  polygonal meshes containing $114$, $498$, $2063$, $8912$, and $32768$ elements obtained by successive agglomeration of a very fine triangular background mesh made of $10^6$ elements. Each of the polygonal elements contains at least $50$ edges, see Figure \ref{crazy_mesh} for an illustration.
\begin{figure}[!tb]
	\begin{center}
		\begin{tabular}{cc}
			\includegraphics[width=0.41\linewidth]{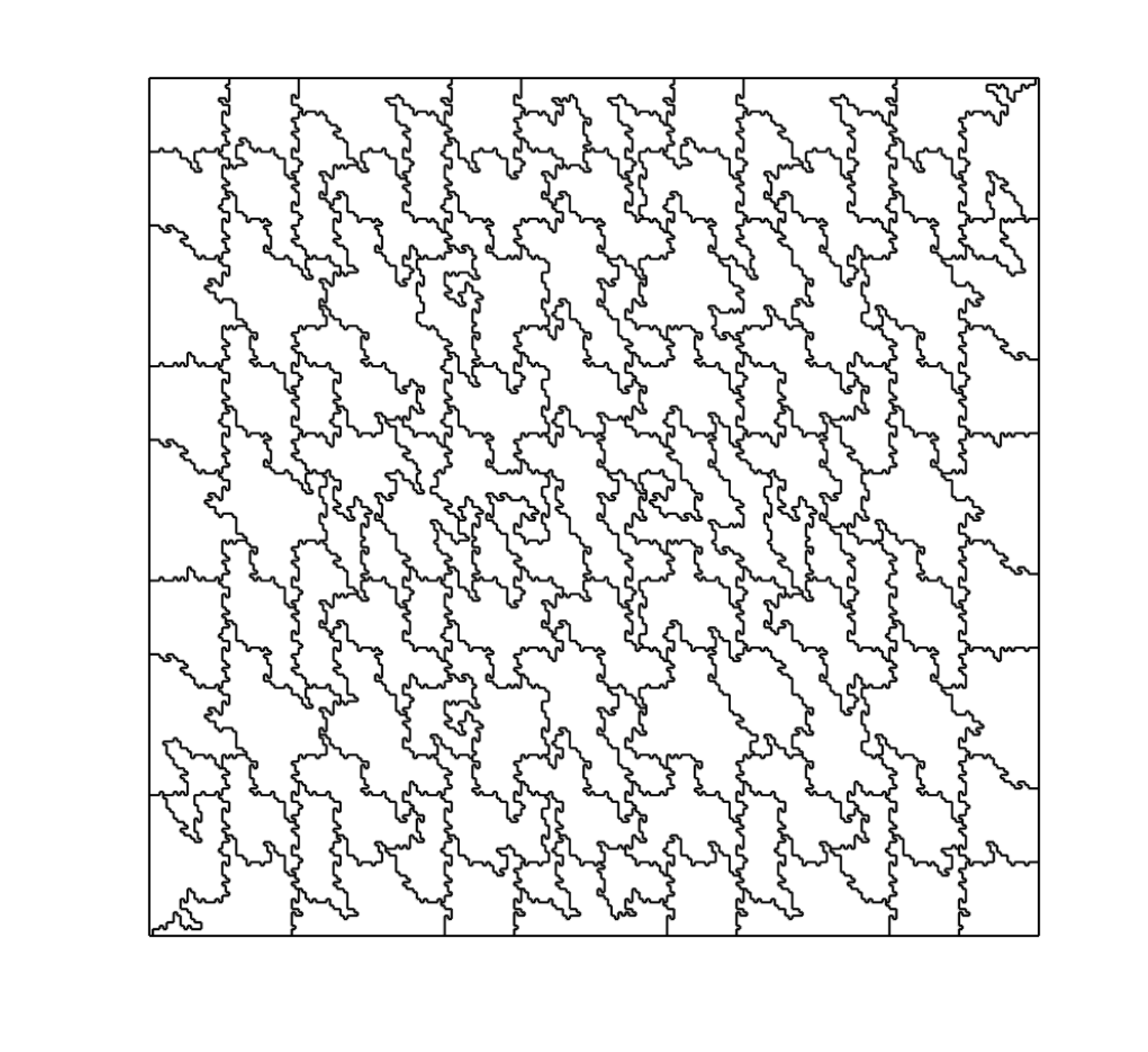} &
			\includegraphics[width=0.4\linewidth]{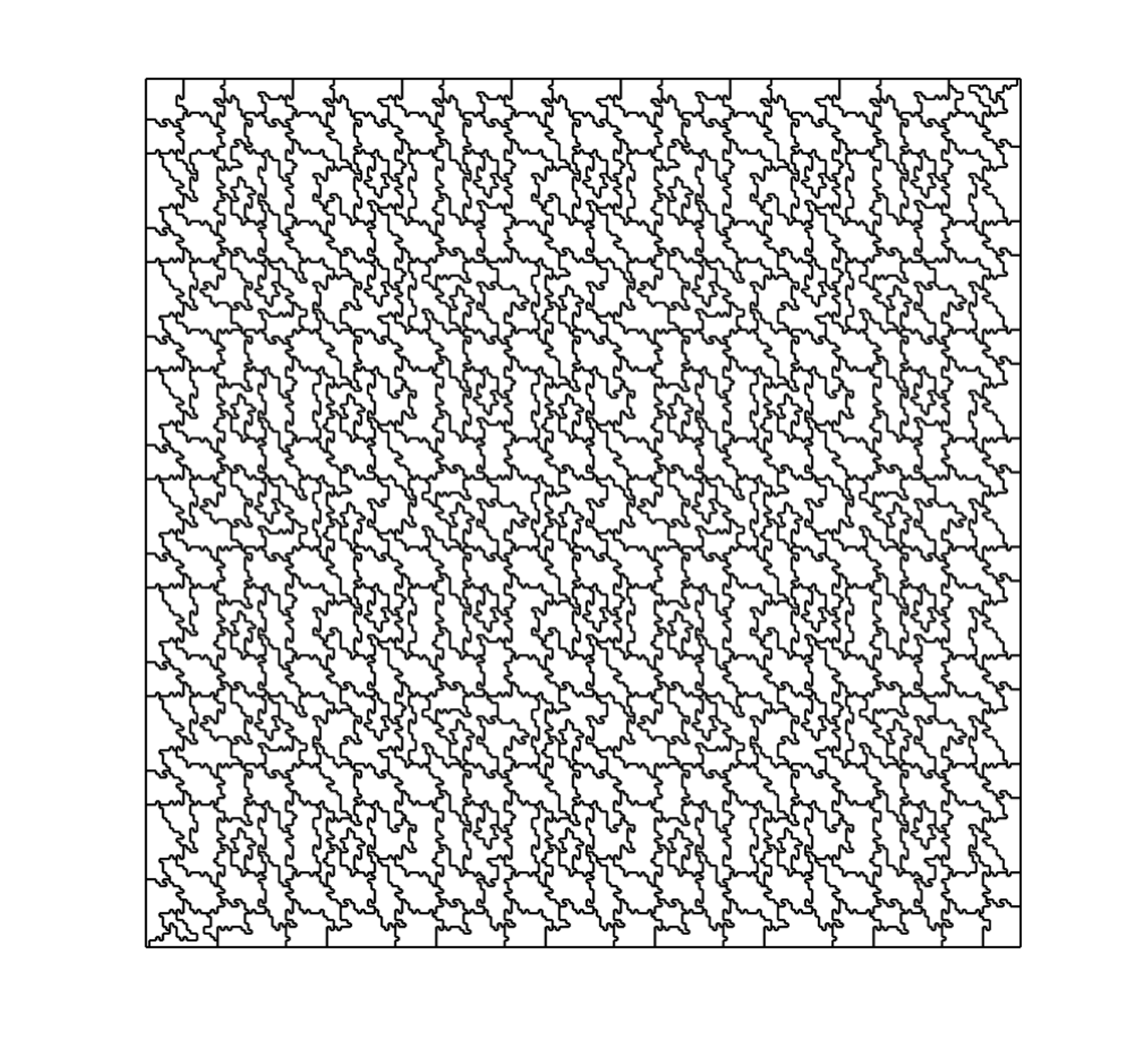}
		\end{tabular}
	\end{center}
	\caption{Example 1. Sample meshes with 114 (left) and 498 (right) polygonal elements obtained by agglomeration of a fine triangular mesh made of one million elements.}\label{crazy_mesh}
\end{figure}

We consider the problem \eqref{Problem}, with   $\diff = I_{2\times 2}$ and  $\Omega := (-1,1)^2$. The Dirichlet boundary conditions and the source term $f$ are determined by the exact solution $ u = \sin(\pi x)\sin(\pi y)$.
Numerical results for $p=1,2,3,4$ are displayed in Figure~\ref{Ex:error_uniform_refinement}. The observed convergence rate of both the error and the estimator is $\mathcal{O}(DoFs^{-\frac{p}{2}})$, i.e., optimal in terms of the total number of degrees of freedom $DoFs$. Moreover, the effectivity index is  bounded between 1.2 and 2.6, hence showing that  efficiency is not affected by the complexity of the element shapes. This numerical observation reflects that Assumption \ref{A3} may not be necessary.
\begin{figure}[!tb]
	\begin{center}
		\begin{tabular}{cc}
			\includegraphics[width=0.4\linewidth]{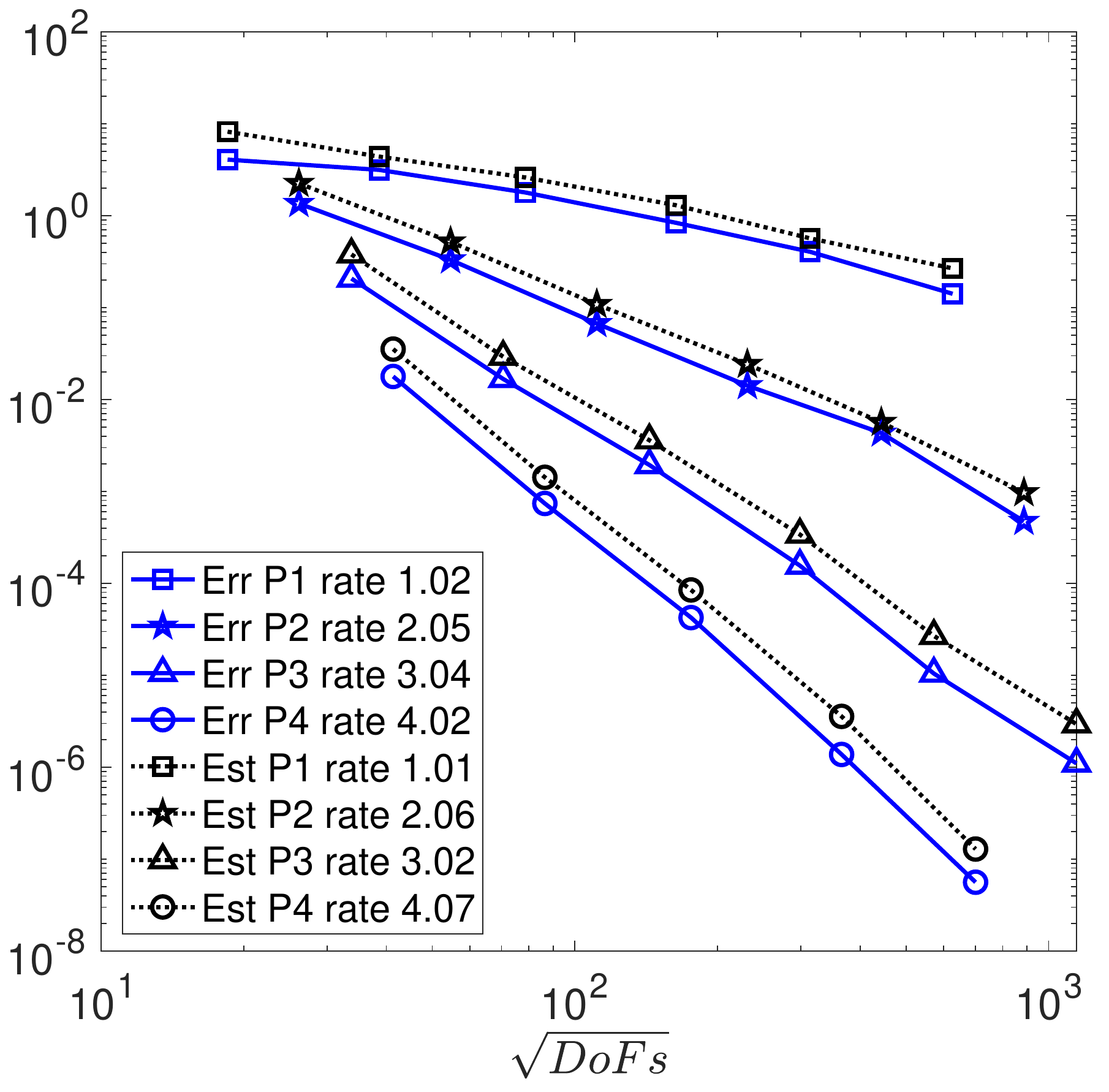} &
			\includegraphics[width=0.4\linewidth]{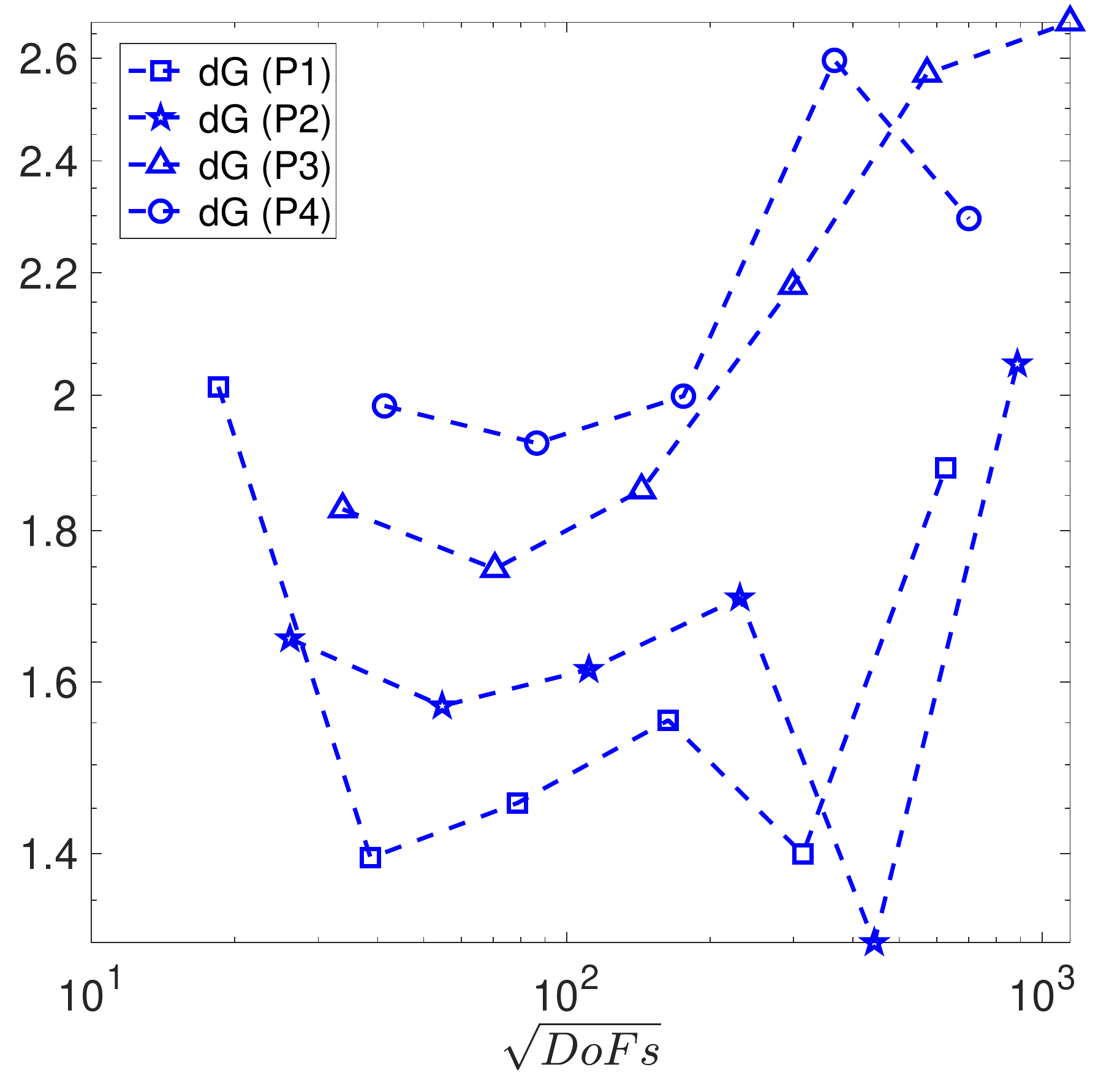}
		\end{tabular}
	\end{center}
	\caption{Example 1. Convergence history of the error and error estimator of Equation~\ref{full bound} for $p=1,2,3,4$ (left) and respective effectivity measured as the ratio estimator over error (right).}\label{Ex:error_uniform_refinement}
\end{figure}

Next, we compare the percentage contribution of the different components to the total estimator. Setting $R_X:=(\sum_{K\in\mathcal{T}}R_{K,X}^2)^{1/2}$ for $X\in\{E,N,J,T\}$, in Table \ref{ex1:table Comparison} we provide the percentage of total element residual $R_{E}$, total jump residual $R_{J}$, total jump of the normal flux residual $R_{N}$, and total jump of the tangential flux residual $R_{T}$ for $p=1,2$. For the coarse meshes of Figure \ref{Ex:error_uniform_refinement}, the element residual dominates the total estimator, followed by $R_{J}$. For finer meshes, we observe significant contribution by $R_{E}$, $R_{J}$,  and combined $R_{T}$ and $R_{N}$.
\begin{table}[!htb]
	\centering
		\begin{tabular}{|c|c|c|c|c||c|c|c|c|}
			\hline
&
			\multicolumn{4}{|c||}{$p=1$}
			&
			\multicolumn{4}{|c|}{$p=2$}  \\
			\hline
			\hline
			$\#$ elem  & $R_{E}$ & $R_{J}$ &   $R_{N}$&$R_{T}$ & $R_{E}$ & $R_{J}$ &   $R_{N}$&$R_{T}$   \\
			\hline
			$114$ & $64\%$  & $18\%$ &   $9\%$ &  $9\%$ & $56\%$  & $26\%$ &   $9\%$ &  $9\%$  \\
			\hline
			$498$ & $59\%$  & $23\%$ & $9\%$  & $9\%$ &  $45\%$  & $36\%$ & $9\%$  & $10\%$ \\
			\hline
			$2063$ & $36\%$  & $37\%$ & $13\%$  & $14\%$  & $45\%$  & $33\%$ & $10\%$  & $12\%$   \\
			\hline
			$8912$ & $32\%$ & $37\%$ &  $14\%$ & $17\%$ &$43\%$ & $31\%$ &  $11\%$ & $15\%$  \\
			\hline
			\hline
		\end{tabular}
	\caption{Example 1. Estimator's components percentage contributions to the dG error for $p=1,2$: element residual $R_{E}$, jump residual $R_{J}$,  jump of the normal flux residual $R_{N}$, and jump of the tangential flux residual $R_{T}$.}
	\label{ex1:table Comparison}
\end{table}
Further, to highlight the importance of the presence of $R_T$ in the new estimator presented in this work, we compare it to the \emph{a posteriori} error estimator that can be derived using standard techniques from \cite{KP}, which does not contain the  tangential flux jump $R_{T}$. As already mentioned in Remark \ref{KP_noKP}, it is immediate to bound $R_{K,T}$ from $R_{K,J}$ through an inverse estimate on each face $F$, giving $R_{\k,T} \lesssim \frac{h_K}{\rho_F^{}}R_{K,J}$, with $\rho_F$ denoting the inscribed radius of the face $F$. Clearly, for small faces, we have $\rho_F\le h_F\ll h_K$, showcasing that $R_{K,T}$ is theoretically sharper than $R_{J,T}$. In Figure \ref{Ex: comparison to the classical estimator}, we present the error, the estimator from \eqref{full bound} and the \emph{classic} estimator whereby $R_{\k,T}$ is replaced by $\frac{h_K}{\rho_F^{}}R_{K,J}$, for $p=1,2$. The superiority of the estimator presented in this work is evident for coarse meshes with large ratio $\frac{h_K}{h_F}$.  We note that the jump terms account for more than 80\% of the classical error estimator, thus indicating that the term $R_{J,T}$ is indeed responsible for the relative over-estimation of the error. This confirms the theoretical intuition and showcases the practicality of the estimator proven in this work.
\begin{figure}[!htb]
	\begin{center}
		\begin{tabular}{cc}
			\includegraphics[width=0.41\linewidth]{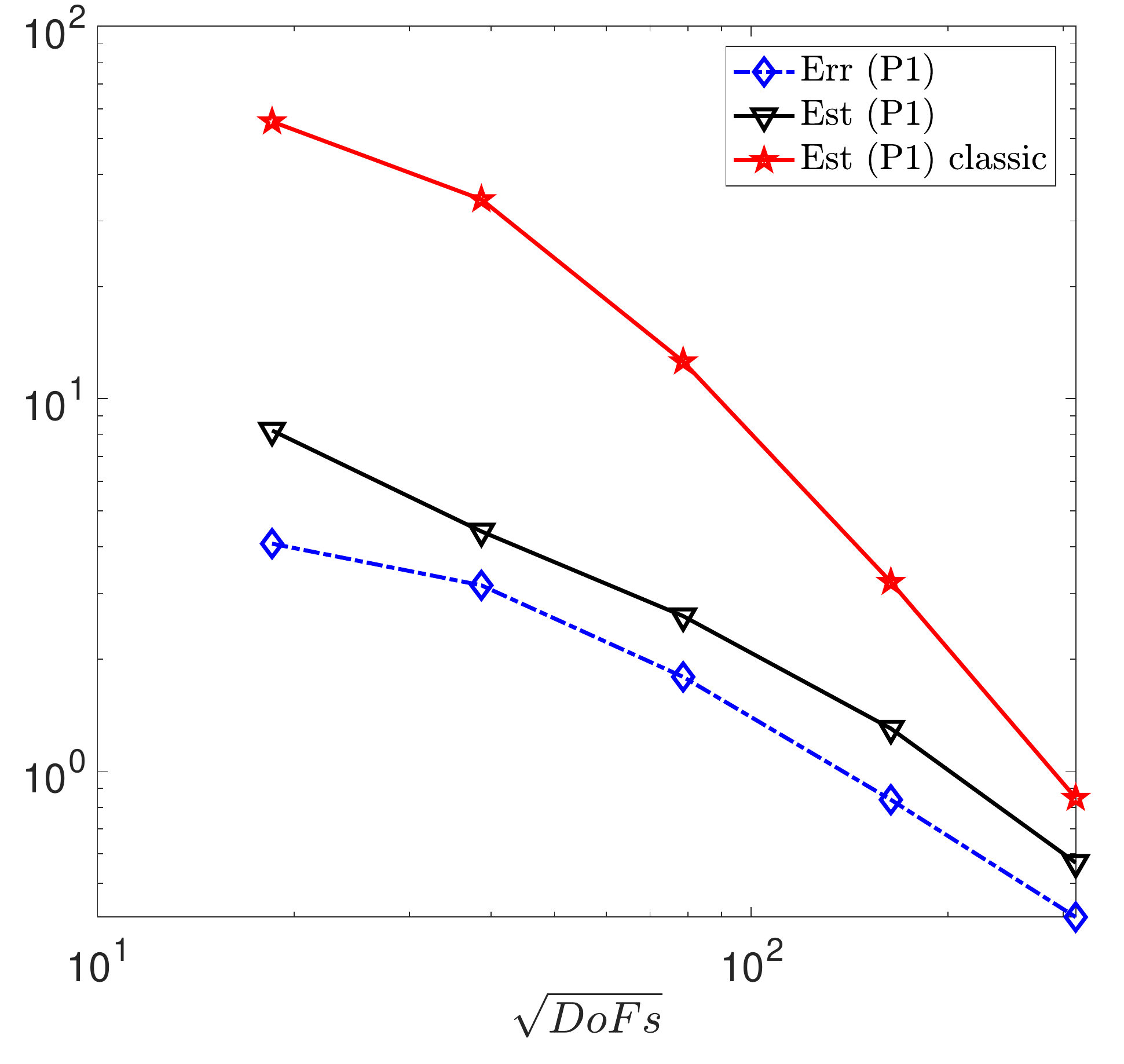} &
			\includegraphics[width=0.41\linewidth]{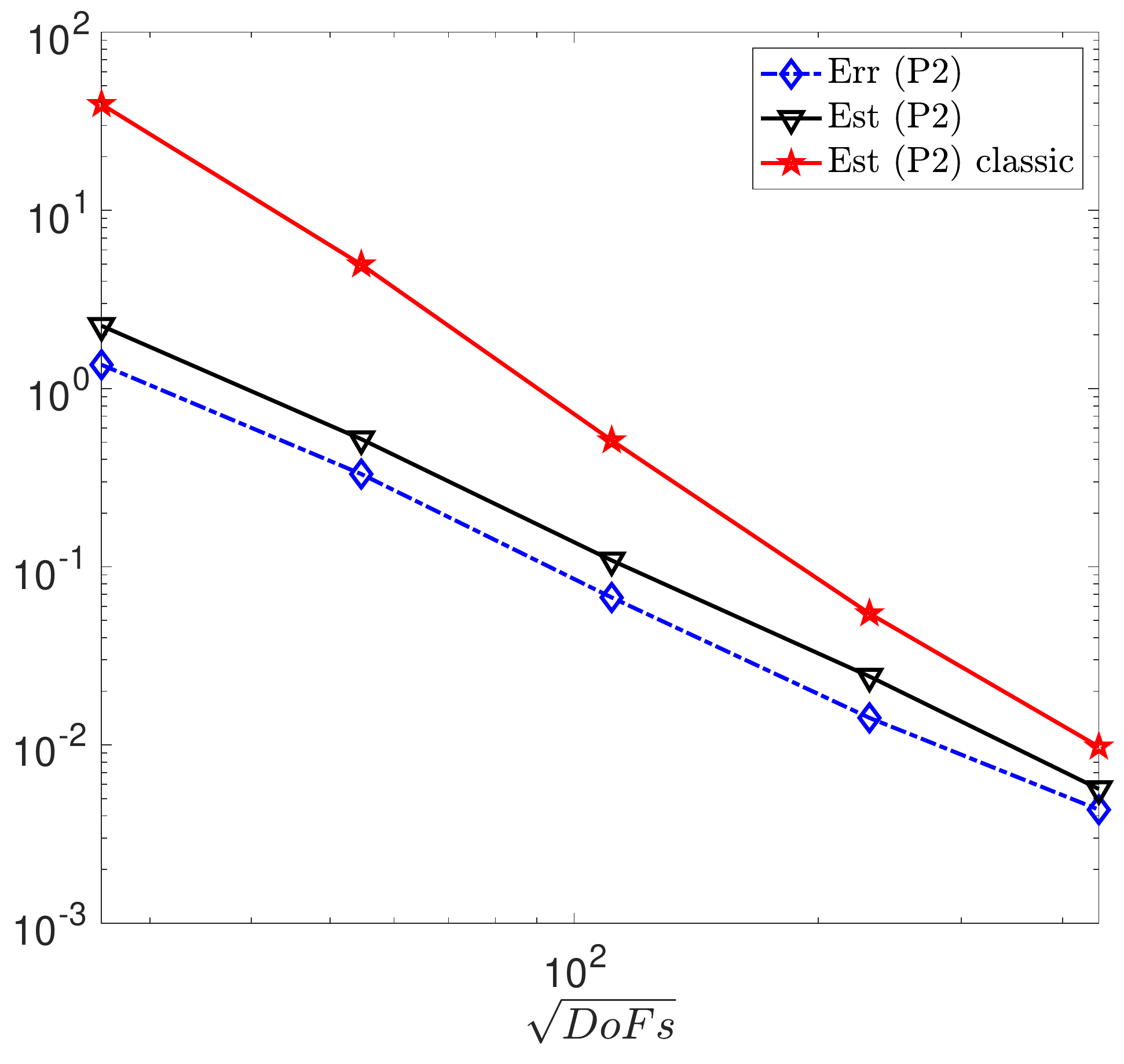}
		\end{tabular}
	\end{center}
	\caption{Example 1. Convergence history of the error, error estimator of Equation~\ref{full bound} and the classical estimator for $p=1,2$.}\label{Ex: comparison to the classical estimator}
\end{figure}

\subsection{Example 2}
We test a new adaptive algorithm driven by the error estimator from Section~\ref{apost_sec}. Starting from a relatively coarse simplicial mesh, we use the estimator~\eqref{full bound} to mark simplicial elements for refinement through a bulk-chasing criterion (also known as D\"orfler marking), and also mark \emph{pairs} of elements for agglomeration based on the size of the jump residual terms on elemental interfaces. Refinement of simplicial elements is performed via a newest vertex bisection algorithm. In the agglomeration step, general, polygonal meshes will be generated. In successive iterations, polygonal meshes which are marked for refinement are subdivided into either a finer polygonal mesh or a simplicial mesh, depending on their level of agglomeration. For simplicity, we do not consider the data oscillation terms. The adaptive algorithm can thus be described as:
\[
\text{SOLVE} \longrightarrow \text{ESTIMATE} \longrightarrow  \text{MARK} \longrightarrow \text{REFINE/AGGLOMERATE}.
\]

We consider the problem~\eqref{Problem} with   $\diff = I_{2\times 2}$ on  $\Omega := (-1,1)^2 \setminus (0,1)\times(-1,0)$. The Dirichlet boundary conditions and the source term $f$ are determined by the exact solution
\begin{equation*}
\begin{aligned}
u = &r^{2/3} \sin(2\psi/3) + \exp(-1000((x-0.5)^2+(y-0.25)^2 )) \\
&+ \exp(-1000((x-0.5)^2+(y-0.75)^2 )),
\end{aligned}
\end{equation*}
which has a point singularity at  the origin. We  test the adaptive dG algorithm described above with $p=1,2,3$, with D\"orfler's marking strategy $25\%$ for refinement and the maximum marking strategy $5\%$ for agglomeration. We point out that the agglomeration step is driven by the jump terms $R_{F,N}$, $R_{F,J}$, and $R_{F,T}$ for all faces $F \subseteq \partial\k_i\cap\partial\k_j$ on the meshes interface between element $\k_i$ and $\k_j$.

The performance of the proposed adaptive algorithm is showcased in Figure \ref{Ex:error}. The convergence rates of both error and estimator are optimal in terms of the total number of degrees of freedom ($DoFs$), which is $\mathcal{O}(DoFs^{-\frac{p}{2}})$. Further, on coarse mesh levels, the mesh agglomeration dominates the mesh refinement. Consequently, the number of $DoFs$ is initially reduced by the adaptive algorithm while the error is also reduced. Another important observation is that the coarse mesh level's effectivity index still seems quite reasonable, namely 2 to 3 times greater than the asymptotic value. This is in spite of the presence of  some very large polygonal elements next to small shape-regular triangles. These can be seen, for example, in the sequence of meshes produced by the adaptive algorithm  with $p=3$ shown in Figure~\ref{Ex:mesh}.

Finally, we perform a comparison between the adaptive algorithm with and without agglomeration for $p=2,3$, using the same marking parameter and stopping criterion. The results are presented in Figure \ref{Ex:compare_PR_AR}. Clearly, agglomeration helps reducing DoFs almost without influencing the accuracy on coarse meshes. As the meshes are refined, the advantage of polygonal elements is gradually reduced, as expected.
\begin{figure}[!tb]
	\begin{center}
		\begin{tabular}{ccc}
			\includegraphics[width=0.3\linewidth]{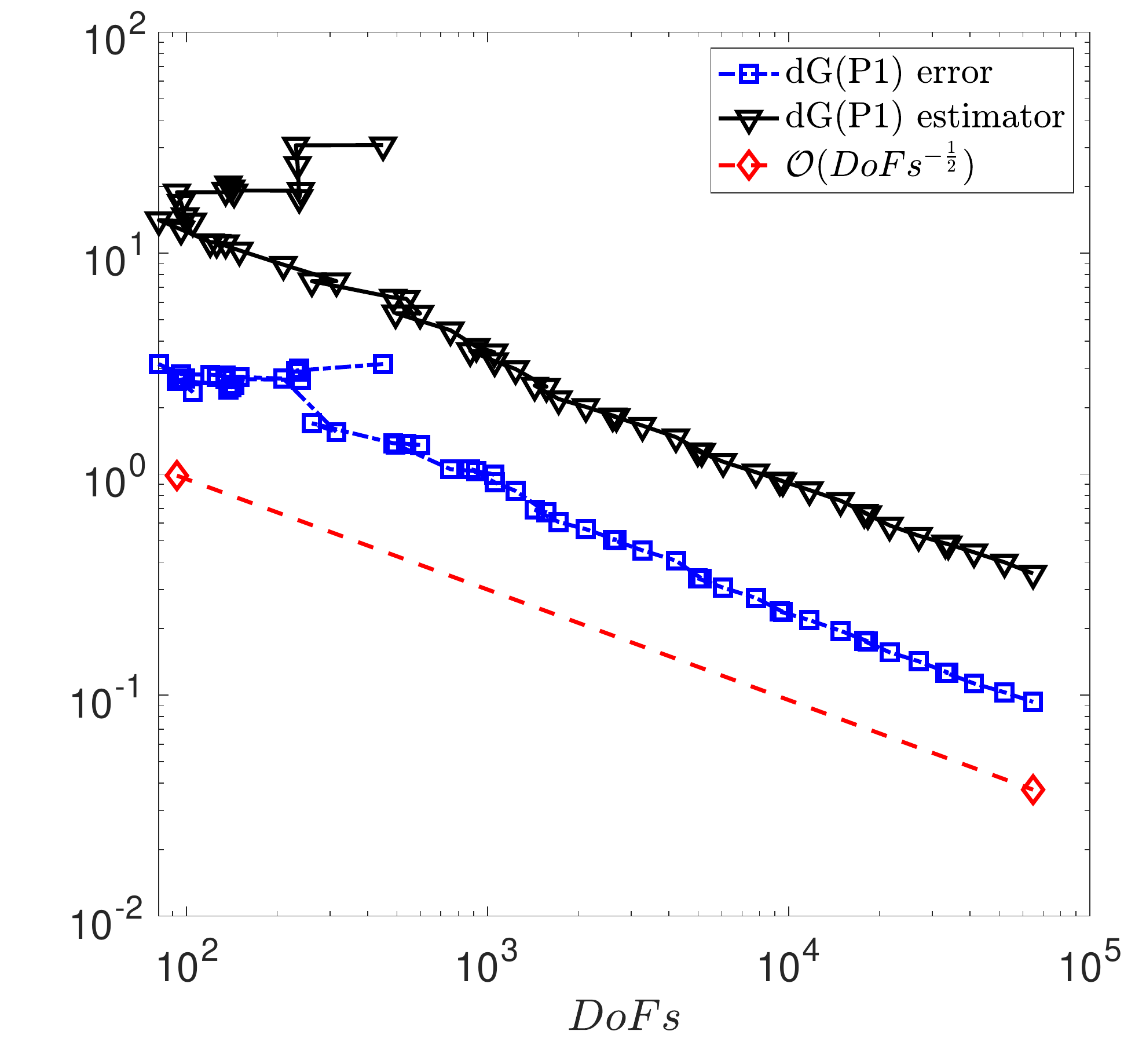} &
			\includegraphics[width=0.3\linewidth]{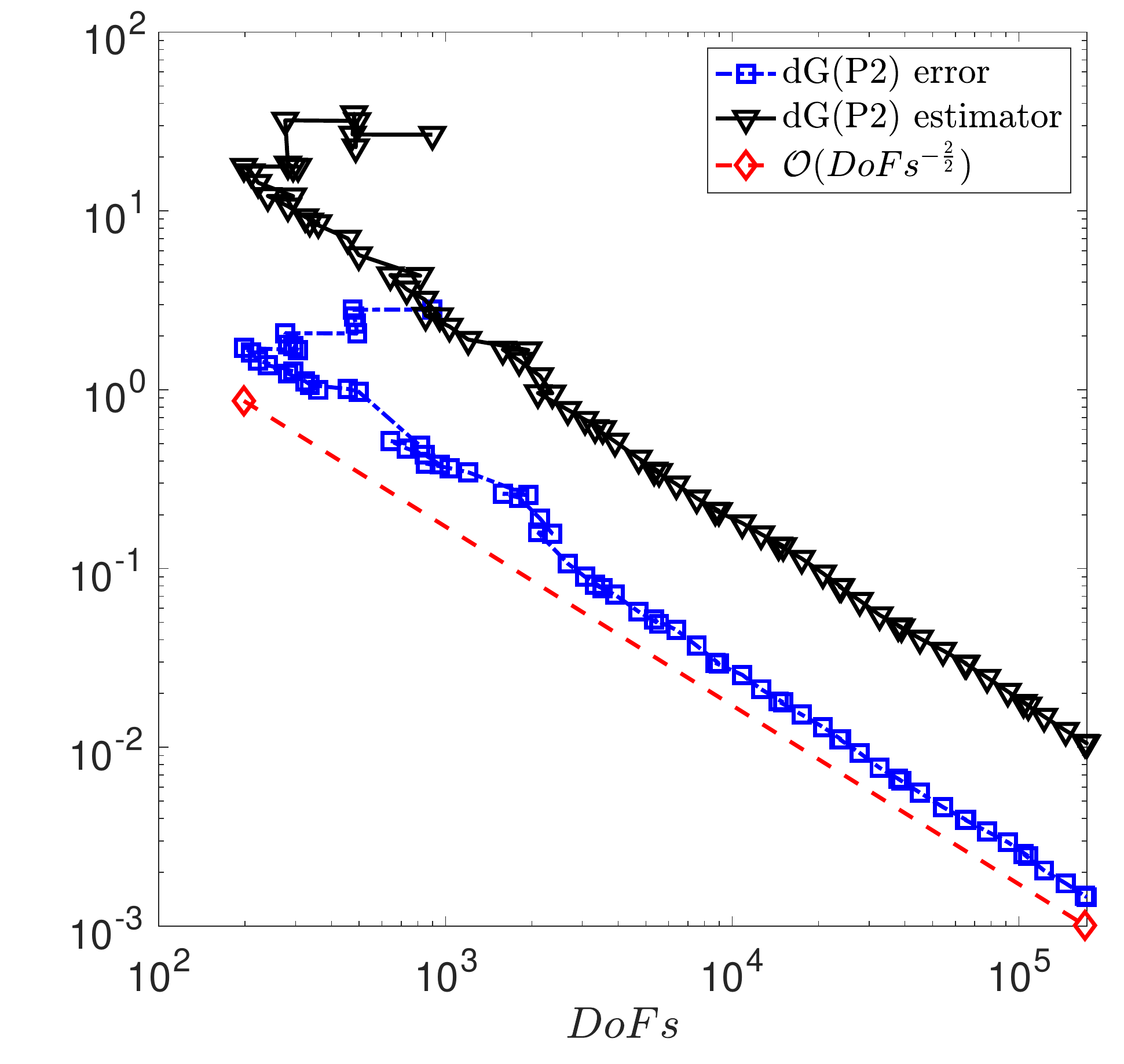} &
			\includegraphics[width=0.3\linewidth]{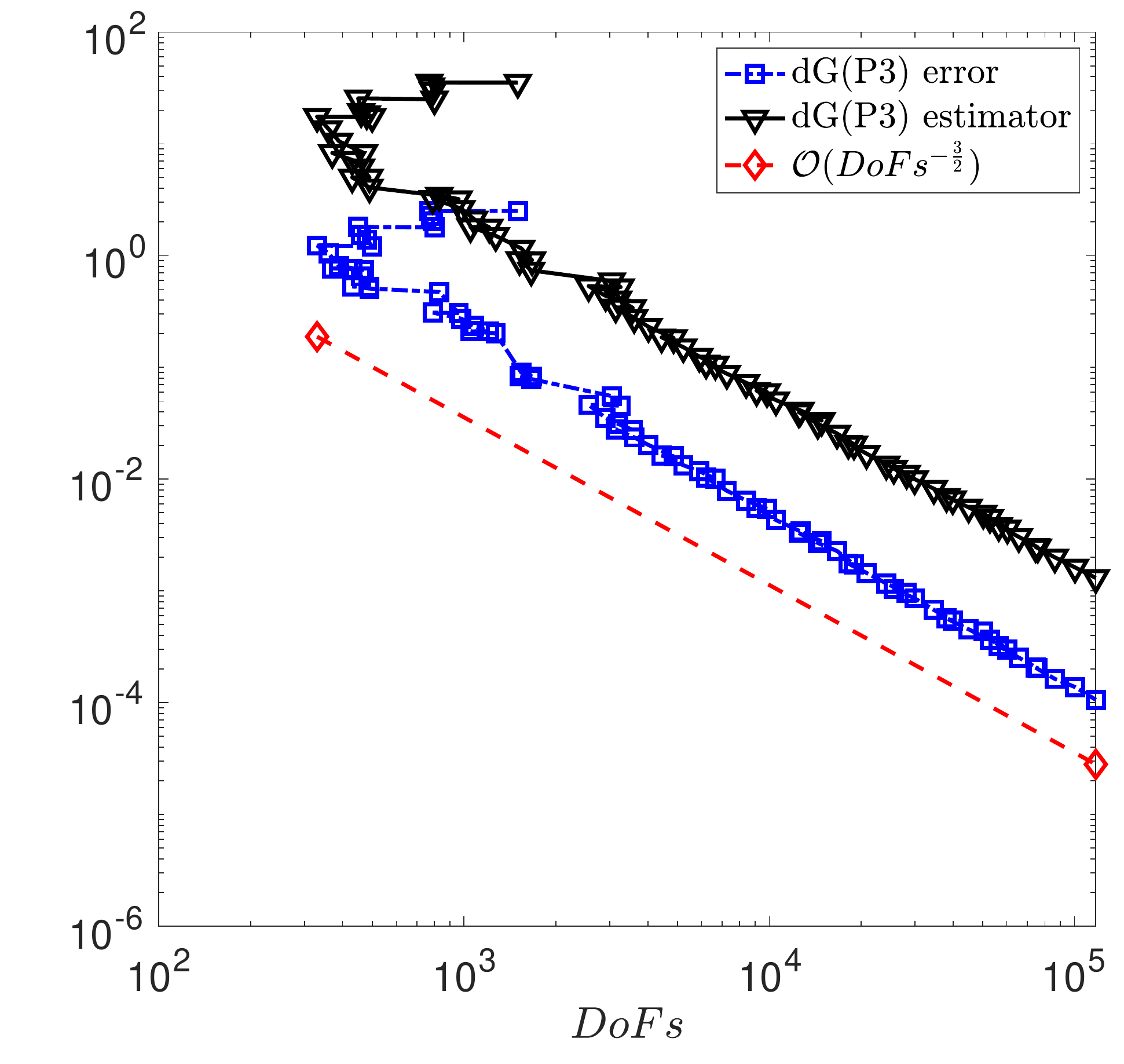} \\
			\includegraphics[width=0.3\linewidth]{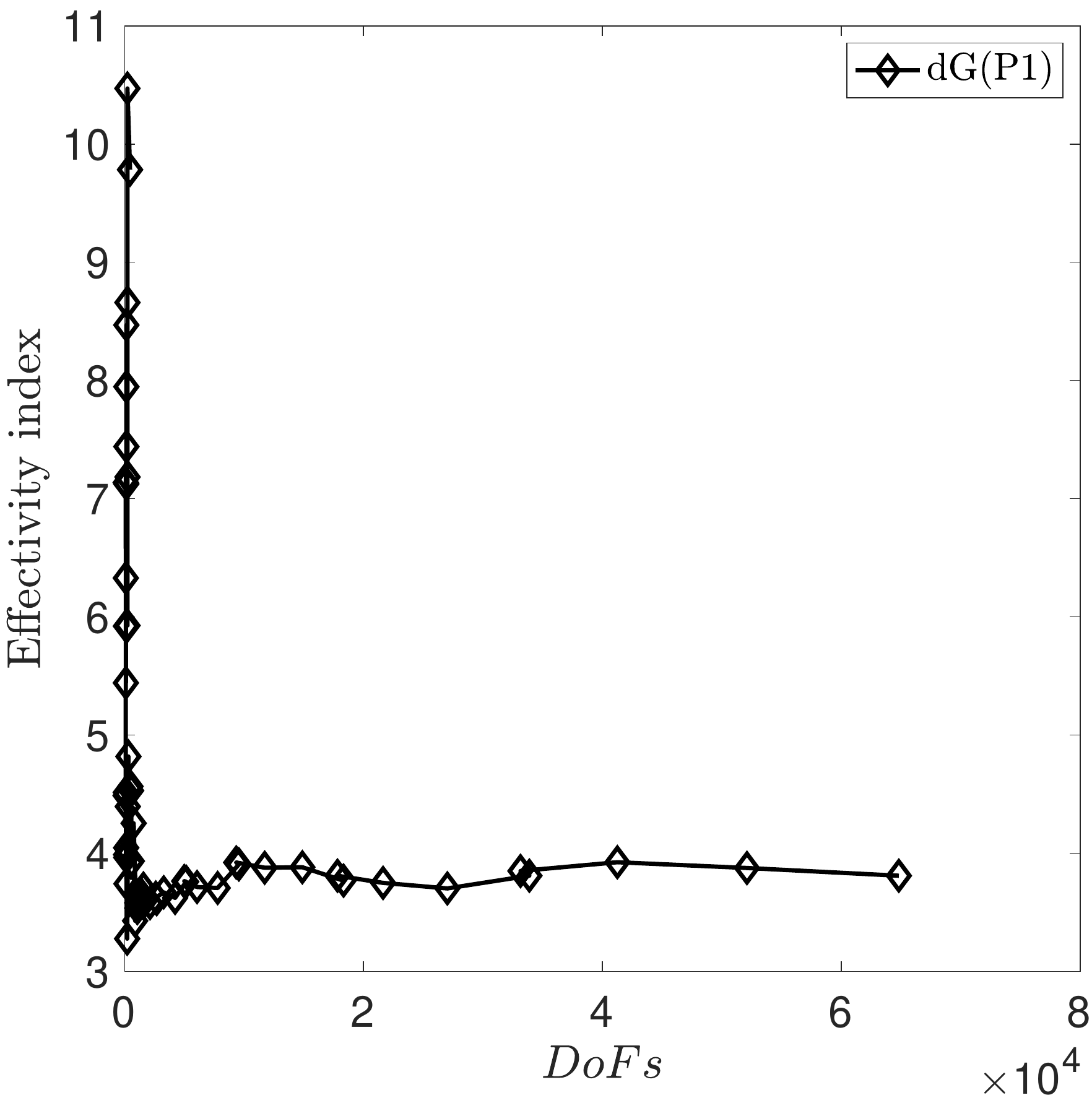} &
			\includegraphics[width=0.3\linewidth]{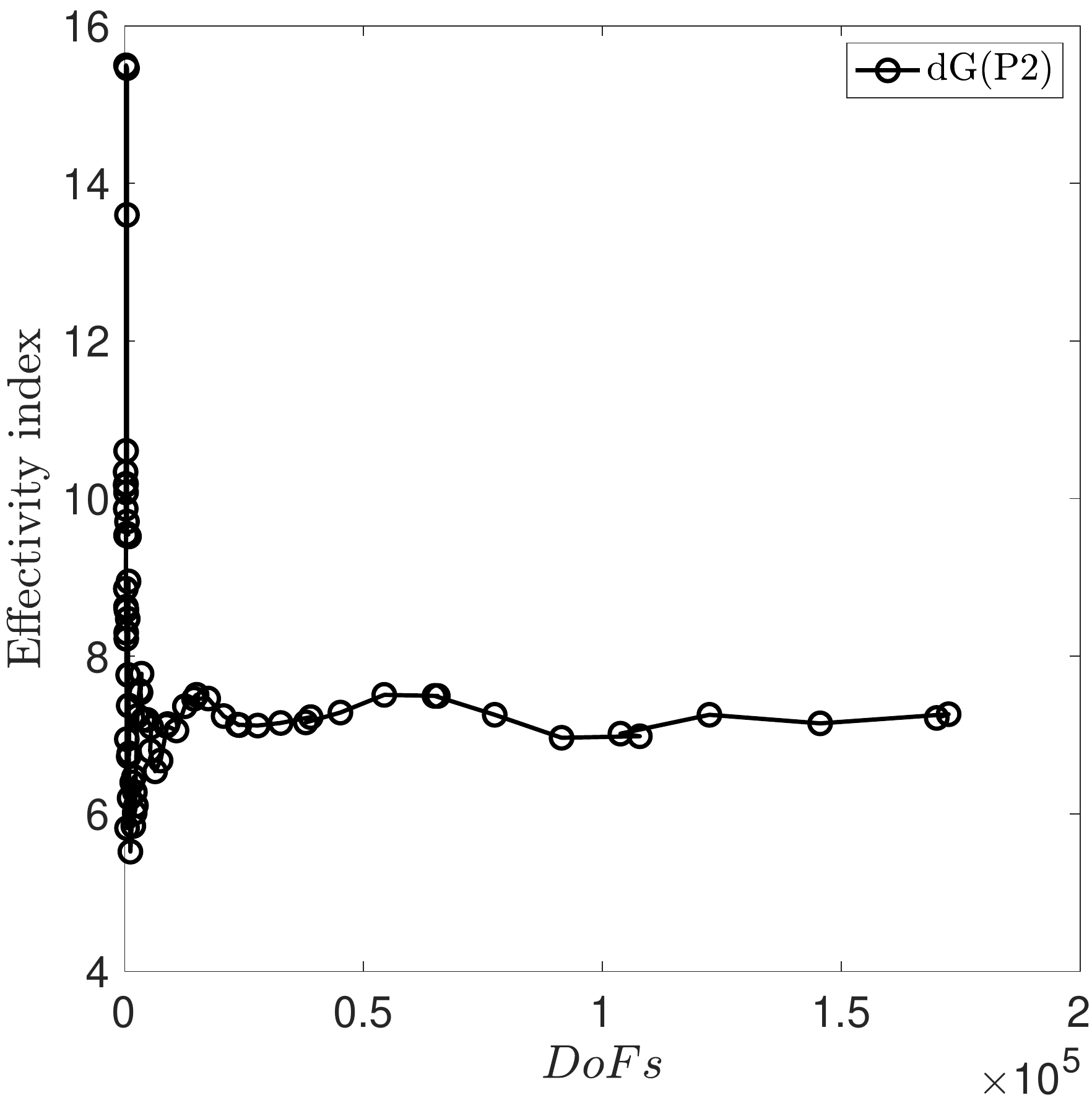} &
			\includegraphics[width=0.3\linewidth]{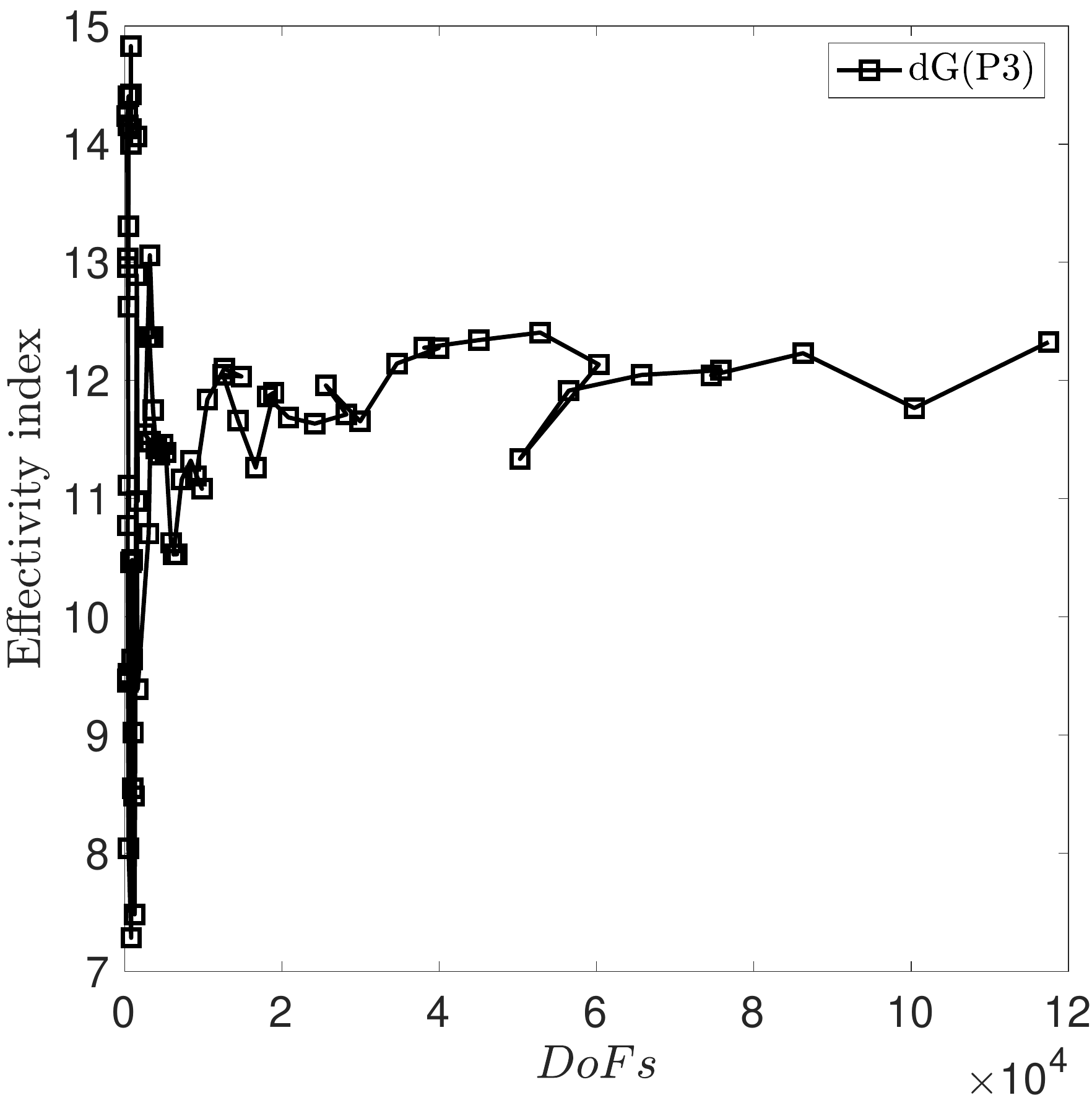}
		\end{tabular}
	\end{center}
	\caption{Error, estimator (top) and effectivity index (bottom) of the new adaptive polytopic dG method with $p=1,2,3$.}\label{Ex:error}
\end{figure}

\begin{figure}[!tb]
	\begin{center}
		\begin{tabular}{ccc}
			\includegraphics[width=0.3\linewidth]{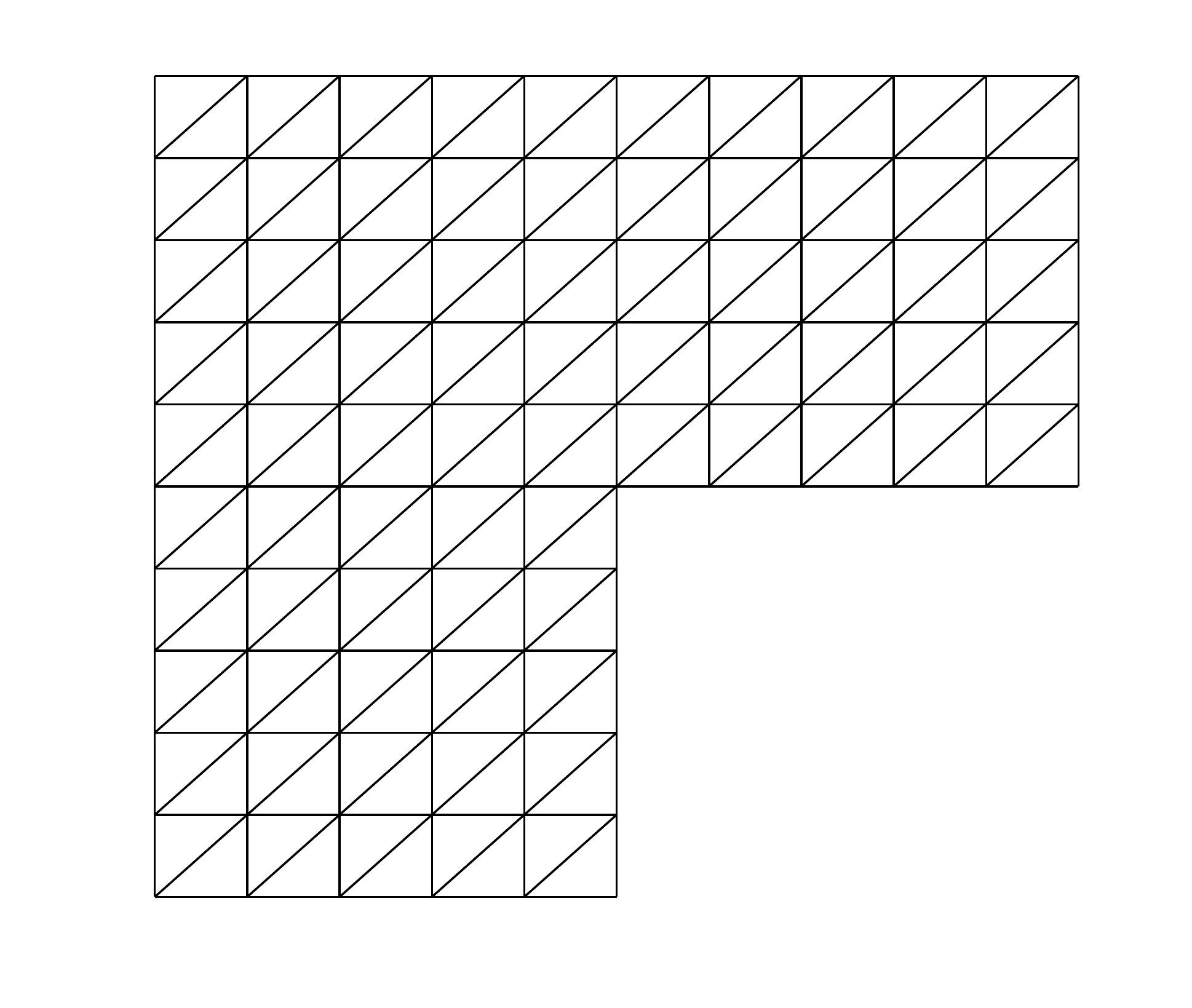} &
			\includegraphics[width=0.3\linewidth]{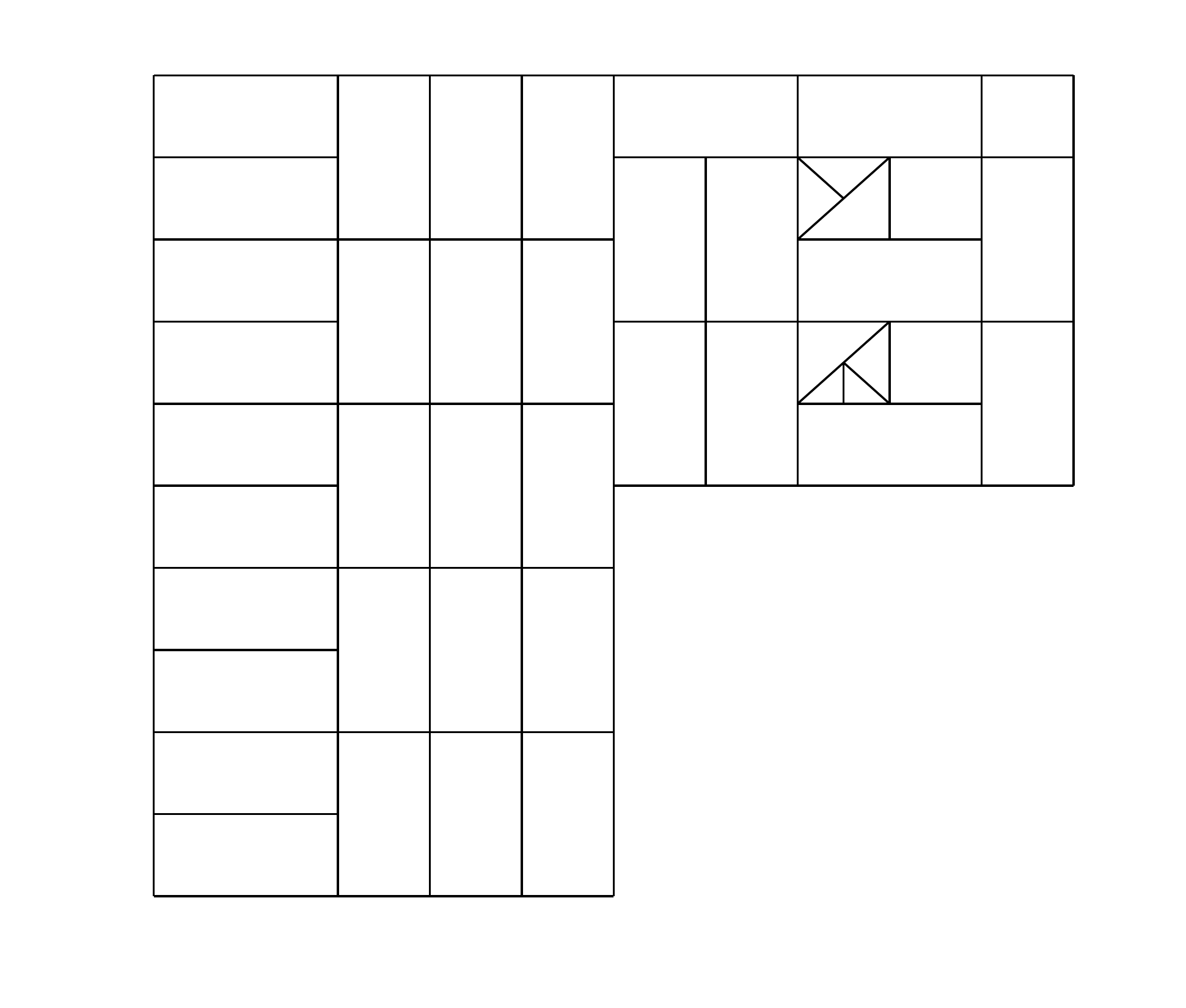} &
			\includegraphics[width=0.3\linewidth]{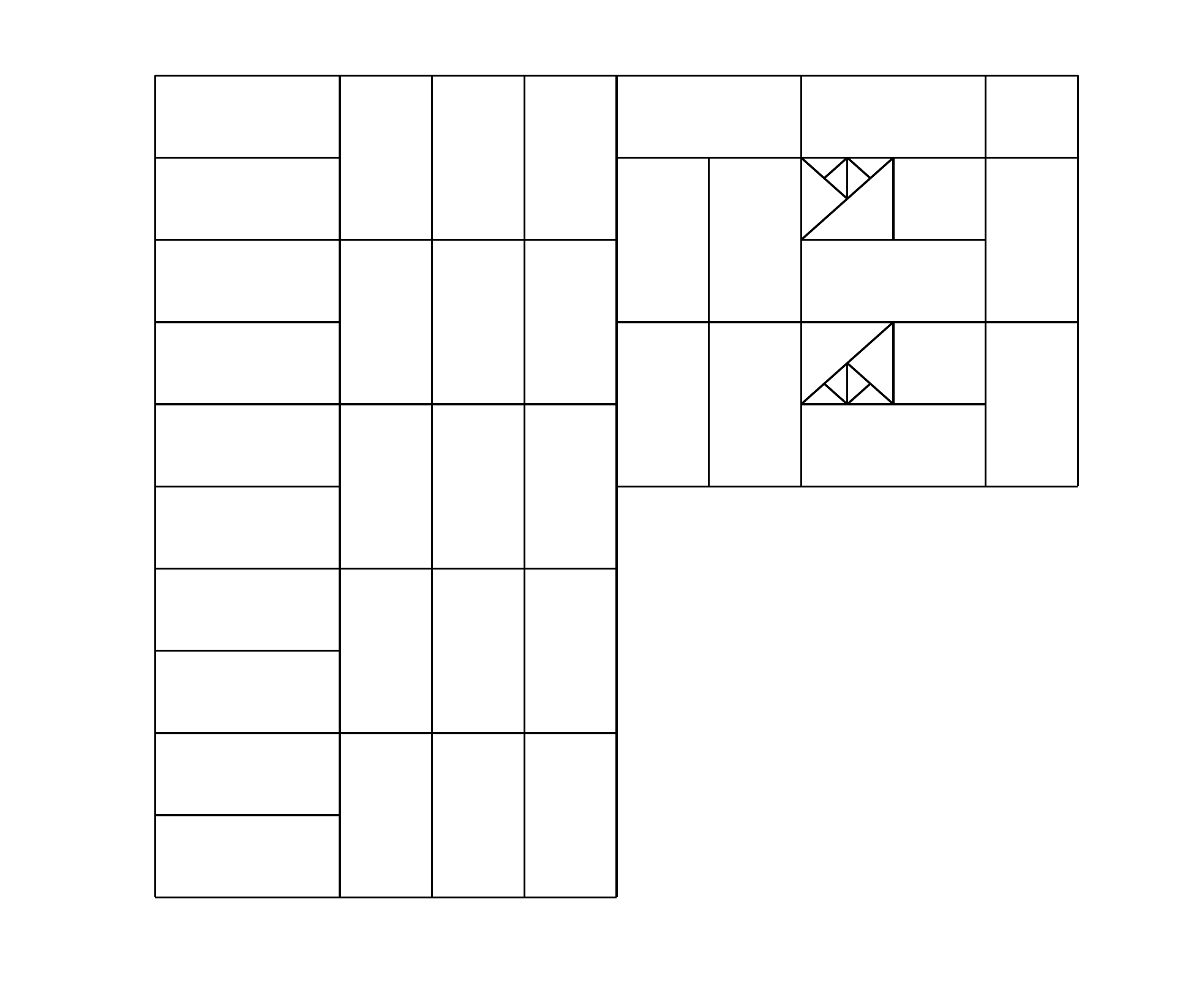} \\
			\includegraphics[width=0.3\linewidth]{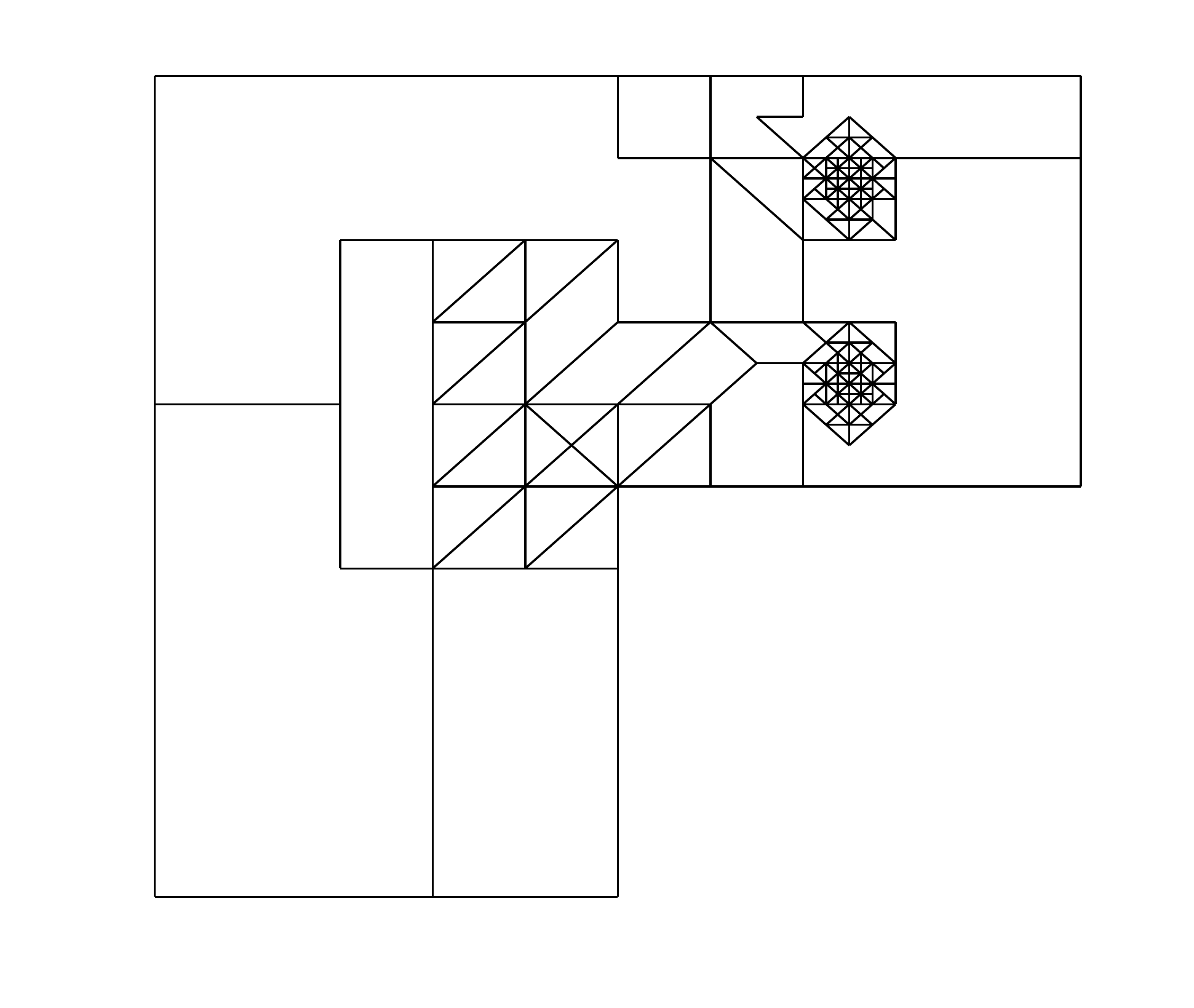} &
			\includegraphics[width=0.3\linewidth]{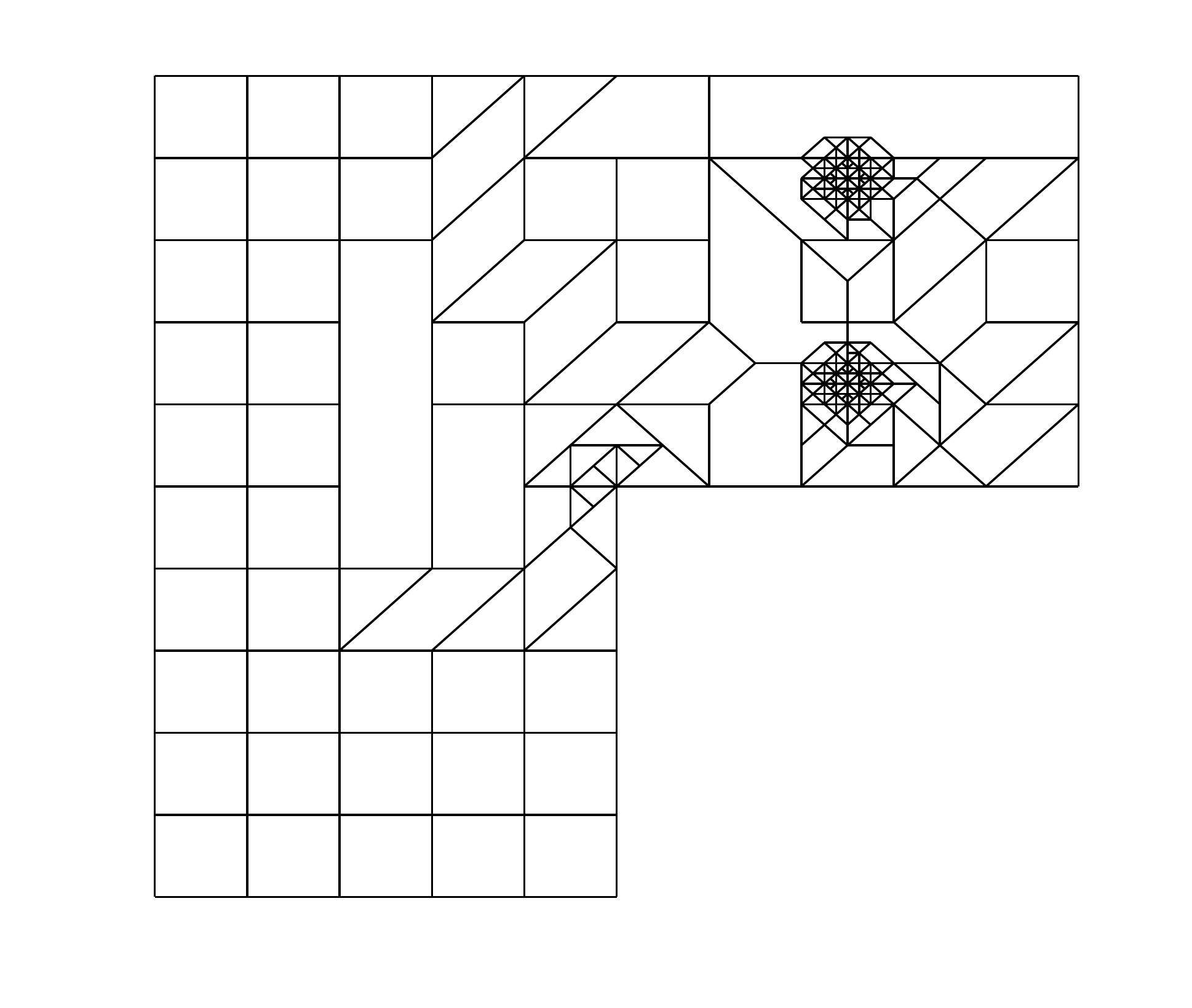} &
			\includegraphics[width=0.3\linewidth]{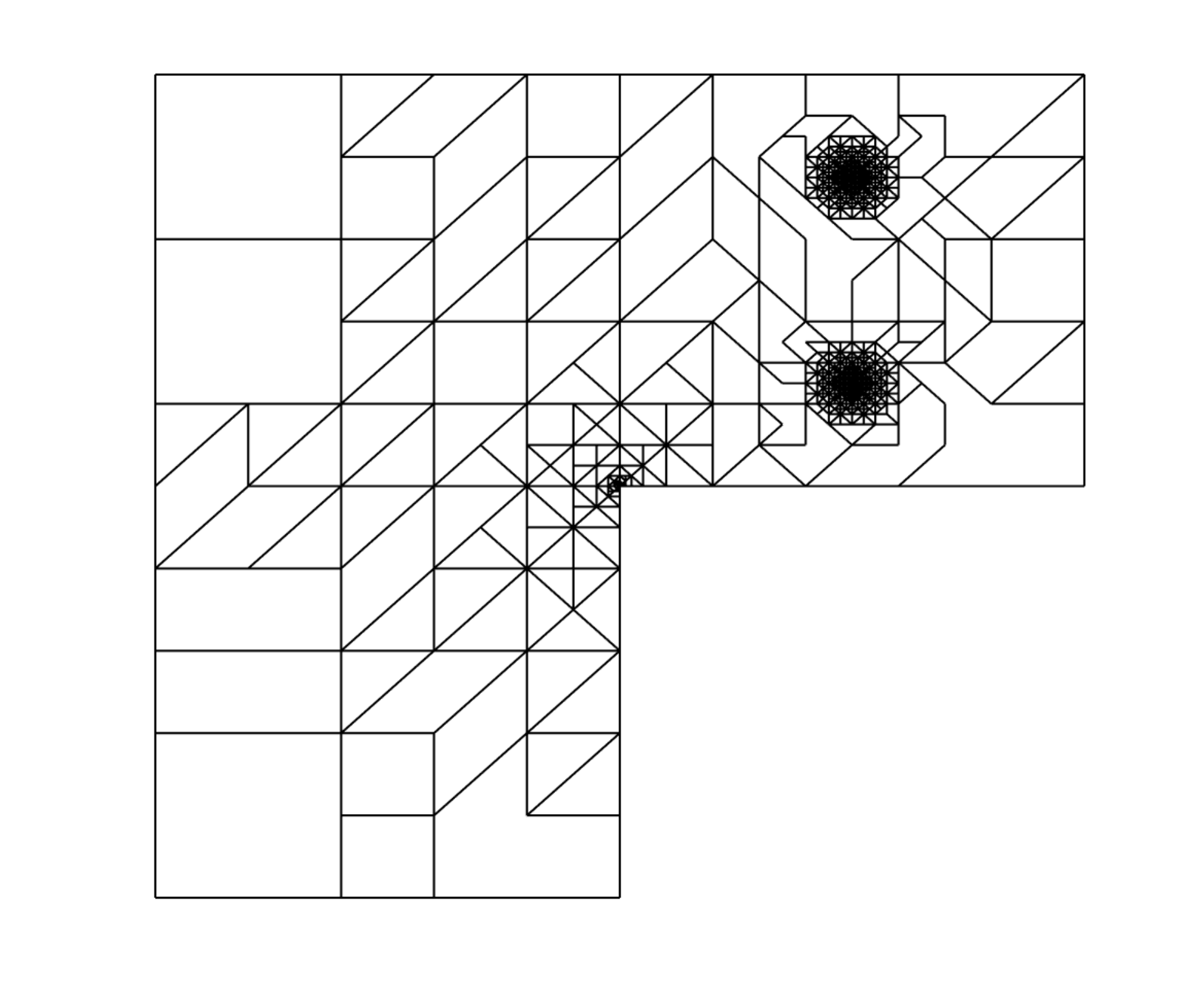}
		\end{tabular}
	\end{center}
	\caption{The sequence of meshes generated by the adaptive polytopic dG algorithm  with $p=3$.}\label{Ex:mesh}
\end{figure}

\begin{figure}[!tb]
	\begin{center}
		\begin{tabular}{cc}
			\includegraphics[width=0.4\linewidth]{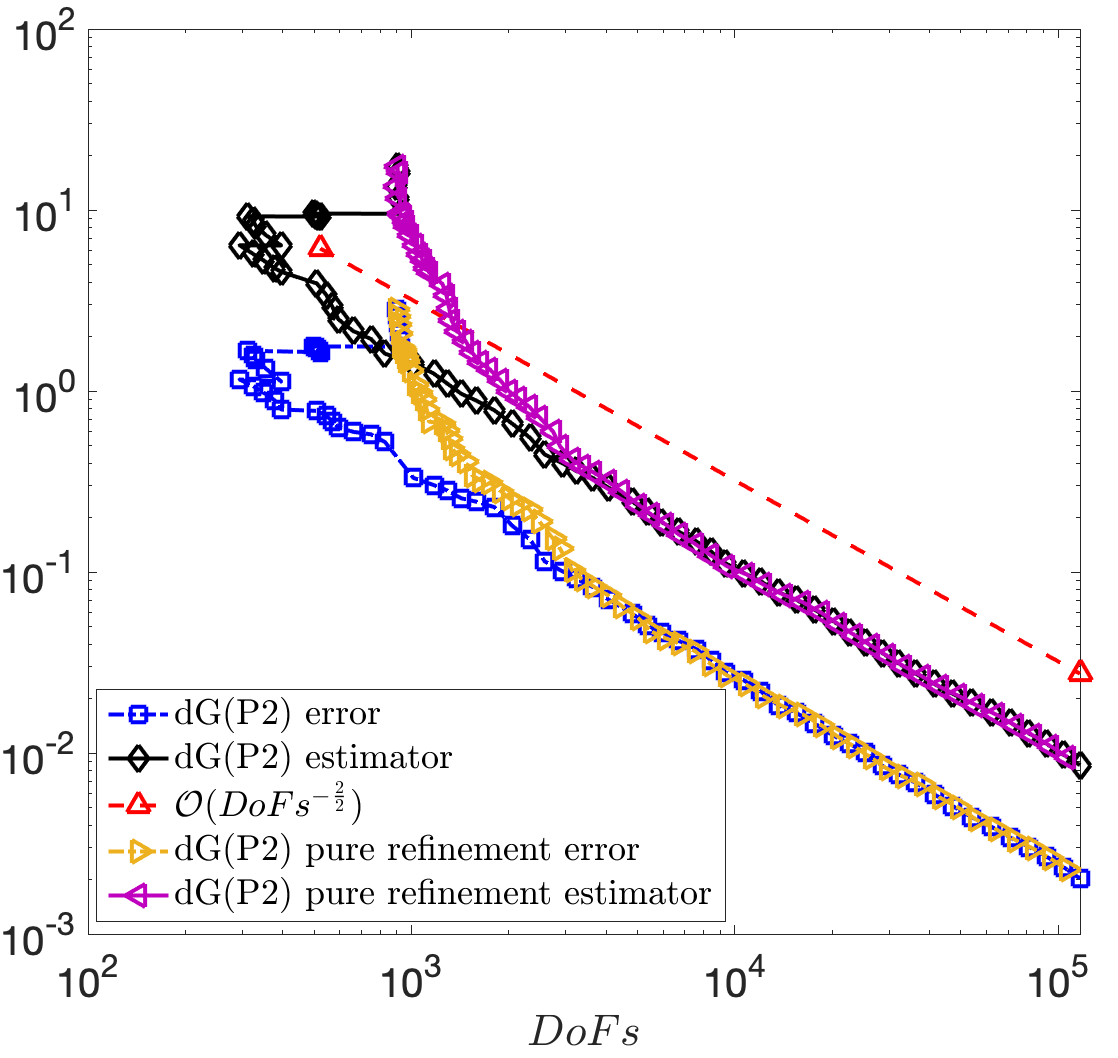} &
			\includegraphics[width=0.4\linewidth]{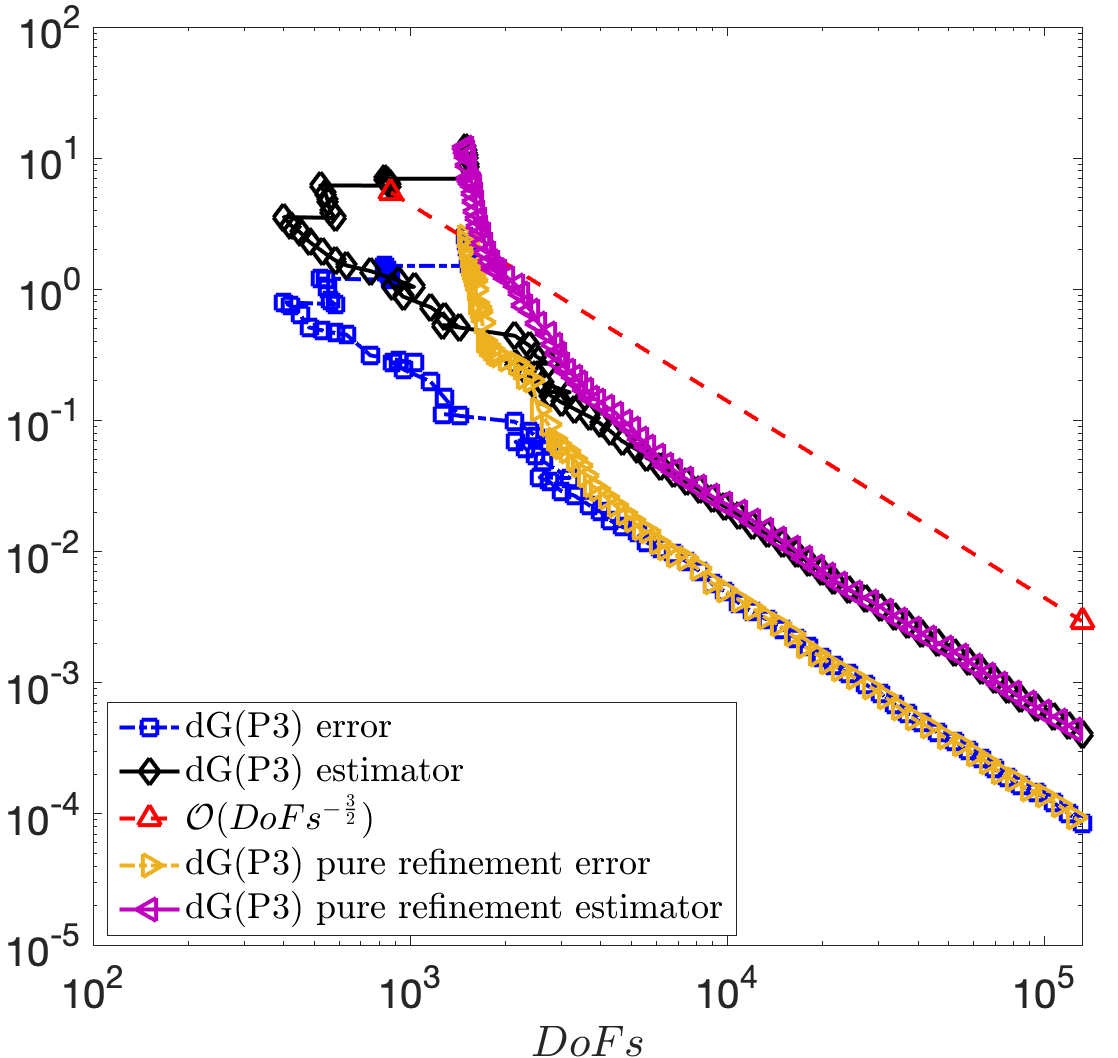}
		\end{tabular}
	\end{center}
	\caption{Example 2. Error and estimator of the new adaptive polytopic dG method and adaptive dG method without agglomeration for $p=2,3$.}\label{Ex:compare_PR_AR}
\end{figure}

\bibliographystyle{siam}
\bibliography{references}
\end{document}